\documentclass[a4paper]{amsart}
\usepackage{amssymb}
\usepackage{mathrsfs}
\usepackage{accents}
\usepackage{verbatim}
\usepackage{epic}
\usepackage{eepic}
\usepackage{url}
\usepackage[dvipdfmx]{graphicx}
\usepackage[color,final]{showkeys}
\definecolor{refkey}{rgb}{1,0,0}
\definecolor{labelkey}{rgb}{0,0,1}
\theoremstyle{plain}
  \newtheorem{thm}{Theorem}[section]
  \newtheorem{lem}[thm]{Lemma}
  \newtheorem{cor}[thm]{Corollary}
  \newtheorem{prop}[thm]{Proposition}
  \newtheorem{conj}[thm]{Conjecture}

  \newtheorem*{obs*}{Observation}
\theoremstyle{definition}
  \newtheorem{defn}[thm]{Definition}

\theoremstyle{remark}
  \newtheorem{rem}[thm]{Remark}

\newcommand{\Z}{\mathbb{Z}}
\newcommand{\C}{\mathbb{C}}
\newcommand{\R}{\mathbb{R}}
\newcommand{\Q}{\mathbb{Q}}
\newcommand{\Vol}{\operatorname{Vol}}

\newcommand{\CS}{\operatorname{CS}}
\newcommand{\Li}{\operatorname{Li}}
\newcommand{\Hom}{\operatorname{Hom}}
\newcommand{\cs}{\operatorname{cs}}
\renewcommand{\L}{\mathcal{L}}
\newcommand{\SL}{\rm{SL}}
\newcommand{\Tr}{\operatorname{Tr}}
\newcommand{\Tor}{\operatorname{Tor}}
\newcommand{\arccosh}{\operatorname{arccosh}}

\newcommand{\pic}[2]{\raisebox{-0.5\height}{\includegraphics[scale=#1]{#2.eps}}}
\renewcommand{\Re}{\operatorname{Re}}
\renewcommand{\Im}{\operatorname{Im}}

\newcommand{\Ker}{\operatorname{Ker}}
\renewcommand{\i}{\sqrt{-1}}
\newcommand{\FE}{\mathscr{E}}
\newcommand{\oH}{\overline{H}}
\newcommand{\uH}{\underline{H}}
\makeatletter
\newcommand{\oset}[3][0ex]{%
  \mathrel{\mathop{#3}\limits^{
    \vbox to#1{\kern-1\ex@
    \hbox{$\scriptstyle#2$}\vss}}}}
\newcommand{\uset}[3][0ex]{%
  \mathrel{\mathop{#3}\limits_{
    \vbox to#1{\kern-5\ex@
    \hbox{$\scriptstyle#2$}\vss}}}}
\makeatother
\newcommand{\Rpath}{\oset{\frown}{\mathbb{R}}}

\numberwithin{equation}{section}

\allowdisplaybreaks[1]
\begin{document}
\title[The colored Jones polynomial of the figure-eight knot]
{The colored Jones polynomial of \\ the figure-eight knot \\ and an $\SL(2;\R)$-representation}
\author{Hitoshi Murakami}
\address{
Graduate School of Information Sciences,
Tohoku University,
Aramaki-aza-Aoba 6-3-09, Aoba-ku,
Sendai 980-8579, Japan}
\email{hitoshi@tohoku.ac.jp}
\date{\today}
\begin{abstract}
We study the asymptotic behavior of the $N$-dimensional colored Jones polynomial of the figure-eight knot, evaluated at $\exp\bigl((u+2p\pi\i)/N\bigr)$ as $N$ tends to infinity, where $u>\arccosh(3/2)$ is a real number and $p\ge1$ is an integer.
It turns out that it corresponds to an $\SL(2;\R)$ representation of the fundamental group of the knot complement.
Moreover, it defines the adjoint Reidemeister torsion and the Chern--Simons invariant associated with the representation.
\end{abstract}
\keywords{colored Jones polynomial, figure-eight knot, volume conjecture, Chern--Simons invariant, Reidemeister torsion, SL(2;R) representation}
\subjclass{Primary 57K10 57K14 57K16}
\thanks{The author is supported by JSPS KAKENHI Grant Numbers JP20K03601, JP22H01117, JP20K03931.}
\maketitle
\section{Introduction}\label{sec:intro}
For a link $L$ in the three-sphere $S^3$ and an integer $N\ge2$, let $J_N(L;q)\in\Z[q^{1/2},q^{-1/2}]$ be the colored Jones polynomial associated with the $N$-dimensional irreducible representation of the Lie algebra $\mathfrak{sl}(2;\C)$.
We normalize it so that $J_N(U;q)=1$ for the unknot $U$, and that $J_2(L;q)$ is (a version of) the original Jones polynomial \cite{Jones:BULAM31985} satisfying the following skein relation:
\begin{equation*}
  qJ_2\left(\raisebox{-3mm}{\includegraphics[scale=0.15]{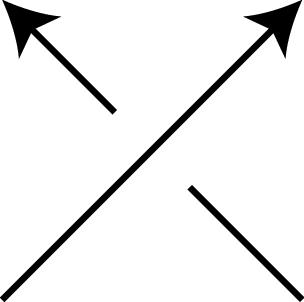}};q\right)
  -
  q^{-1}J_2\left(\raisebox{-3mm}{\includegraphics[scale=0.15]{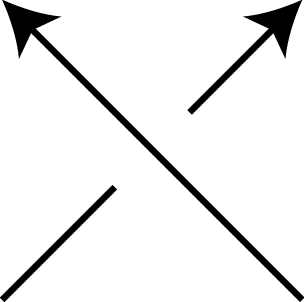}};q\right)
  =
  \left(q^{1/2}-q^{-1/2}\right)
  J_2\left(\raisebox{-3mm}{\includegraphics[scale=0.15]{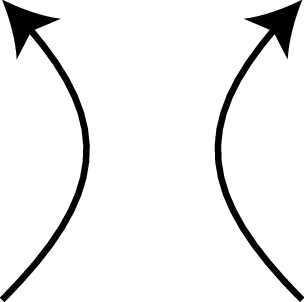}};q\right).
\end{equation*}
\par
In \cite{Kashaev:MODPLA94,Kashaev:MODPLA95}, R.~Kashaev introduced a link invariant $\langle L\rangle_N\in\C$ by using the cyclic quantum dilogarithm.
He conjectured that his invariant $\langle L\rangle_N$ grows exponentially with growth rate proportional to the hyperbolic volume of $S^3\setminus{L}$ provided that $L$ is hyperbolic, that is, the complement $S^3\setminus{L}$ has a complete hyperbolic structure with finite volume \cite{Kashaev:LETMP97}.
\par
In \cite{Murakami/Murakami:ACTAM12001}, J.~Murakami and the author proved that Kashaev's invariant coincides with the $N$-dimensional colored Jones polynomial evaluated at $q=e^{2\pi\i/N}$.
So Kashaev's conjecture turns out to be
\begin{equation}\label{eq:Kashaev}
  \lim_{N\to\infty}\frac{\log\left|J_N(L;e^{2\pi\i/N})\right|}{N}
  =
  \frac{\Vol(L)}{2\pi}
\end{equation}
for a hyperbolic link $L$, where$\Vol(L)$ is the hyperbolic volume of $S^3\setminus{L}$.
We also generalized it to any knot taking the simplicial volume (Gromov's invariant) \cite{Gromov:INSHE82} (see also \cite[Chapter~6]{Thurston:GT3M}) into account.
Here the simplicial volume $\|K\|$ for a knot in $S^3$ is known to be the sum of the hyperbolic volumes of the hyperbolic pieces in the knot complement.
\begin{conj}[{Volume Conjecture, \cite[Conjecture~5.1]{Murakami/Murakami:ACTAM12001}}]\label{conj:VC}
For any knot $K\subset S^3$, we have
\begin{equation*}
  \lim_{N\to\infty}\frac{\log\left|J_N(K;e^{2\pi\i/N})\right|}{N}
  =
  \frac{\|K\|}{2\pi}.
\end{equation*}
\end{conj}
\begin{rem}
Kashaev's conjecture is for {\it any hyperbolic link} and the volume conjecture is for {\it any knot}, not necessarily hyperbolic.
\end{rem}
\par
As far as the author knows, the volume conjecture (for knots) is proved for
\begin{itemize}
\item
Torus knots by Kashaev and O.~Tirkkonen \cite{Kashaev/Tirkkonen:ZAPNS2000}.
Note that the simplicial volume of a torus knot is zero because the complement of a torus knot is Seifert fibered and so there is no hyperbolic pieces.
\item
Hyperbolic knot $4_1$, also known as the figure-eight knot by T.~Ekholm.
For his proof, see for example \cite[3.2]{Murakami/Yokota:2018}.
\item
Hyperbolic knot $5_2$ by T.~Ohtsuki \cite{Ohtsuki:QT2016}.
\item
Hyperbolic knots $6_1$, $6_2$, and $6_3$ by Ohtsuki and Yokota \cite{Ohtsuki/Yokota:MATPC2018}.
\item
Hyperbolic knots $7_2,7_3,\cdots,7_7$ by Ohtsuki \cite{Ohtsuki:INTJM62017}.
\item
$(2,2m+1)$-cable of the figure-eight knot by T.~Le and A.~Tran \cite{Le/Tran:JKNOT2010}.
\item
Whitehead doubles of torus knots by H.~Zheng \cite{Zheng:CHIAM22007}.
\end{itemize}
\par
Kashaev's conjecture and the volume conjecture have been generalized in several ways.
\par
Kashaev's conjecture was complexified by J.~Murakami, M.~Okamoto, T.~Takata, Yokota and the author as follows.
\begin{conj}[{Complexification of Kashaev's conjecture, \cite[Conjecture~1.2]{Murakami/Murakami/Okamoto/Takata/Yokota:EXPMA02}}]\label{conj:CVC}
For a hyperbolic link $L$, we have
\begin{equation}\label{eq:CVC}
  \lim_{N\to\infty}\frac{\log{J_N\left(L;e^{2\pi\i/N}\right)}}{N}
  =
  \frac{\Vol(L)+\i\CS^{\mathrm{SO}(3)}(L)}{2\pi},
\end{equation}
where $\CS^{\mathrm{SO}(3)}(L)$ is the $\mathrm{SO}(3)$ Chern--Simons invariant associated with the Levi-Civita connection \cite{Chern/Simons:ANNMA21974,Meyerhoff:LMSLN112}.
\end{conj}
\begin{rem}
Note that the real part of \eqref{eq:CVC} appears in Kashaev's conjecture \eqref{eq:Kashaev}.
\end{rem}
The following conjecture refines Kashaev's conjecture \cite{Gukov:COMMP2005,Gukov/Murakami:FIC2008,Dimofte/Gukov/Lenells/Zagier:CNTP2010,Zagier:2010}.
Note that it also generalizes Conjecture~\ref{conj:CVC}.
\begin{conj}[Refinement of Kashaev's conjecture]\label{conj:RVC}
For a hyperbolic knot $K$, we have
\begin{equation*}
  J_N\left(K;e^{2\pi\i/N}\right)
  \underset{N\to\infty}{\sim}
  \tau(K)N^{3/2}
  \exp\left(\frac{N}{2\pi}\bigl(\Vol(K)+\i\CS^{\mathrm{SO}(3)}(K)\bigr)\right)
\end{equation*}
as $N\to\infty$, where $2\i\tau(K)^{-2}$ equals the homological adjoint Reidemeister torsion twisted by the holonomy representation associated with the complete hyperbolic structure.
\end{conj}
\par
The refinement of Kashaev's conjecture is proved for hyperbolic knots with crossings less than or equal to seven ($4_1$ by J.~Andersen and S.~Hansen \cite{Andersen/Hansen:JKNOT2006}, $5_2$ by Ohtsuki \cite{Ohtsuki:QT2016}, $6_1$, $6_2$, and $6_3$ by Ohtsuki and Yokota \cite{Ohtsuki/Yokota:MATPC2018}, and $7_2,7_3,\cdots,7_7$ by Ohtsuki \cite{Ohtsuki:INTJM62017}).
\par
Notice that we can replace $J_N(K;e^{2\pi\i/N})$ with Kashaev's invariant $\langle L\rangle_N$ in Conjectures~\ref{conj:VC}, \ref{conj:CVC}, and \ref{conj:RVC}.
\par
Since the colored Jones polynomial has $q$ as a complex parameter, we can perturb $2\pi\i$ in $J_N(K;e^{2\pi\i/N})$ by adding a complex number $u$.
So we are now interesting in the asymptotic behavior of $J_N(e^{(2\pi\i+u)/N})$, or more generally  $J_N(e^{(2p\pi\i+u)/N})$ ($p\in\Z$) as $N\to\infty$.
\par
Letting $(\mu,\lambda)$ be the standard meridian/longitude pair of $K\subset S^3$, consider an irreducible representation $\rho_u\colon\pi_1(S^3\setminus{K})\to\SL(2;\C)$ parametrized by $u$ such that
\begin{equation*}
  \rho_u(\mu)
  :=
  \begin{pmatrix}
    e^{u/2}&\ast\\0&e^{-u/2}
  \end{pmatrix},
  \quad
  \rho_u(\lambda)
  :=
  \begin{pmatrix}
    -e^{-v(u)/2}&\ast\\0&-e^{-v(u)/2}
  \end{pmatrix}.
\end{equation*}
Here we choose the longitude so that the linking number between $\lambda$ and $K$ is zero.
Then one can define the homological adjoint Reidemeister torsion of $\rho_u$ associated with $\mu$, which we denote by $T_{K}(\rho_u)$, and the $\mathrm{PSL}(2;\C)$ Chern--Simons invariant of $\rho_u$ associated with $\bigl(u,v(u)\bigr)$, which we denote by $\CS_{K}(\rho_u;u,v(u))$.
If $K$ is hyperbolic and $\rho_0$ is the holonomy representation, then we have $\CS_K(\rho_0;0,0)=\i\Vol(K)-\CS^{\rm{SO}(3)}(K)$.
See for example \cite[Chapter~5]{Murakami/Yokota:2018}.
\par
The following conjecture is proposed in \cite{Murakami:JTOP2013} (see also \cite{Gukov/Murakami:FIC2008,Dimofte/Gukov:Columbia}).
\begin{conj}[Parametrization of the volume conjecture]\label{conj:PVC}
For a hyperbolic knot $K$, there exists a neighborhood $U$ of $0$ in $\C$ such that for $u\in U\setminus{\pi\i\Q}$ such that
\begin{multline}
  J_N\left(K;e^{(2\pi\i+u)/N}\right)
  \\
  \underset{N\to\infty}{\sim}
  \frac{\sqrt{-\pi}}{2\sinh(u/2)}
  T_{K}(\rho_u)^{-1/2}
  \left(\frac{N}{2\pi\i+u}\right)^{1/2}
  \exp\left(\frac{S_K(u)}{2\pi\i+u}N\right),
\end{multline}
where $T_K(\rho_u)$ is the adjoint Reidemeister torsion, and $\CS_{u,v(u)}(\rho_u):=S_K(u)-u\pi\i-uv(u)/4$ is the Chern--Simons invariant defined as above.
Moreover we have $v(u)=2\frac{d}{d\,u}S_K(u)-2\pi\i$.
\end{conj}
\par
A weaker version of Conjecture~\ref{conj:PVC} for the figure-eight knot was proved by Yokota and the author \cite{Murakami/Yokota:JREIA2007}.
\begin{thm}\label{thm:Murakami/Yokota}
There exists a neighborhood $U\subset\C$ of $0$ such that if $u\in U\setminus\pi\i\Q$, then for the figure-eight knot $\FE$ we have
\begin{equation*}
  \lim_{N\to\infty}\frac{\log J_N(\FE;e^{(2\pi\i+u)/N})}{N}
  =
  S_{\FE}(u),
\end{equation*}
where we put
\begin{align}
  S_{\FE}(u)
  &:=
  \Li_2\left(e^{-u-\varphi(u)}\right)-\Li_2\left(e^{-u+\varphi(u)}\right)
  +
  u\bigl(\varphi(u)+2\pi\i\bigr),
  \label{eq:S_def}
  \\
  \varphi(u)
  &:=
  \arccosh(\cosh{u}-1/2)
  \label{eq:phi_def}
\end{align}
with $\Li_2(z)$ the dilogarithm function: $\Li_2(z):=-\int_{0}^{z}\frac{\log(1-w)}{w}\,dw$.
\end{thm}
\par
In the following, we study the asymptotic behavior of $J_N(\FE;e^{(u+2p\pi\i)/N})$ as $N\to\infty$ for the figure-eight knot $\FE$ with $u\in\R$ and $p\in\Z$.
\par
The author proved Conjecture~\ref{conj:PVC} for the figure-eight knot $\FE$ when $|u|$ is small and $p\ne0$.
\begin{thm}\label{thm:small_u}
Let $u$ a real number with $0<u<\kappa:=\arccosh(3/2)$.
Putting $\xi:=2p\pi\i+u$ for an integer $p\ge1$, we have
\begin{multline*}
  J_N\left(\FE;e^{\xi/N}\right)
  \\
  \underset{N\to\infty}{\sim}
  J_p\left(\FE;e^{4N\pi^2/\xi}\right)
  \frac{\sqrt{-\pi}}{2\sinh(u/2)}
  T_{\FE}(u)^{-1/2}
  \left(\frac{N}{\xi}\right)^{1/2}
  \exp\left(\frac{S_{\FE}(u)}{\xi}N\right),
\end{multline*}
where we put $S_{\FE}(u)$ as in \eqref{eq:S_def}, and
\begin{equation}\label{eq:T_def}
  T_{\FE}(\rho_u)
  :=
  \frac{\sqrt{(2\cosh{u}+1)(2\cosh{u}-3)}}{2}.
\end{equation}
Moreover, $T_{\FE}(\rho_u)$ is the homological adjoint Reidemester torsion, and the Chern--Simons invariant is given by $S_{\FE}(u)$ as follows.
\begin{equation*}
  \CS_{u,v(u)}(\rho_u)
  =
  S_{\FE}(u)-u\pi\i-uv(u)/4
\end{equation*}
with
\begin{equation*}
  v(u)
  =
  2\frac{d}{d\,u}S_{\FE}(u)-2\pi\i.
\end{equation*}
\end{thm}
Note that the formula of the Reidemeister torsion is given in \cite{Porti:MAMCAU1997}.
\begin{rem}
In Theorem~\ref{thm:small_u}, we assume that $u$ and $p$ are positive.
Since $\FE$ is amphicheiral, that is, $\FE$ is equivalent to its mirror image, we have $J_N(\FE;q)=J_N(\FE;q^{-1})$.
Therefore if $u$ and $p$ are real, we have $J_N(\FE;e^{-u-2p\pi\i})=J_N(\FE;e^{u+2\pi\i})$ and $J_N(\FE;e^{-u+2p\pi\i})=J_N(\FE;e^{u-2p\pi\i})=\overline{J_N(\FE;e^{u+2p\pi\i})}$, where $\overline{z}$ is the complex conjugate of $z$.
\end{rem}
Therefore, in the following we also assume that both $u$ and $p$ are positive.
\par
When $p=0$, it is proved that $J_N(\FE;e^{u/N})$ converges \cite{Murakami:JPJGT2007} if $|u|$ is small.
In fact, if $0<u<\kappa$ we have
\begin{equation*}
  \lim_{N\to\infty}J_N(\FE;e^{u/N})
  =
  \frac{1}{\Delta(\FE;e^{u})},
\end{equation*}
where $\Delta(\FE;t)=-t+3-t^{-1}$ is the Alexander polynomial.
See \cite{Murakami:JPJGT2007} for a wider region of $u$ where the colored Jones polynomial converges.
It is known that for any knot $K$ there exists a neighborhood $U_K$ of $0$ such that $J_N(K;e^{u/N})$ converges to $1/\Delta(K;e^{u})$ if $u\in U_K$ \cite{Garoufalidis/Le:GEOTO2011}.
\par
The next theorem describes the case where $u=\kappa$ \cite{Murakami:arXiv2023}.
\begin{thm}\label{thm:kappa}
Putting $\xi:=\kappa+2p\pi\i$ for $p\ge1$, we have
\begin{multline*}
  J_N\left(\FE;e^{\xi/N}\right)
  \\
  \underset{N\to\infty}{\sim}
  J_p\left(\FE;e^{4N\pi^2/\xi}\right)
  \frac{\Gamma(1/3)e^{2\pi\i/6}}{3^{1/6}}
  \left(\frac{N}{\xi}\right)^{2/3}
  \exp\left(\frac{S_{\FE}(\kappa)}{\xi}N\right),
\end{multline*}
where $\Gamma(x)$ is the gamma function.
Note that $S_{\FE}(\kappa)=2\kappa\pi\i$ since $\varphi(\kappa)=0$ from \eqref{eq:phi_def}.
\end{thm}
When $p=0$, $J_N(\FE;e^{\kappa/N})$ grows polynomially as follows \cite[Theorem~1.1]{Hikami/Murakami:COMCM2008}:
\begin{equation*}
  J_N\left(\FE;e^{\kappa/N}\right)
  \\
  \underset{N\to\infty}{\sim}
  \frac{\Gamma(1/3)}{3^{2/3}}
  \left(\frac{N}{\kappa}\right)^{2/3}.
\end{equation*}
\par
The main result in the paper is to prove a result similar to Theorems~\ref{thm:small_u} and \ref{thm:kappa} for the case $u>\kappa$.
\begin{thm}\label{thm:main}
For a positive integer $p$ and a real number $u>\kappa:=\arccosh(3/2)$, putting $\xi:=u+2p\pi\i$, we have
\begin{multline}\label{eq:main}
  J_N\left(\FE;e^{\xi/N}\right)
  \\
  \underset{N\to\infty}{\sim}
  J_p\left(\FE;e^{4N\pi^2/\xi}\right)
  \frac{\sqrt{\pi}}{2\sinh(u/2)}
  T_{\FE}(u)^{-1/2}
  \left(\frac{N}{\xi}\right)^{1/2}
  \exp\left(\frac{S_{\FE}(u)}{\xi}N\right)
\end{multline}
as $N\to\infty$, where $T_{\FE}(u)$ is the homological adjoint Reidemeister torsion, and $\CS_{u,v(u)}\bigl(\FE;\rho_u\big):=S_{\FE}(u)-\pi\i u-\frac{1}{4}uv(u)$ is the Chern--Simons invariant.
Moreover, we have
\begin{equation}\label{eq:v_main}
  v(u)
  =
  2\frac{d}{d\,u}S_{\FE}(u)-2\pi\i.
\end{equation}
\end{thm}
Note that the case where $p=0$ is already proved by A.~Tran and the author:
\begin{equation*}
  J_N(\FE;e^{u/N})
  \underset{N\to\infty}{\sim}
  \frac{\sqrt{\pi}}{2\sinh(u/2)}
  T_{\FE}(u)^{-1/2}\left(\frac{N}{u}\right)^{1/2}
  \exp\left(\frac{S_{\FE}(u)}{u}N\right).
\end{equation*}
\begin{rem}
When $u>\kappa$, $\varphi(u)$ is a positive real number since $\cosh{u}>\cosh{\kappa}=3/2$.
When $0<u<\kappa$ we define $\arccosh{x}$ for $1/2<x<1$ to be $\log(x-\i\sqrt{1-x^2})$ so that $\varphi(u)$ is purely imaginary with $-\pi/3<\Im\varphi(u)<0$.
See Remark~\cite[3.2]{Murakami:JTOP2013}.
\end{rem}
\par
The paper is organized as follows.
\par
In Section\ref{sec:dilog}, we introduce a quantum dilogarithm following \cite{Murakami:arXiv2023}.
Then we obtain a summation formula of the colored Jones polynomial in terms of the quantum dilogarithm in Section~\ref{sec:sum}.
We study the defining domain of the quantum dilogarithm in detail in Section~\ref{sec:Xi}.
In Section~\ref{sec:Poisson} we replace the sum into an integral by using the Poisson summation formula, and then we use the saddle point method to obtain the asymptotic formula \eqref{eq:main} in Theorem~\ref{thm:main}.
A topological interpretation of $T_{\FE}(\rho_u)$ and $S_{\FE}(u)$ appearing in \eqref{eq:main} is described in Section~\ref{sec:CS}.
Appendix~\ref{sec:P1} is devoted to a proof of a lemma used in Section~\ref{sec:Xi}.

\section{Quantum dilogarithm}\label{sec:dilog}
In this section, we summarize some results in \cite{Murakami/Tran:Takata,Murakami:arXiv2023}.
Details are omitted.
\par
Fix a complex number $\gamma$ in the fourth quadrant.
We will introduce a quantum dilogarithm following \cite{Faddeev:LETMP1995}.
See also \cite[(3.3)]{Kashaev:LETMP97}, \cite{Ohtsuki:QT2016}.
Put
\begin{equation}\label{eq:def_TN}
  T_{N}(z)
  :=
  \frac{1}{4}\int_{\Rpath}\frac{e^{(2z-1)t}}{t\sinh(t)\sinh(\gamma t/N)}\,dt
\end{equation}
for an integer $N>|\gamma|/\pi$, where $\Rpath:=(-\infty,-1]\cup\{w\in\C\mid|w|=1,\Im{w}\ge0\}\cup[1,\infty)$ with orientation from $-\infty$ to $\infty$.
Note that $\Rpath$ avoids the poles of the integrand.
We know that the integral \eqref{eq:def_TN} converges if $-\frac{\Re\gamma}{2N}<\Re{z}<1+\frac{\Re\gamma}{2N}$ \cite[Lemma~2.1]{Murakami:arXiv2023}.
\par
We define related integrals $\int_{\Rpath}\frac{e^{(2z-1)t}}{t^{k}\sinh(t)}\,dt$ for $k=0,1,2$.
The following lemma is proved in \cite[Lemma~2.1]{Murakami/Tran:Takata}.
\par
\begin{lem}\label{lem:integrals}
If $0<\Re{z}<1$, then the integrals $\int_{\Rpath}\frac{e^{(2z-1)t}}{t^{k}\sinh(t)}\,dt$ {\rm(}$k=0,1,2${\rm)} converge.
\end{lem}
We introduce the following three functions for $z$ with $0<\Re{z}<1$ using the three integrals above.
\begin{align*}
  \L_0(z)
  &:=
  \int_{\Rpath}\frac{e^{(2z-1)t}}{\sinh(\pi t)}\,dt,
  \\
  \L_1(z)
  &:=
  -\frac{1}{2}
  \int_{\Rpath}\frac{e^{(2z-1)t}}{t\sinh(\pi t)}dt,
  \\
  \L_2(z)
  &:=
  \frac{\pi\i}{2}
  \int_{\Rpath}\frac{e^{(2z-1)t}}{t^2\sinh(\pi t)}dt.
\end{align*}
These functions can be written in terms of well-known functions \cite[Lemma~2.5]{Murakami/Tran:Takata}.
\begin{lem}
The functions $\L_0(z)$, $\L_1(z)$, and $\L_2(z)$ satisfy the following equalities:
\begin{equation}\label{eq:L0_L1_L2}
\begin{split}
  \L_0(z)
  &=
  \frac{-2\pi\i}{1+e^{-2\pi\i z}},
  \\
  \L_1(z)
  &=
  \begin{cases}
    \log\left(1-e^{2\pi\i z}\right)&\text{if $\Im{z}\ge0$,}
    \\[3mm]
    2\pi\i z-\pi\i+\log\left(1-e^{-2\pi\i z}\right)&\text{if $\Im{z}<0$,}
  \end{cases}
  \\[5mm]
  \L_2(z)
  &=
  \begin{cases}
    \Li_2\left(e^{2\pi\i z}\right)&\text{if $\Im{z}\ge0$,}
    \\[3mm]
    2\pi^2z^2-2\pi^2z+\frac{\pi^2}{3}-\Li_2\left(-e^{-2\pi\i z}\right)
    &\text{if $\Im{z}<0$}.
  \end{cases}
\end{split}
\end{equation}
Here the logarithm function $\log(1-w)$ is defined by the infinite series $-\sum_{n=1}^{\infty}\dfrac{w^n}{n}$ for $|w|\le1$ {\rm(}$w\ne1${\rm)}, and dilogarithm function $\Li_2(w)$ is defined by the infinite series $\sum_{n=1}^{\infty}\dfrac{w^n}{n^2}$ for $|w|\le1$.
\end{lem}
If $\Im{z}<0$, we also have the following lemma \cite[Lemma~2.5]{Murakami:arXiv2023}.
\begin{lem}\label{lem:floor}
When $\Im{z}<0$ we have the following equalities.
\begin{align*}
  \L_1(z)
  &=
  \log\left(1-e^{2\pi\i z}\right)+2\lfloor\Re{z}\rfloor\pi\i,
  \\
  \L_2(z)
  &=
  \Li_2\left(e^{2\pi\i z}\right)
  -
  2\pi^2\lfloor\Re{z}\rfloor\bigl(\lfloor\Re{z}\rfloor-2z+1\bigr),
\end{align*}
where $\lfloor{x}\rfloor$ is the greatest integer that does not exceed $x$.
\end{lem}
We can extend both $\L_1$ and $\L_2$ to holomorphic functions in $\C\setminus\bigl((-\infty,0]\cup[1,\infty)\bigr)$ by using the right hand sides of \eqref{eq:L0_L1_L2}.
\par
As for derivatives of $\L_1(z)$ and $\L_2(z)$, we can prove the following formulas \cite[Lemma~2.9]{Murakami/Tran:Takata}:
\begin{lem}\label{lem:der_L2_L1}
We have
\begin{align*}
  \frac{d\,\L_2}{d\,z}(z)
  &=
  -2\pi\i\L_1(z),
  \\
  \dfrac{d\,\L_1}{d\,z}(z)
  &=
  -\L_0(z).
\end{align*}
\end{lem}
\par
As shown in \cite[\S~2]{Murakami:arXiv2023}, we can extend the quantum dilogarithm $T_{N}$ to the following region:
\begin{equation}\label{eq:Delta01}
  \left\{z\in\C\Bigm|-1+\frac{\Re\gamma}{2N}<\Re{z}<2-\frac{\Re\gamma}{2N}\right\}
  \setminus
  \bigl(\Delta_{0}^{+}\cup\Delta_{1}^{+}\bigr),
\end{equation}
where we put
\begin{align*}
  \Delta_{0}^{+}
  &:=
  \left\{
    z\in\C\Bigm|
    \text{$-1+\frac{\Re\gamma}{2N}<\Re{z}\le0$,
    $\Im{z}\ge0$, and $\Im\left(\frac{z}{\gamma}\right)\le0$}
  \right\},
  \\
  \Delta_{1}^{-}
  &:=
  \left\{
    z\in\C\Bigm|
    \text{$1\le\Re{z}<2-\frac{\Re\gamma}{2N}$,
    $\Im{z}\le0$, and $\Im\left(\frac{z-1}{\gamma}\right)\ge0$}
  \right\}.
\end{align*}
The function $T_{N}(z)$ defined as above is holomorphic in the region \eqref{eq:Delta01}, and satisfies the following equalities:
\begin{align*}
  T_N(z)-T_N(z+1)
  &=
  \L_1\left(\frac{N}{\gamma}z+\frac{1}{2}\right),
  \\
  T_N\left(z-\frac{\gamma}{2N}\right)-T_N\left(z+\frac{\gamma}{2N}\right)
  &=
  \L_1(z).
\end{align*}
\par
It is also proved in \cite[Proposition~2.25]{Murakami:arXiv2023} that the series of functions $\left\{\frac{1}{N}T_N(z)\right\}$ converges to $\frac{1}{2\pi\i\gamma}\L_2(z)$ as $N\to\infty$.
In fact we have the following proposition.
\begin{prop}\label{prop:TN_L2}
We have
\begin{equation*}
  T_{N}(z)
  =
  \frac{N}{2\pi\i\gamma}\L_2(z)
  +
  O(N^{-1})
\end{equation*}
as $N\to\infty$ in the following region:
\begin{equation*}
  \left\{z\in\C\Bigm|-1+\nu\le\Re{z}\le2-\nu, |\Im{z}|\le M\right\}
  \setminus
  \left(\Delta^{+}_{0,\nu}\cup\Delta^{-}_{1,\nu}\right),
\end{equation*}
where $\nu$ is a sufficiently small positive number and $M$ is a positive number, and we put
\begin{align*}
  \Delta^{+}_{0,\nu}
  &:=
  \left\{
    z\in\C\Bigm|
    \text{$-1+\nu\le\Re{z}<\nu$,
    $\Im{z}>-\nu$,  and $\Im\left(\frac{z-\nu}{\gamma}\right)<0$}
  \right\},
  \\
  \Delta^{-}_{1,\nu}
  &:=
  \left\{
    z\in\C\Bigm|
    \text{$1-\nu<\Re{z}\le2-\nu$,
    $\Im{z}<\nu$, and $\Im\left(\frac{z-1+\nu}{\gamma}\right)>0$}
  \right\}.
\end{align*}
See Figure~\ref{fig:domain_TN}.
\begin{figure}[h]
\pic{0.3}{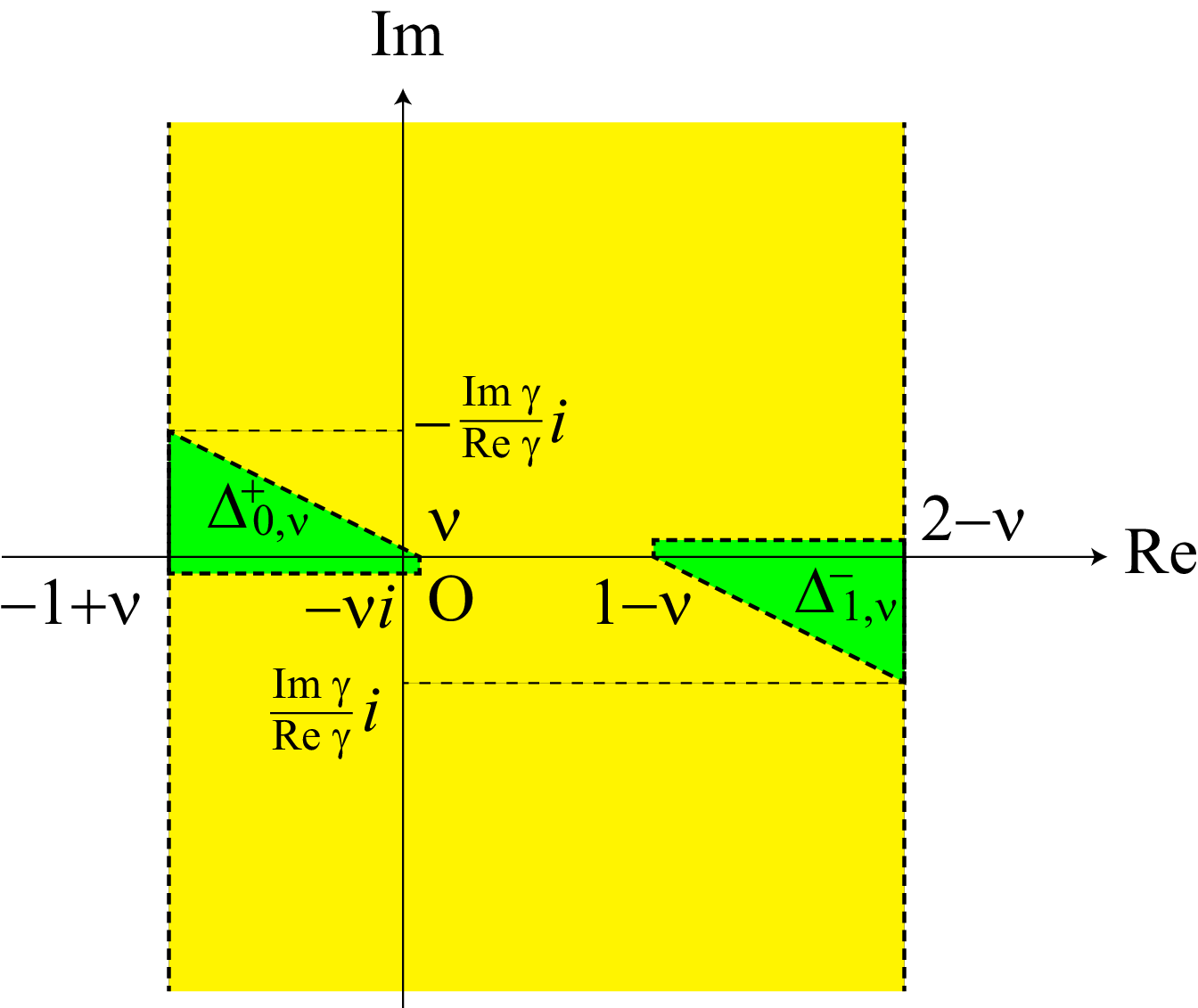}
\caption{The series of functions $\left\{\frac{1}{N}T_N(z)\right\}$ converges to $\frac{1}{2\pi\i\gamma}\L_2(z)$ in the yellow region.
The green trapezoidal regions are $\Delta^{+}_{0,\nu}$ and $\Delta^{-}_{1,\nu}$.}
\label{fig:domain_TN}
\end{figure}
\end{prop}

\section{Summation formula}\label{sec:sum}
Let $\FE$ be the figure-eight knot.
Its colored Jones polynomial has the following simple formula due to K.~Habiro \cite[P.~36 (1)]{Habiro:SURIK2000} and T.~Le \cite[1.22 Example, P.~129]{Le:TOPOA2003} (see also \cite[Theorem~5.1]{Masbaum:ALGGT12003}):
\begin{equation}\label{eq:Jones}
\begin{split}
  J_N(\FE;q)
  =&
  \sum_{k=0}^{N-1}
  \prod_{l=1}^{k}
  (q^{(N-l)/2}-q^{-(N-l)/2})(q^{(N+l)/2}-q^{-(N+l)/2})
  \\
  =&
  \sum_{k=0}^{N-1}q^{-kN}
  \prod_{l=1}^{k}
  (1-q^{N-l})(1-q^{N+l}).
\end{split}
\end{equation}
Putting $\xi:=u+2p\pi\i$, we have
\begin{equation}\label{eq:fig8}
  J_N(\FE;e^{\xi/N})
  =
  \sum_{k=0}^{N-1}e^{-k\xi}
  \prod_{l=1}^{k}\left(1-e^{(N-l)\xi/N}\right)\left(1-e^{(N+l)\xi/N}\right).
\end{equation}
\par
We will replace the product in \eqref{eq:fig8} with a single function following \cite[\S~3]{Murakami:arXiv2023}.
Put $\gamma:=\frac{\xi}{2\pi\i}$.
Since $\gamma$ is in the fourth quadrant, we can use the quantum dilogarithm $T_N(z)$ introduced in Section~\ref{sec:dilog} to define a function $f_N(z)$ as follows:
\begin{equation*}
  f_N(z)
  :=
  \frac{1}{N}T_N\bigl(\gamma(1-z)-p+1\bigr)
  -
  \frac{1}{N}T_N\bigl(\gamma(1+z)-p\bigr)
  -
  uz-\frac{2p\pi\i}{\gamma}
  +
  2\pi\i,
\end{equation*}
where we add $2\pi\i$ on purpose.
The function $f_N(z)$ is defined in the region
\begin{equation*}
  \left\{
    z\in\C
    \Bigm|
    -1+\frac{p}{2N}<\Re(\gamma z)<2-\frac{p}{2N}
  \right\}
  \setminus
  \bigl(
    \underline{\nabla}_{0}^{+}
    \cup
    \underline{\nabla}_{0}^{-}
    \cup
    \overline{\nabla}_{0}^{+}
    \cup
    \overline{\nabla}_{0}^{-}
  \bigr)
\end{equation*}
with
\begin{align*}
  \underline{\nabla}_{0}^{+}
  :=&
  \left\{
    z\in\C\Bigm|
    -1+\frac{p}{2N}<\Re(\gamma z)\le0,
    \Im(\gamma z)\ge-\frac{u}{2\pi},
    \Im{z}\le-\frac{pu}{2\pi|\gamma|^2}
  \right\},
  \\
  \underline{\nabla}_{0}^{-}
  :=&
  \left\{
    z\in\C\Bigm|
    1\le\Re(\gamma z)<2-\frac{p}{2N},
    \Im(\gamma z)\le-\frac{u}{2\pi},
    \Im{z}\ge\frac{(1-p)u}{2\pi|\gamma|^2}
  \right\},
  \\
  \overline{\nabla}_{0}^{+}
  :=&
  \left\{
    z\in\C\Bigm|
    -1+\frac{p}{2N}<\Re(\gamma z)\le0,
    \Im(\gamma z)\ge\frac{u}{2\pi},
    \Im{z}\le\frac{pu}{2\pi|\gamma|^2}
  \right\},
  \\
  \overline{\nabla}_{0}^{-}
  :=&
  \left\{
    z\in\C\Bigm|
    1\le\Re(\gamma z)<2-\frac{p}{2N},
    \Im(\gamma z)\le\frac{u}{2\pi},
    \Im{z}\ge\frac{(p+1)u}{2\pi|\gamma|^2}
  \right\}.
\end{align*}
See Figure~\ref{fig:domain_fN}, where we put
\begin{align}
  K_{s}
  &:=
  \left\{z\in\C\mid\Im(\gamma z)=\frac{s}{2\pi}\right\}
  =
  \{z\in\C\mid\Re(\xi z)=-s\},
  \label{eq:K}
  \\
  L_{t}
  &:=
  \{z\in\C\mid\Re(\gamma z)=t\}
  =
  \{z\in\C\mid\Im(\xi z)=2\pi t\}.
  \label{eq:L}
\end{align}
\begin{rem}
The yellow strip in Figure~\ref{fig:domain_fN} is between $L_{-1+p/(2N)}$ and $L_{2-p/(2N)}$.
The green triangles can be described as follows.
\begin{itemize}
\item
$\overline{\nabla}^{+}_{0}$ is  surrounded by the lines $L_{-1+p/(2N)}$, $K_{u}$, and $\Im{z}=\frac{pu}{2\pi|\gamma|^2}$.
\item
$\overline{\nabla}^{-}_{0}$ is  surrounded by the lines $L_{2-p/(2N)}$, $K_{u}$, and $\Im{z}=\frac{(1+p)u}{2\pi|\gamma|^2}$.
\item
$\underline{\nabla}^{+}_{0}$ is  surrounded by the lines $L_{-1+p/(2N)}$, $K_{-u}$, and $\Im{z}=\frac{-pu}{2\pi|\gamma|^2}$.
\item
$\underline{\nabla}^{-}_{0}$ is  surrounded by the lines $L_{2-p/(2N)}$, $K_{-u}$, and $\Im{z}=\frac{(1-p)u}{2\pi|\gamma|^2}$.
\end{itemize}
\end{rem}
\begin{figure}[h]
\pic{0.3}{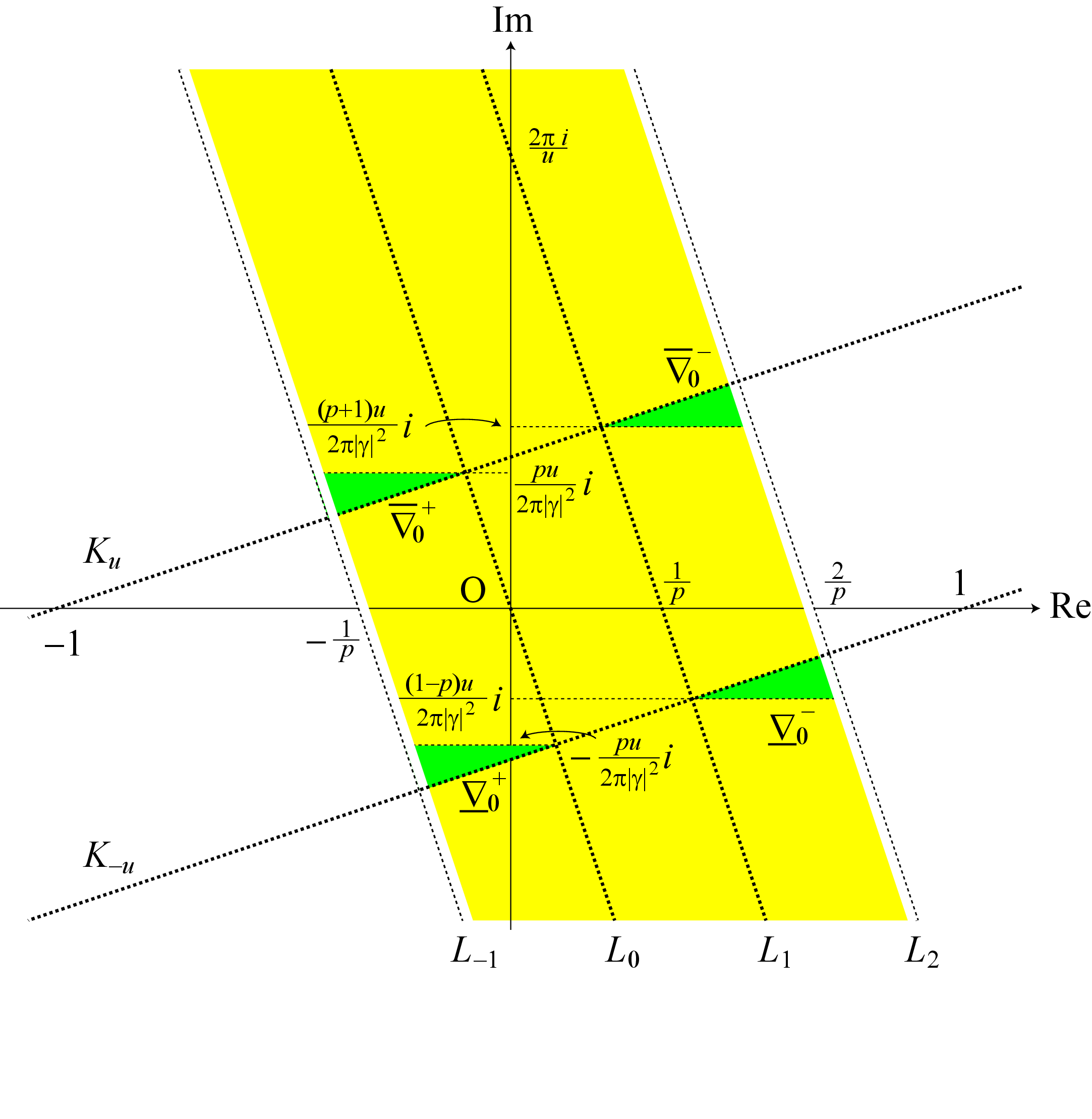}
\caption{The function $f_N$ is defined in the yellow region.
The green triangles are $\underline{\nabla}_{0}^{+}$, $\underline{\nabla}_{0}^{-}$, $\overline{\nabla}_{0}^{+}$, and $\overline{\nabla}_{0}^{-}$.}
\label{fig:domain_fN}
\end{figure}
\par
In the same way as \cite[Equation~(3.2)]{Murakami:CANJM2023}, we can express \eqref{eq:fig8} in terms of $f_N(z)$ using the results in Section~\ref{sec:dilog}.
\begin{equation}\label{eq:Jones_sum}
\begin{split}
  J_N(\FE;e^{\xi/N})
  =&
  \frac{1-e^{2pN\pi\i/\gamma}}{2\sinh(u/2)}
  \\
  &\times
  \sum_{m=0}^{p-1}
  \left(
    \beta_{p,m}
    \sum_{mN/p<k<(m+1)N/p}
    \exp\left(N\times f_{N}\left(\frac{2k+1}{2N}-\frac{m}{\gamma}\right)\right)
  \right),
\end{split}
\end{equation}
where we put
\begin{equation}\label{eq:def_beta}
  \beta_{p,m}
  :=
  e^{2mpN\pi\i/\gamma}
  \prod_{j=1}^{m}
  \left(1-e^{-2(p-j)N\pi\i/\gamma}\right)
  \left(1-e^{-2(p+j)N\pi\i/\gamma}\right).
\end{equation}
\par
From Proposition~\ref{prop:TN_L2}, the series of functions $\{f_N(z)\}_{N=2,3,4,\dots}$ uniformly converges to the following function
\begin{equation}\label{eq:F_def}
  F(z)
  :=
  \frac{1}{\xi}\L_2\bigl(\gamma(1-z)-p+1\bigr)
  -
  \frac{1}{\xi}\L_2\bigl(\gamma(1+z)-p\bigr)
  -uz
  +
  \frac{4p\pi^2}{\xi}
  +2\pi\i
\end{equation}
in the region $\Theta_{0,\nu}$ defined in \eqref{eq:Theta} below.
For a proof see \cite[Lemma~3.2]{Murakami:arXiv2023}.
\par
From Lemma~\ref{lem:der_L2_L1}, the derivatives of $F(z)$ are
\begin{align}
  F'(z)
  &=
  \L_1\bigl(\gamma(1-z)-p+1\bigr)+\L_1\left(\gamma(1+z)-p\right)-u,
  \label{eq:F'}
  \\
  F''(z)
  &=
  \frac{-\xi\sinh(\xi z)}{\cosh(u)-\cosh(\xi z)}.
  \label{eq:F''}
\end{align}
To find a point where $F'(z)$ vanishes, we introduce the following function:
\begin{equation}\label{eq:phi}
\begin{split}
  \varphi(u)
  :=&
  \arccosh(\cosh{u}-1/2)
  \\
  =&
  \log
  \left(
    \cosh{u}-\frac{1}{2}+\frac{1}{2}\sqrt{(2\cosh{u}-3)(2\cosh{u}+1)}
  \right)
\end{split}
\end{equation}
for $u>\kappa:=\arccosh(3/2)$.
\begin{rem}
Note that $\varphi(u)$ is real and positive since $\cosh{u}>\cosh(\kappa)=3/2$.
Note also that in \cite[Equation~(4.3)]{Murakami:CANJM2023}, the sign in front of the square root in \eqref{eq:phi} is different.
\end{rem}
If we put
\begin{equation}\label{eq:sigma_0}
  \sigma_0
  :=
  \frac{\varphi(u)+2\pi\i}{\xi},
\end{equation}
then we have from \eqref{eq:F'}
\begin{equation*}
\begin{split}
  F'(\sigma_0)
  =&
  \L_1\left(\frac{u-\varphi(u)}{2\pi\i}\right)
  +
  \L_1\left(1+\frac{u+\varphi(u)}{2\pi\i}\right)
  -u
  \\
  =&
  \log\left(1-e^{-u+\varphi(u)}\right)+\log\left(1-e^{-u-\varphi(u)}\right)
  +u
  \\
  =&
  \log\left(e^u+e^{-u}-e^{\varphi(u)}-e^{-\varphi(u)}\right)
  =0,
\end{split}
\end{equation*}
where the second equality follows from \eqref{eq:L0_L1_L2} because $u>\varphi(u)$ from Lemma~\ref{lem:phi} (i) below.
\begin{rem}
Since $\Re(\gamma\sigma_0)=1$, $\Im(\gamma\sigma_0)=-\frac{\varphi(u)}{2\pi}$, the point $\sigma_0$ is the crossing between $L_1$ and $K_{-\varphi(u)}$ from \eqref{eq:L} and \eqref{eq:K}.
\end{rem}
Here are more properties of $\varphi(u)$ that are used later.
\begin{lem}\label{lem:phi}
The function $\varphi(u)$ {\rm(}$u>\kappa${\rm)} is real analytic and satisfies the following properties.
\begin{enumerate}
\item
$0<u-\kappa<\varphi(u)<u$,
\item
$\varphi(u)$ is strictly monotonically increasing,
\item
$u-\varphi(u)$ is strictly monotonically decreasing with $\lim_{u\to\infty}\bigl(u-\varphi(u)\bigr)=0$,
\item
$\varphi(u)/u$ is strictly monotonically increasing,
\item
$\varphi'(u)$ is strictly monotonically decreasing with $\lim_{u\to\kappa}\varphi'(u)=\infty$ and $\lim_{u\to\infty}\varphi'(u)=1$,
\item
$e^{u}-e^{\varphi(u)}$ is strictly monotonically decreasing,
\item
$\lim_{u\to\infty}\left(e^{u}-e^{\varphi(u)}\right)=1$.
\end{enumerate}
\end{lem}
\begin{proof}
Since $u>\kappa=\arccosh(3/2)$, it is clear from the definition \eqref{eq:phi} that $\varphi(u)$ is real analytic.
\begin{enumerate}
\item
Since we have $\cosh\varphi(u)=\cosh{u}-1/2<\cosh{u}$ and $\arccosh$ is monotonically increasing, the inequality $\varphi(u)<u$ follows.
So we have $\varphi'(u)=\frac{\sinh{u}}{\sinh\varphi(u)}>1$, which implies that the function $\varphi(u)-u+\kappa$ is strictly increasing and vanishes when $u=\kappa$.
It follows that $\varphi(u)-u+\kappa>0$ for $u>\kappa$.
\item
This follows since $\varphi'(u)=\frac{\sinh{u}}{\sinh\varphi(u)}>1$.
\item
Since $\frac{d}{d\,u}\bigl(u-\varphi(u)\bigr)=1-\frac{\sinh{u}}{\sinh\varphi(u)}<0$, we conclude that $u-\varphi(u)$ is strictly monotonically decreasing.
The second statement follows since we have
\begin{equation*}
  \lim_{u\to\infty}e^{u-\varphi(u)}
  =
  \lim_{u\to\infty}\frac{\cosh(u)}{\cosh\varphi(u)}
  =
  \lim_{u\to\infty}\frac{\cosh(u)}{\cosh(u)-1/2}
  =
  1,
\end{equation*}
where the first equality holds because $\varphi(u)\to\infty$ as $u\to\infty$ from (i).
\item
The derivative of $\varphi(u)/u$ equals $\frac{u\varphi'(u)-\varphi(u)}{u^2}$, which is greater than $\frac{\varphi'(u)-1}{u}$ from (i).
Since $\varphi'(u)>1$ from the proof of (ii), we conclude that $\varphi(u)/u$ is increasing.
\item
We show that $\varphi''(u)<0$ for $u>\kappa$.
Since $\varphi'(u)=\frac{\sinh{u}}{\sinh\varphi(u)}$ and $\cosh\varphi(u)=\cosh{u}-1/2$, we have
\begin{equation*}
\begin{split}
  \varphi''(u)
  =&
  \frac{\cosh{u}\sinh\varphi(u)-\varphi'(u)\sinh{u}\cosh\varphi(u)}{\sinh^2\varphi(u)}
  \\
  =&
  \frac{\cosh{u}\sinh^2\varphi(u)-\sinh^2{u}\cosh\varphi(u)}{\sinh^3\varphi(u)}
  \\
  =&
  \frac{\cosh{u}\bigl((\cosh{u}-1/2)^2-1\bigr)-\sinh^2{u}\bigr(\cosh{u}-1/2\bigr)}
       {\sinh^3\varphi(u)}
  \\
  =&
  \frac{-2\cosh^2{u}+\cosh{u}-2}{4\sinh^3\varphi(u)}
  \\
  =&
  \frac{-2\bigl(\cosh{u}-1/4\bigr)^2-15/8}{4\sinh^3\varphi(u)}
  <0.
\end{split}
\end{equation*}
Since $\varphi(\kappa)=0$, we have $\lim_{u\to\kappa}\varphi'(u)=\infty$.
We also have $\lim_{u\to\infty}\varphi'(u)=\lim_{u\to\infty}\frac{e^u-e^{-u}}{e^{\varphi(u)}-e^{-\varphi(u)}}=\lim_{u\to\infty}e^{u-\varphi(u)}=1$ from (iii).
\item
This follows from $\frac{d}{d\,u}\left(e^u-e^{\varphi(u)}\right)=e^u-\frac{\sinh{u}}{\sinh\varphi(u)}e^{\varphi(u)}<0$.
\item
Since $2\cosh{u}-2\cosh\varphi(u)=1$, we have $e^u-e^{\varphi(u)}+e^{-u}-e^{-\varphi(u)}=1$.
Letting $u\to\infty$, we obtain the desired limit because $\lim_{u\to\infty}e^{-\varphi(u)}=0$ from (i).
\end{enumerate}
\end{proof}
\par
The following lemma and its corollary make calculations easy.
\begin{lem}
If $z$ is between $K_{u}$ and $K_{-u}$, or between $L_{0}$ and $L_{1}$, we have
\begin{equation}\label{eq:FKL}
  F(z)
  =
  \frac{1}{\xi}\Li_2\left(e^{-\xi(1+z)}\right)
  -
  \frac{1}{\xi}\Li_2\left(e^{-\xi(1-z)}\right)
  +uz.
\end{equation}
Moreover, if $z$ is between $L_{0}$ and $L_{1}$, we also have
\begin{equation}\label{eq:FL}
  F(z)
  =
  \frac{1}{\xi}\Li_2\left(e^{\xi(1-z)}\right)
  -
  \frac{1}{\xi}\Li_2\left(e^{\xi(1+z)}\right)
  -uz+\frac{4p\pi^2}{\xi}+2\pi\i.
\end{equation}
\end{lem}
Since the proof is similar to that of \cite[Lemma~3.3]{Murakami:arXiv2023}, we omit it.
\par
As a corollary we have
\begin{cor}\label{cor:F'}
If $z$ is between $K_{u}$ and $K_{-u}$, or between $L_{0}$ and $L_{1}$, we have
\begin{equation*}
\begin{split}
  F'(z)
  =&
  \log(1-e^{-u-\xi z})+\log(1-e^{-u+\xi z})+u
  \\
  =&
  \log\bigl(2\cosh{u}-2\cosh(\xi z)\bigr).
\end{split}
\end{equation*}
\end{cor}
The second equality follows from the same reason as \cite[Equation~(4.2)]{Murakami:CANJM2023}.
\par
Now, we put
\begin{equation*}
  g_{N,m}(z)
  :=
  f_N(z-m/\gamma),
\end{equation*}
so that we have
\begin{multline}\label{eq:sum_f_g}
  \sum_{m/p<k/N<(m+1)/p}
  \exp\left(N\times f_N\left(\frac{2k+1}{2N}-\frac{m}{\gamma}\right)\right)
  \\
  =
  \sum_{m/p<k/N<(m+1)/p}
  \exp\left(N\times g_{N,m}\left(\frac{2k+1}{2N}\right)\right).
\end{multline}
For each $m$ ($0\le m\le p-1$), the series of functions $\{g_{N,m}(z)\}_{N=2,3,\dots}$ uniformly converges to
\begin{equation*}
  G_m(z)
  :=
  F(z-m/\gamma)
\end{equation*}
in the region
\begin{multline}\label{eq:Theta}
  \Theta_{m,\nu}
  :=
  \left\{
    z\in\C\Bigm|
    m-1+\nu\le\Re(\gamma z)\le m+2-\nu,
    |\Im(\gamma z)|\le M-\frac{u}{2\pi}
  \right\}
  \\
  \setminus
  \left(
    \underline{\nabla}_{m,\nu}^{+}
    \cup
    \underline{\nabla}_{m,\nu}^{-}
    \cup
    \overline{\nabla}_{m,\nu}^{+}
    \cup
    \overline{\nabla}_{m,\nu}^{-}
  \right),
\end{multline}
where $M>u/(2\pi)$, and we put
\begin{equation}\label{eq:nabla}
\begin{split}
  \underline{\nabla}_{m,\nu}^{+}
  :=&
  \Bigl\{
    z\in\C\Bigm|
    m-1+\nu\le\Re(\gamma z)<m+\nu,
    \Im(\gamma z)>-\frac{u}{2\pi}-\nu,
  \\
  &\phantom{\Bigl\{z\in\C\Bigm|m-1+\nu<\Re(\gamma z)<m+2-\nu}\quad
    \Im{z}<\frac{(m-p+\nu)u}{2\pi|\gamma|^2}
  \Bigr\},
  \\
  \underline{\nabla}_{m,\nu}^{-}
  :=&
  \Bigl\{
    z\in\C\Bigm|
    m+1-\nu<\Re(\gamma z)\le m+2-\nu,
    \Im(\gamma z)<-\frac{u}{2\pi}+\nu,
  \\
  &\phantom{\Bigl\{z\in\C\Bigm|m-1+\nu<\Re(\gamma z)<m+2-\nu}\quad
    \Im{z}>\frac{(1-p+m-\nu)u}{2\pi|\gamma|^2}
  \Bigr\},
  \\
  \overline{\nabla}_{m,\nu}^{+}
  :=&
  \Bigl\{
    z\in\C\Bigm|
    m-1+\nu\le\Re(\gamma z)<m+\nu,
    \Im(\gamma z)>\frac{u}{2\pi}-\nu,
  \\
  &\phantom{\Bigl\{z\in\C\Bigm|m-1+\nu<\Re(\gamma z)<m+2-\nu}\quad
    \Im{z}<\frac{(p+m+\nu)u}{2\pi|\gamma|^2}
  \Bigr\},
  \\
  \overline{\nabla}_{m,\nu}^{-}
  :=&
  \Bigl\{
    z\in\C\Bigm|
    m+1-\nu<\Re(\gamma z)\le m+2-\nu,
    \Im(\gamma z)<\frac{u}{2\pi}+\nu,
  \\
  &\phantom{\Bigl\{z\in\C\Bigm|m-1+\nu<\Re(\gamma z)<m+2-\nu}\quad
    \Im{z}>\frac{(p+m+1-\nu)u}{2\pi|\gamma|^2}
  \Bigr\}.
\end{split}
\end{equation}
See Figure~\ref{fig:domain_g}, where we put
\begin{equation}\label{eq:I}
  I_{r}
  :=
  \{z\in\C\mid\Re{z}=r/p\}.
\end{equation}
\begin{figure}[h]
\pic{0.3}{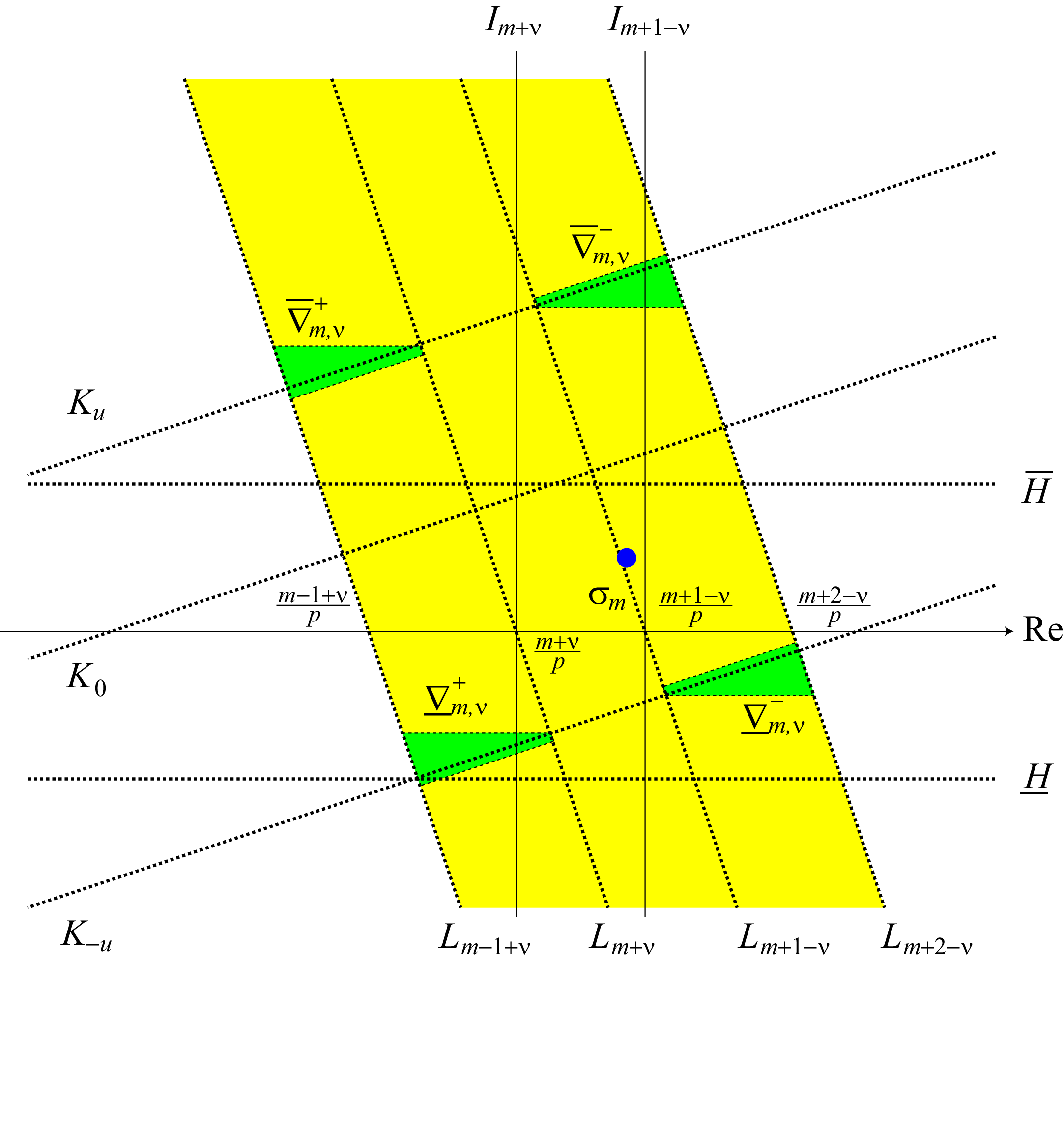}
\caption{The series of functions $\{g_{N,m}(z)\}_{N=2,3,\dots}$ converges to $G_m(z)$ in the yellow region $\Theta_{m,\nu}$.}
\label{fig:domain_g}
\end{figure}
\begin{rem}
The trapezoidal regions defined above can be described as follow.
\begin{itemize}
\item
$\overline{\nabla}_{m,\nu}^{+}$ is surrounded by the lines $L_{m-1+\nu}$, $L_{m+\nu}$, $K_{u-2\pi\nu}$, and $\Im{z}=\frac{(p+m+\nu)u}{2\pi|\gamma|^2}$,
\item
$\overline{\nabla}_{m,\nu}^{-}$ is surrounded by the lines $L_{m+2-\nu}$, $L_{m+1-\nu}$, $K_{u+2\pi\nu}$, and $\Im{z}=\frac{(1+p+m-\nu)u}{2\pi|\gamma|^2}$,
\item
$\underline{\nabla}_{m,\nu}^{+}$ is surrounded by the lines $L_{m-1+\nu}$, $L_{m+\nu}$, $K_{-u-2\pi\nu}$, and $\Im{z}=\frac{(m-p+\nu)u}{2\pi|\gamma|^2}$,
\item
$\underline{\nabla}_{m,\nu}^{-}$ is surrounded by the lines $L_{m+2-\nu}$, $L_{m+1-\nu}$, $K_{-u+2\pi\nu}$, and $\Im{z}=\frac{(1-p+m-\nu)u}{2\pi|\gamma|^2}$.
\end{itemize}
\end{rem}
\par
We define
\begin{equation}\label{eq:sigma_m}
  \sigma_m
  :=
  \sigma_0+\frac{2m\pi\i}{\xi}
  =
  \frac{\varphi(u)+2(m+1)\pi\i}{\xi}.
\end{equation}
Then, we have the following lemma.
\begin{lem}\label{lem:sigma}
The point $\sigma_m$ is the intersection of the two lines $L_{m+1}$ and $K_{-\varphi(u)}$.
Moreover, it is below {\rm(}above, on, respectively{\rm)} the real axis if and only if $\varphi(u)/u>(m+1)/p$ {\rm(}$\varphi(u)/u<(m+1)/p$, $\varphi(u)/u=(m+1)/p$, respectively{\rm)}.
\end{lem}
\begin{proof}
The first statement follows from \eqref{eq:K} and \eqref{eq:L} since $\xi\sigma_m=\varphi(u)+2(m+1)\pi\i$.
The second one also follows since $|\xi|^2\Im\sigma_m=\bigl((m+1)u-p\varphi(u)\bigr)\times2\pi$.
\end{proof}
\begin{rem}\label{rem:sigma}
Note that if $m=p-1$, then the point $\sigma_{p-1}$ is always above the real axis because $\varphi(u)/u<1=(m+1)/p$ from Lemma~\ref{lem:phi} (i).
\end{rem}
For the case where $m<p-1$, we have the following lemma.
\begin{lem}\label{lem:upsilon}
If $m<p-1$, there uniquely exists a real number $\upsilon_{p,m}>\kappa$ such that $\varphi(u)/u>(m+1)/p$ {\rm(}$\varphi(u)/u<(m+1)/p$, $\varphi(u)/u=(m+1)/p$, respectively{\rm)} if and only if $u>\upsilon_{p,m}$ {\rm(}$u<\upsilon_{p,m}$, $u=\upsilon_{p,m}$, respectively{\rm)}.
\par
Moreover, we have the inequality $\upsilon_{p,m}<\sqrt{p}$.
\end{lem}
\begin{proof}
For a constant $0<a<1$, we have
\begin{equation*}
  \frac{d}{d\,u}\bigl(\cosh{u}-\cosh(au)\bigr)
  =
  \sinh{u}-a\sinh(au)>0.
\end{equation*}
Putting $a:=(m+1)/p$, we conclude that the function $\cosh{u}-\cosh\bigl((m+1)u/p\bigr)$ is increasing, equals $0$ when $u=0$, and goes to infinity when $u\to\infty$ because $\cosh{u}\sim\frac{1}{2}e^{u}$ as $u\to\infty$.
So there exists a unique positive number $\upsilon_{p,m}$ such that $\cosh(\upsilon_{p,m})-\cosh\left(\frac{m+1}{p}\upsilon_{p,m}\right)=\frac{1}{2}$, that is, $\varphi(\upsilon_{p,m})=\frac{m+1}{p}\upsilon_{p,m}$.
\par
We can also see that $\cosh{u}-\cosh\left(\frac{m+1}{p}u\right)$ is greater than $1/2$ (less than $1/2$, respectively) if and only if $u>\upsilon_{p,m}$ ($u<\upsilon_{p,m}$, respectively).
This shows that $\varphi(u)/u>(m+1)/p$ ($\varphi(u)/u<(m+1)/p$, respectively) if and only if $u>\upsilon_{p,m}$ ($u<\upsilon_{p,m}$, respectively).
Note that $\upsilon_{p,m}>\kappa$ because $\cosh(\kappa)-\cosh\left(\frac{m+1}{p}\kappa\right)<\frac{3}{2}-1=\frac{1}{2}$.
\par
To show that $\upsilon_{p,m}<\sqrt{p}$, we only need to prove that $\cosh(\sqrt{p})-\cosh\left(\frac{m+1}{\sqrt{p}}\right)>\frac{1}{2}$.
Since $m+1\le p-1$, we have $\cosh(\sqrt{p})-\cosh\left(\frac{m+1}{\sqrt{p}}\right)\ge\cosh(\sqrt{p})-\cosh(\sqrt{p}-1/\sqrt{p})=2\sinh\left(\sqrt{p}-\frac{1}{2\sqrt{p}}\right)\sinh\left(\frac{1}{2\sqrt{p}}\right)>2\sinh(1/2)^2=0.543081\ldots>1/2$, where we use Mathematica \cite{Mathematica} to evaluate $2\sinh(1/2)^2$.
\par
The proof is complete.
\end{proof}
\par
If $z$ is between $K_{u}$ and $K_{-u}$, or between $L_{m}$ and $L_{m+1}$, we have
\begin{equation}\label{eq:Gm}
\begin{split}
  G_m(z)
  &=
  \frac{1}{\xi}\Li_2\left(e^{-\xi(1+z)}\right)
  -
  \frac{1}{\xi}\Li_2\left(e^{-\xi(1-z)}\right)
  +
  u\left(z-\frac{m}{\gamma}\right),
  \\
  G_m'(z)
  &=
  \log(2\cosh{u}-2\cosh(\xi z)),
  \\
  G_m''(z)
  &=
  \frac{-\xi\sinh(\xi z)}{\cosh{u}-\cosh(\xi z)}
\end{split}
\end{equation}
from \eqref{eq:FKL}, Corollary~\ref{cor:F'}, and \eqref{eq:F''}.
So we have
\begin{equation}\label{eq:Gmcosh}
\begin{split}
  G_m(\sigma_m)
  =&
  F(\sigma_0),
  \\
  G_m'(\sigma_m)
  =&
  0,
  \\
  G_m''(\sigma_m)
  =&
  -2\xi\sqrt{\cosh^2\varphi(u)-1}
  =
  -\xi\sqrt{(2\cosh{u}+1)(2\cosh{u}-3)}.
\end{split}
\end{equation}
From these equalities, we see that $G_m(z)$ is of the form
\begin{equation}\label{eq:Gm_saddle}
  G_m(z)
  =
  F(\sigma_0)
  +
  a_2(z-\sigma_m)^2
  +
  a_3(z-\sigma_m)^3
  +
  a_4(z-\sigma_m)^4
  +\dots
\end{equation}
with $a_2:=-\frac{\xi}{2}\sqrt{(2\cosh{u}+1)(2\cosh{u}-3)}$.
\par
Now, we will prove that if $\sigma_m$ is below or on the real axis, then we can ignore the sum \eqref{eq:sum_f_g}.
In fact we can prove the following proposition.
\begin{prop}\label{prop:sigma_below}
If $\varphi(u)/u\ge(m+1)/p$, then the sum $\sum_{m/p<k/N<(m+1)/p}\exp\left(N g_{N,m}\left(\frac{2k+1}{2N}\right)\right)$ is of order $O\left(N e^{N\bigl(\Re{F(\sigma_0)}-\varepsilon\bigr)}\right)$ as $N\to\infty$ for some $\varepsilon>0$.
\end{prop}
Note that $\Re{F(\sigma_0)}>0$ from Lemma~\ref{lem:F_sigma0} below.
\par
To prove Proposition~\ref{prop:sigma_below}, we prepare several lemmas.
\begin{lem}\label{lem:G_der_x}
If $z$ is between $K_{\varphi(u)}$ and $K_{-\varphi(u)}$, then $\Re G_{m}(z)$ is monotonically increasing with respect to $\Re{z}$.
\end{lem}
\begin{proof}
We will show that $\frac{\partial}{\partial\,x}\Re{G_m(z)}>0$, where we put $x:=\Re{z}$.
From \eqref{eq:Gm}, we have
\begin{equation*}
  \frac{\partial}{\partial\,x}\Re{G_m(z)}
  =
  \log\bigl|2\cosh{u}-2\cosh(\xi z)\bigr|.
\end{equation*}
Writing $\xi z=s+t\i$ with $s,t\in\R$, we have
\begin{equation*}
  \bigl|2\cosh{u}-2\cosh(\xi z)\bigr|^2
  =
  4\bigl(\cosh{u}-\cosh(s)\cos(t)\bigr)^2
  +
  4\sinh^2(s)\sin^2(t).
\end{equation*}
\par
We see that $\cosh{u}>3/2$ since $u>\kappa$, and $1\le\cosh{s}<\cosh{u}-1/2$ since $s=\Re(\xi z)$ is between $-\varphi(u)$ and $\varphi(u)$.
Putting $S:=\cosh{s}$, $T:=\cos{t}$, $U:=\cosh{u}$, we define
\begin{equation*}
\begin{split}
  h(S)
  :=&
  \bigl(U-ST\bigr)^2
  +
  (S^2-1)(1-T^2)
  \\
  =&
  (S-UT)^2-(U^2-1)(T^2-1),
\end{split}
\end{equation*}
and regard it as a quadratic function of $S$ fixing $U>3/2$ and $|T|\le1$.
We will show that $h(S)>1/4$ when $1\le S<U-1/2$ in a usual way.
\begin{enumerate}
\item
When $UT\le 1$:
We see that $h(S)\ge h(1)=(U-T)^2>1/4$ since $U>3/2$ and $T\le1$.
\par
\item
When $1<UT\le U-1/2$:
The minimum of $h(S)$ takes place at $S=UT$.
Now, we have
\begin{equation*}
\begin{split}
  h(UT)
  =&
  U^2+T^2-U^2T^2-1
  \\
  &\text{(since $UT\le U-1/2$)}
  \\
  \ge&
  U^2+T^2-(U-1/2)^2-1
  \\
  =&
  T^2+U-\frac{5}{4}
  >\frac{1}{4}.
\end{split}
\end{equation*}
\item
When $U-1/2<UT$:
In this case we see that $h(S)>h(U-1/2)$.
We calculate
\begin{equation*}
\begin{split}
  h(U-1/2)
  =&
  2(1-T)U^2-(1-T)U+T^2-3/4
  \\
  =&
  2(1-T)(U-1/4)^2+T^2+T/8-7/8
  \\
  &\text{(since $U>3/2$)}
  \\
  >&
  2(1-T)(5/4)^2+T^2+T/8-7/8
  \\
  =&
  (T-3/2)^2
  \ge1/4.
\end{split}
\end{equation*}
\end{enumerate}
Therefore we conclude that $h(S)>1/4$, which shows that $\frac{\partial}{\partial\,x}\Re G_m(z)>0$, completing the proof.
\end{proof}
Using Lemma~\ref{lem:G_der_x}, we can prove the following lemma.
\begin{lem}\label{lem:saddle_below}
Suppose that $\varphi(u)/u\ge(m+1)/p$, that is, $\sigma_m$ is below or on the real axis.
Then $\Re G_m(x)<\Re G_m\bigl((m+1)/p\bigr)$ for $x\in\R$ with $m/p\le x<(m+1)/p$.
\end{lem}
\begin{proof}
From Lemma~\ref{lem:G_der_x} it is enough to show that the segment $[m/p,(m+1)/p]$ in the real axis is between $K_{\varphi(u)}$ and $K_{-\varphi(u)}$.
If $x\in\R$ satisfies $m/p\le x<(m+1)/p$, then $|\Re(\xi x)|=ux<\varphi(u)$ since $\varphi(u)/u\ge(m+1)/p$, proving that $x$ is between $K_{-\varphi(u)}$ and $K_{\varphi(u)}$ from \eqref{eq:K}.
\end{proof}
Define the region $W_m$ as
\begin{equation}\label{eq:Wm}
  W_m
  :=
  \{z\in\Theta_{m,\nu}\mid\Re G_m(z)<\Re F(\sigma_0)\}.
\end{equation}
\begin{lem}\label{lem:sigma_m+1}
Let $z$ be a point on $L_{m+1}$ between $K_{u}$ and $K_{-u}$.
\begin{enumerate}
\item
If $\sigma_m$ is below the real axis, then $z\in W_m$ if $z$ is between $\sigma_m$ and $(m+1)/p$, including $(m+1)/p$ but excluding $\sigma_m$.
\item
If $\sigma_m$ is above the real axis, then $z\in W_m$ if $z$ is between $K_{\varphi(u)}$ and $(m+1)/p$, unless $z=\sigma_m$.
\end{enumerate}
\end{lem}
\begin{proof}
A point on $L_{m+1}$ between $K_u$ and $K_{-u}$ is parametrized as $(m+1)/p+\overline{\xi}t$ with $\frac{-(p+m+1)u}{p|\xi|^2}\le t\le\frac{(p-m-1)u}{p|\xi|^2}$.
\par
Note the following:
\begin{itemize}
\item
$t=\frac{-(p+m+1)}{p|\xi|^2}$, $\frac{-p\varphi(u)-(m+1)u}{p|\xi|^2}$, $\frac{p\varphi(u)-(m+1)u}{p|\xi|^2}$, and $\frac{(p-m-1)u}{p|\xi|^2}$ correspond to the points on $K_{u}$, $K_{\varphi(u)}$, $K_{-\varphi(u)}$, and $K_{-u}$, respectively,
\item
$t=\frac{p\varphi(u)-(m+1)u}{p|\xi|^2}$ corresponds to $\sigma_m$,
\item
$t=0$ corresponds to the point $(m+1)/p$.
\end{itemize}
\par
From \eqref{eq:Gm}, we have
\begin{equation*}
\begin{split}
  &\frac{d}{d\,t}
  \Re G_m((m+1)/p+\overline{\xi}t)
  \\
  =&
  \Re
  \left(
    \overline{\xi}
    \log\left(2\cosh{u}-2\cosh\left(\frac{(m+1)\xi}{p}+|\xi|^2t\right)\right)
  \right)
  \\
  =&
  u
  \log\left(2\cosh{u}-2\cosh\left(\frac{(m+1)u}{p}+|\xi|^2t\right)\right),
\end{split}
\end{equation*}
which is positive if and only if $\cosh{u}-\cosh\bigl((m+1)u/p+|\xi|^2t\bigr)>1/2$, that is, $\frac{-p\varphi(u)-(m+1)u}{p|\xi|^2}<t<\frac{p\varphi(u)-(m+1)u}{p|\xi|^2}$.
\begin{enumerate}
\item
If $\sigma_m$ is below  the real axis, we have $\varphi(u)/u>(m+1)/p$ from Lemma~\ref{lem:sigma}.
If $z$ is between $\sigma_m$ and $(m+1)/p$, then $0\le t\le\frac{p\varphi(u)-(m+1)u}{p|\xi|^2}$.
So we conclude that $\Re G_m((m+1)/p+\overline{\xi}t)$ is strictly increasing with respect to $t$ (see Table~\ref{table:ReG_m_below}), proving $\Re G_m(z)<\Re F(\sigma_0)$ unless $z=\sigma_m$.
\begin{table}[h]
{\tiny
\begin{tabular}{|c||c|c|c|c|c|c|c|c|c|c|}
  \hline
  $t$&
  $\frac{-(p+m+1)u}{p|\xi|^2}$&&
  $\frac{-p\varphi(u)-(m+1)u}{p|\xi|^2}$&&
  $0$&&
  $\frac{p\varphi(u)-(m+1)u}{p|\xi|^2}$&&
  $\frac{(p-m-1)u}{p|\xi|^2}$ \\
  \hline
  $\frac{d\,\Re{G_m}}{d\,t}$&&$-$&$0$&$+$&$+$&$+$&$0$&$-$& \\
  \hline
  $\Re{G_m}$&
  &$\searrow$&&$\nearrow$&$\nearrow$&$\nearrow$&$\Re{F(\sigma_0)}$&$\searrow$& \\
  \hline
  location&on $K_u$&&on $K_{\varphi(u)}$&&$\frac{m+1}{p}$&&$\sigma_m$&&on $K_{-u}$ \\
  \hline
\end{tabular}}
\caption{The derivative sign chart of $\Re{G_m}\bigl((m+1)/p-\overline{\xi}t\bigr)$ when $p\varphi(u)-(m+1)u>0$.}
\label{table:ReG_m_below}
\end{table}
\item
If $\sigma_m$ is above the real axis, we have $\varphi(u)/u<(m+1)/p$ from Lemma~\ref{lem:sigma} again.
We have the following table (Table~\ref{table:ReG_m_above}) from the argument above.
Note that in this case $\frac{p\varphi(u)-(m+1)u}{p|\xi|^2}<0$.
If $z$ is between $K_{\varphi(u)}$ and $(m+1)/p$, then we have $\frac{-p\varphi(u)-(m+1)u/p}{|\xi|^2}\le t\le0$.
From Table~\ref{table:ReG_m_above}, we see that $\Re G_m((m+1)/p+\overline{\xi}t)$ is increasing (decreasing, respectively) with respect to $t$ when $\frac{-p\varphi(u)-(m+1)u}{p|\xi|^2}\le t<\frac{p\varphi(u)-(m+1)u}{p|\xi|^2}$ ($\frac{p\varphi(u)-(m+1)u}{p|\xi|^2}<t\le0$, respectively), proving that $\Re G_m(z)<\Re F(\sigma_0)$ if $z$ is between $K_{\varphi(u)}$ and $(m+1)/p$, unless $z=\sigma_m$.
\begin{table}[h]
{\tiny
\begin{tabular}{|c||c|c|c|c|c|c|c|c|c|c|}
  \hline
  $t$&
  $\frac{-(p+m+1)u}{p|\xi|^2}$&&
  $\frac{-p\varphi(u)-(m+1)u}{p|\xi|^2}$&&
  $\frac{p\varphi(u)-(m+1)u}{p|\xi|^2}$&&
  $0$&&
  $\frac{(p-m-1)u}{p|\xi|^2}$ \\
  \hline
  $\frac{d\,\Re{G_m}}{d\,t}$&&$-$&$0$&$+$&$0$&$-$&$-$&$-$& \\
  \hline
  $\Re{G_m}$&
  &$\searrow$&&$\nearrow$&$\Re{F(\sigma_0)}$&$\searrow$&$\searrow$&$\searrow$& \\
  \hline
  location&on $K_u$&&on $K_{\varphi(u)}$&&$\sigma_m$&&$\frac{m+1}{p}$&&on $K_{-u}$ \\
  \hline
\end{tabular}}
\caption{The derivative sign chart of $\Re{G_m}\bigl((m+1)/p-\overline{\xi}t\bigr)$ when $p\varphi(u)-(m+1)u<0$.}
\label{table:ReG_m_above}
\end{table}
\end{enumerate}
\end{proof}
\begin{proof}[Proof of Proposition~\ref{prop:sigma_below}]
First note that $m$ is less than $p-1$ from Remark~\ref{rem:sigma}.
\par
From Lemmas~\ref{lem:saddle_below} and \ref{lem:sigma_m+1} (i), there exists $\varepsilon>0$ such that
\begin{equation}\label{eq:sigma_below_proof}
  \Re G_m(x)<\Re F(\sigma_0)-2\varepsilon
\end{equation}
if $m/p<x<(m+1)/p$.
Since the series of functions $\{g_{N,m}(z)\}_{N=2,3,\dots}$ uniformly converges to $G_m(z)$, we have $\left|g_{N,m}(z)-G_{m}(z)\right|<\varepsilon$ if $N$ is sufficiently large.
So we have
\begin{equation*}
\begin{split}
  &\left|
    \sum_{m/p<k/N<(m+1)/p}
    \exp
    \left(N\times g_{N,m}\left(\frac{2k+1}{2N}\right)\right)
  \right|
  \\
  \le&
  \sum_{m/p<k/N<(m+1)/p}
  \exp
  \left(N\times\Re g_{N,m}\left(\frac{2k+1}{2N}\right)\right)
  \\
  <&
  \sum_{m/p<k/N<(m+1)/p}
  \exp\left(N\left(\Re G_m\left(\frac{2k+1}{2N}\right)+\varepsilon\right)\right).
\end{split}
\end{equation*}
Note that if $\frac{m}{p}<\frac{k}{N}<\frac{m+1}{p}$, then $\frac{m}{p}<\frac{2k+1}{2N}$ but $\frac{2k+1}{2N}$ may not be less than $\frac{m+1}{p}$.
\par
If $\frac{m}{p}<\frac{k}{N}<\frac{2k+1}{2N}<\frac{m+1}{p}$, then from \eqref{eq:sigma_below_proof} we have $\Re G_m\left(\frac{2k+1}{2N}\right)+\varepsilon<\Re F(\sigma_0)-\varepsilon$.
Therefore we have
\begin{equation*}
  \left|
    \sum_{m/p<k/N<(m+1)/p}
    \exp
    \left(Ng_{N,m}\left(\frac{2k+1}{2N}\right)\right)
  \right|
  <
  \frac{N}{p}
  e^{N\bigl(\Re F(\sigma_0)-\varepsilon\bigr)},
\end{equation*}
proving the proposition.
\par
Suppose that $k'$ satisfies $\frac{k'}{N}<\frac{m+1}{p}\le\frac{2k'+1}{2N}$.
Notice that there exists at most one such $k'$ because $\frac{k'+1}{N}>\frac{2k'+1}{2N}$.
Moreover, we may assume that $\Re G_m\left(\frac{2k'+1}{2N}\right)<\Re F(\sigma_0)+\varepsilon$ for sufficiently large $N$.
So we have
\begin{equation*}
\begin{split}
  &\left|
    \sum_{m/p<k/N<(m+1)/p}
    \exp
    \left(N\times g_{N,m}\left(\frac{2k+1}{2N}\right)\right)
  \right|
  \\
  <&
  \frac{N}{p}e^{N\bigl(\Re F(\sigma_0)-\varepsilon\bigr)}
  +
  e^{\Re F(\sigma_0)+\varepsilon},
\end{split}
\end{equation*}
which is of order $O\left(Ne^{N\left(\Re F(\sigma_0)-\varepsilon\right)}\right)$ because $\Re F(\sigma_0)>0$ from Lemma~\ref{lem:F_sigma0} below.
\par
This completes the proof.
\end{proof}
Before proving the following lemma, note that $F(\sigma_0)=S_{\FE}(u)/\xi$ from \eqref{eq:S_def} and \eqref{eq:FKL}.
\begin{lem}\label{lem:F_sigma0}
If $p\ge1$ and $u>\kappa$, then we have $\Re{F(\sigma_0)}>0$.
\end{lem}
\begin{proof}
As noted above, we have
\begin{equation*}
\begin{split}
  \Re{F(\sigma_0)}
  =&
  \Re\left(\frac{S_{\FE}(u)}{\xi}\right)
  \\
  =&
  \frac{u}{|\xi|^2}
  \left(
    \Li_2\left(e^{-u-\varphi(u)}\right)
    -
    \Li_2\left(e^{-u+\varphi(u)}\right)
    +u\varphi(u)
    +
    4p\pi^2
  \right).
\end{split}
\end{equation*}
We will show that
\begin{equation}\label{eq:Lp}
  L_{p}(u)
  :=
  \Li_2\left(e^{-u-\varphi(u)}\right)
  -
  \Li_2\left(e^{-u+\varphi(u)}\right)
  +u\varphi(u)+4p\pi^2
\end{equation}
is positive.
\par
The derivative becomes
\begin{equation}\label{eq:Lp_der}
\begin{split}
  &\frac{d}{d\,u}L_{p}(u)
  \\
  =&
  \bigl(1+\varphi'(u)\bigr)
  \log\left(1-e^{-u-\varphi(u)}\right)
  +
  \bigl(-1+\varphi'(u)\bigr)
  \log\left(1-e^{-u+\varphi(u)}\right)
  \\
  &+\varphi(u)+u\varphi'(u)
  \\
  =&
  \log\left(\frac{1-e^{-u-\varphi(u)}}{1-e^{-u+\varphi(u)}}\right)
  +
  \varphi'(u)
  \left(\log\bigl(2\cosh{u}-2\cosh\varphi(u)\bigr)-u\right)
  +
  \varphi(u)+u\varphi'(u)
  \\
  &\text{(since $\cosh\varphi(u)=\cosh{u}-1/2$)}
  \\
  =&
  \log\left(\frac{1-e^{-u-\varphi(u)}}{1-e^{-u+\varphi(u)}}\right)
  +
  \varphi(u)
  \\
  =&
  \log\left(
  \frac{\left(e^{\varphi(u)}-e^{-u}\right)\left(1-e^{u+\varphi(u)}\right)}
       {\left(1-e^{-u+\varphi(u)}\right)\left(1-e^{u+\varphi(u)}\right)}
  \right)
  \\
  =&
  \log\left(
  \frac{2e^{\varphi(u)}-e^{-u}-e^{u+2\varphi(u)}}
       {e^{\varphi(u)}\left(e^{\varphi(u)}+e^{-\varphi(u)}-e^{u}-e^{-u}\right)}
  \right)
  \\
  =&
  \log\Bigl(2\cosh\bigl(u+\varphi(u)\bigr)-2\Bigr)
  >0.
\end{split}
\end{equation}
\par
Thus, the function $L_p(u)$ is monotonically increasing.
Since $L_p(\kappa)=4p\pi^2>0$, we conclude that $L_p(u)>0$ for $u>\kappa$, proving the lemma.
\end{proof}

\section{The region $\Xi_{m,\mu}$}\label{sec:Xi}
In this section, we introduce and study some properties of the region $\Xi_{m,\mu}\subset\Theta_{m,\nu}$.
See \eqref{eq:Theta} for the definition of $\Theta_{m,\nu}$.
Throughout this section, we assume that $\varphi(u)/u<(m+1)/p$, that is, the point $\sigma_m$ is above the real axis.
\par
Put
\begin{align*}
  \oH
  &:=
  \left\{z\in\C\Bigm|\Im{z}=2\Im\sigma_{m}\right\},
  \\
  \uH
  &:=
  \left\{z\in\C\Bigm|\Im{z}=-2\Im\sigma_{m}\right\}.
\end{align*}
Recall that $2\Im\sigma_{m}=\frac{4\pi\bigl((m+1)u-p\varphi(u)\bigr)}{|\xi|^2}>0$.
\begin{lem}\label{lem:H_L_m+1/2}
The line $L_{m+1/2}$ intersects with $I_{m+1}$ below the line $\uH$.
See \eqref{eq:L} and \eqref{eq:I} for the definitions of $L_t$ and $I_r$, respectively.
\end{lem}
\begin{proof}
The imaginary part of $L_{m+1/2}\cap I_{m+1}$ is $-\pi/u$.
Now, we calculate
\begin{equation*}
\begin{split}
  -\frac{\pi}{u}+2\Im\sigma_m
  =&
  -\frac{\pi}{u}+\frac{4\pi\bigl((m+1)u-p\varphi(u)\bigr)}{|\xi|^2}
  \\
  &\text{(since $m+1\le p$)}
  \\
  \le&
  -\frac{\pi}{u}+\frac{4p\pi\bigl(u-\varphi(u)\bigr)}{|\xi|^2}
  \\
  =&
  \frac{\pi}{u|\xi|^2}\left(4pu\bigl(u-\varphi(u)\bigr)-u^2-4p^2\pi^2\right)
  \\
  &\text{(since $u-\varphi(u)$ is decreasing from Lemma~\ref{lem:phi} (iii))}
  \\
  <&
  \frac{\pi}{\kappa|\xi|^2}\left(4p\kappa^2-\kappa^2-4p^2\pi^2\right)
  \\
  =&
  \frac{\pi}{\kappa|\xi|^2}
  \left(
    -\left(2\pi p-\frac{\kappa^2}{\pi}\right)^2+\frac{\kappa^4}{\pi^2}-\kappa^2
  \right)
  \\
  <0.
\end{split}
\end{equation*}
Therefore we conclude that the intersection between $L_{m+1/2}$ and $I_{m+1}$ is below $\uH$.
\end{proof}
For a small number $\mu>0$, let $J$ be the line connecting the point $(m-\mu)/p$ and the intersection point $L_{m+1/2}\cap\uH$, which is described as
\begin{equation}\label{eq:J}
  J
  :=
  \left\{
    z\in\C\Bigm|
    \Re(\gamma z)
    +
    \frac{\mu+1/2}{2\Im\sigma_m}\Im{z}
    =
    m-\mu
  \right\}.
\end{equation}
From Lemma~\ref{lem:H_L_m+1/2}, we see that the lines $\oH$, $\uH$, $I_{m-\mu}$, $I_{m+1+\mu}$, and $J$ form a pentagon.
\begin{defn}
Let $\mu>0$ be a small number.
\par
If $m\le p-2$, let $\Xi_{m,\mu}$ be the open, pentagonal region surrounded by $\oH$, $\uH$, $I_{m-\mu}$, $I_{m+1+\mu}$, and $J$.
\par
If $m=p-1$, let $\Xi_{m,\mu}$ be the open, pentagonal region surrounded by $\oH$, $\uH$, $I_{m-\mu}$, $I_{m+1+\mu}$, and $J$, minus $\underline{\nabla}_{p-1,\nu}^{-}$.
Note that $\underline{\nabla}_{p-1,\nu}^{-}$ contains the point $1$ from the definition \eqref{eq:nabla}.
\end{defn}
\begin{lem}
We can choose $\mu>0$ so that $\Xi_{m,\mu}$ is contained in $\Theta_{m,\nu}$, provided that $\nu>0$ is sufficiently small and $M>0$ is sufficiently large.
\end{lem}
\begin{proof}
See \eqref{eq:Theta} for the definition of $\Theta_{m,\nu}$.
\par
It is enough to show the following.
\begin{enumerate}
\item
$\Xi_{m,\mu}$ is below $L_{m+2-\nu}$,
\item
$\Xi_{m,\mu}$ is above $L_{m-1+\nu}$, and
\item
$\Xi_{m,\mu}\cap\overline{\nabla}^{\pm}_{m,\nu}=\Xi_{m,\mu}\cap\underline{\nabla}^{\pm}_{m,\nu}=\emptyset$.
\end{enumerate}
To prove (i), we will show that the crossing between $\oH$ and $I_{m+1+\mu}$ is below $L_{m+2-\nu}$.
Since its coordinate is $(m+1+\mu)/p+2\Im\sigma_m\i$, we need to show the following inequality:
\begin{equation*}
  \Re\left(\gamma\bigl((m+1+\mu)/p+2\Im\sigma_m\i\bigr)\right)<m+2-\nu
\end{equation*}
from \eqref{eq:L}.
Since we calculate
\begin{equation*}
  \Re\left(\gamma\bigl((m+1+\mu)/p+2\Im\sigma_m\i\bigr)\right)
  =
  m+1+\mu+\frac{2u\bigl((m+1)u-p\varphi(u)\bigr)}{|\xi|^2},
\end{equation*}
we need to prove
\begin{equation}\label{eq:tr}
  1-\frac{2u\bigl((m+1)u-p\varphi(u)\bigr)}{|\xi|^2}>\mu+\nu.
\end{equation}
\par
Since $m\le p-1$, we have
\begin{equation*}
  1-\frac{2u\bigl((m+1)u-p\varphi(u)\bigr)}{|\xi|^2}
  \ge
  1-\frac{2pu\bigl(u-\varphi(u)\bigr)}{|\xi|^2}
  =
  \frac{u^2+4\pi^2p^2-2pu\bigl(u-\varphi(u)\bigr)}{|\xi|^2}.
\end{equation*}
The numerator equals
\begin{equation*}
  \left(2\pi p-\frac{u\bigl(u-\varphi(u)\bigr)}{2\pi}\right)^2
  +
  \frac{u^2}{4\pi^2}\left(4\pi^2-\bigl(u-\varphi(u)\bigr)^2\right),
\end{equation*}
which takes its maximum $\frac{u^2}{4\pi^2}\left(4\pi^2-\bigl(u-\varphi(u)\bigr)^2\right)$ at $p=\frac{u\bigl(u-\varphi(u)\bigr)}{4\pi^2}$, if we regard $p$ as a positive real parameter.
Now, since $\frac{u^2}{4\pi^2}\left(4\pi^2-\bigl(u-\varphi(u)\bigr)^2\right)$ is increasing with respect to $u$ from Lemma~\ref{lem:phi} (iii), it is greater than $\frac{\kappa^2}{4\pi^2}(4\pi^2-\kappa^2)=0.9045\ldots$.
\par
Therefore the left hand side of \eqref{eq:tr} is positive, and so we can choose $\mu>0$ so that \eqref{eq:tr} holds provided that $\nu>0$ is sufficiently small.
\par
Since the point $m/p$ is clearly above $L_{m-1+\nu}$ if $\nu$ is small, and the line $L_{m-1+\nu}$ is steeper than the line $J$, we conclude that $\Xi_{m,\mu}$ is above $L_{m-1+\nu}$, proving (ii).
\par
It remains to show (iii).
\begin{itemize}
\item
$\overline{\nabla}_{m,\nu}^{+}$:
The real part of the bottom right corner of $\overline{\nabla}_{m,\nu}^{+}$ is $\frac{4(m+\nu)p\pi^2-u^2+2\pi u\nu}{|\xi|^2}$, which is less than $(m-\mu)/p$ if
\begin{equation}\label{eq:^+}
  2p\pi(u+2p\pi)\nu+|\xi|^2\mu
  <
  (p+m)u^2.
\end{equation}
So the trapezoid $\overline{\nabla}_{m,\nu}^{+}$ is to the left of $I_{m-\mu}$ if \eqref{eq:^+} holds.
\item
$\overline{\nabla}_{m,\nu}^{-}$:
The imaginary part of the bottom side of $\overline{\nabla}_{m,\nu}^{-}$ is $\frac{(p+m+1-\nu)u}{2\pi|\gamma|^2}=\frac{2\pi(p+m+1-\nu)u}{|\xi|^2}$, which is greater than $2\Im\sigma_m$ if $(p+m+1-\nu)u>2\bigl((m+1)u-p\varphi(u)\bigr)$.
So if
\begin{equation}\label{eq:^-}
  p-(m+1)+\frac{2p\varphi(u)}{u}>\nu,
\end{equation}
then $\overline{\nabla}_{m,\nu}^{-}$ is above $\oH$.
\item
$\underline{\nabla}_{m,\nu}^{+}$:
Let $z_0\in\C$ be the coordinate of the top right corner of the trapezoid $\underline{\nabla}_{m,\nu}^{+}$.
It is enough to show that $z_0$ is below $J$ since $L_{m+\nu}$ is steeper than $J$.
Since $\Im{z_0}=\frac{2\pi(m-p+\nu)u}{|\xi|^2}$ and $\Re(\gamma z_0)=m+\nu$, $z_0$ is below $J$ if
\begin{equation*}
  m+\nu
  +
  \frac{(\mu+1/2)|\xi|^2}{4\pi\bigl((m+1)u-p\varphi(u)\bigr)}
  \times\frac{2\pi(m-p+\nu)u}{|\xi|^2}
  <
  m-\mu
\end{equation*}
from \eqref{eq:J}.
So we conclude that $\underline{\nabla}_{m,\nu}^{+}\cap\Xi_{m,\mu}=\emptyset$ if
\begin{equation}\label{eq:_+}
  \mu\nu+\left(3m+2-p-\frac{2p\varphi(u)}{u}\right)\mu
  +
  \left(2m+\frac{5}{2}-\frac{2p\varphi(u)}{u}\right)\nu
  <\frac{p-m}{2}.
\end{equation}
Note that $2m+5/2-2p\varphi(u)/u>0$ because we assume that $\varphi(u)/u<(m+1)/p$.
Note also that $3m+2-p-2p\varphi(u)/u$ can be negative.
\item
$\underline{\nabla}_{m,\nu}^{-}$:
The real part of the top left corner of the trapezoid $\underline{\nabla}_{m,\nu}^{-}$ is $\frac{4(m+1-\nu)p\pi^2+u(u-2\pi\nu)}{|\xi|^2}$, which is greater than $\frac{m+1+\mu}{p}$ if
\begin{equation}\label{eq:_-}
  |\xi|^2\mu+2p\pi(2p\pi+u)\nu
  <
  (p-m-1)u^2.
\end{equation}
\end{itemize}
From \eqref{eq:^+}--\eqref{eq:_-}, we conclude that if $m<p-1$ and
\begin{multline*}
  \nu
  <
  \min
  \left\{
    \frac{(p+m)u^2}{2p\pi(u+2p\pi)},
    p-m-1+\frac{2p\varphi(u)}{u},
  \right.
  \\
  \left.
    \frac{(p-m)u}{(4m+5)u-4p\varphi(u)},
    \frac{(p-m-1)u^2}{2p\pi(u+2p\pi)}
  \right\},
\end{multline*}
then we can choose $\mu>0$ so that $\Xi_{m,\mu}$ is included in $\Theta_{m,\nu}$.
\par
Next, we consider the case where $m=p-1$.
\par
From the definition and the argument above, we see that $\Xi_{p-1,\mu}$ is between $L_{m+2-\nu}$ and $L_{m-1+\nu}$, and that it avoids  $\overline{\nabla}_{m,\nu}^{\pm}$ and $\underline{\nabla}_{m,\nu}^{+}$.
Moreover, we exclude $\underline{\nabla}_{m,\nu}^{-}$ from $\Xi_{p-1,\mu}$ by the definition.
So we conclude that $\Xi_{p-1,\mu}\subset\Theta_{p-1,\nu}$.
\end{proof}
\par
We name the following intersection points.
We also give their coordinates.
\begin{alignat*}{2}
  I_{m}\cap L_{m}:&\quad P_0&=&\frac{m}{p},
  \\
  \uH\cap L_{m+1/2}:&\quad P_1&=&\frac{2m+1}{2p}+\frac{u\Im\sigma_m}{p\pi}-2\Im\sigma_m\i,
  \\
  \uH\cap I_{m+1}:&\quad P_2&=&\frac{m+1}{p}-2\Im\sigma_m\i,
  \\
  I_{m+1}\cap L_{m+1}:&\quad P_3&=&\frac{m+1}{p},
  \\
  \oH\cap I_{m+1}:&\quad P_4&=&\frac{m+1}{p}+2\Im\sigma_m\i,
  \\
  \oH\cap L_{m+1}:&\quad P_5&=&\frac{m+1}{p}-\frac{u\Im\sigma_m}{p\pi}+2\Im\sigma_m\i,
  \\
  \oH\cap L_{m+1/2}:&\quad P_6&=&\frac{2m+1}{2p}-\frac{u\Im\sigma_m}{p\pi}+2\Im\sigma_m\i,
  \\
  \oH\cap I_{m}:&\quad P_7&=&\frac{m}{p}+2\Im\sigma_m\i,
  \\
  K_{0}\cap L_{m+1}:&\quad Q_1&=&\frac{2(m+1)\pi\i}{\xi},
  \\
  K_{0}\cap L_{m+1/2}:&\quad Q_2&=&\frac{(2m+1)\pi\i}{\xi},
  \\
  K_{0}\cap\oH:&\quad R_1&=&\frac{4p\pi}{u}\Im\sigma_m+2\Im\sigma_m\i,
  \\
  K_{0}\cap I_{m}:&\quad R_2&=&\frac{m}{p}+\frac{mu}{2p^2\pi}\i.
\end{alignat*}
See Figure~\ref{fig:Xi}.
\begin{figure}[h]
\begin{minipage}{39mm}
\begin{center}\pic{0.21}{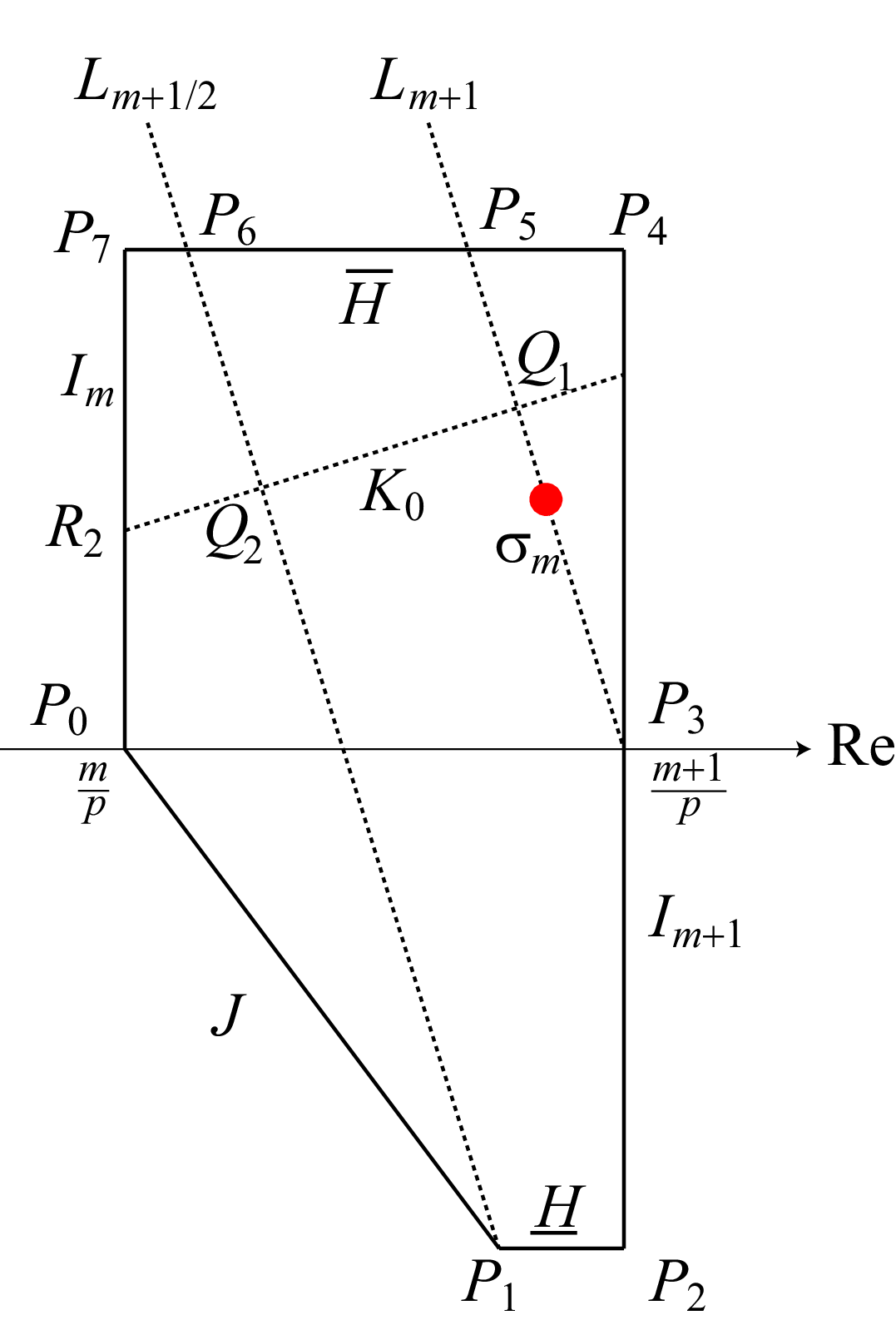}\\(i)\end{center}
\end{minipage}
\quad
\begin{minipage}{39mm}
\begin{center}\pic{0.21}{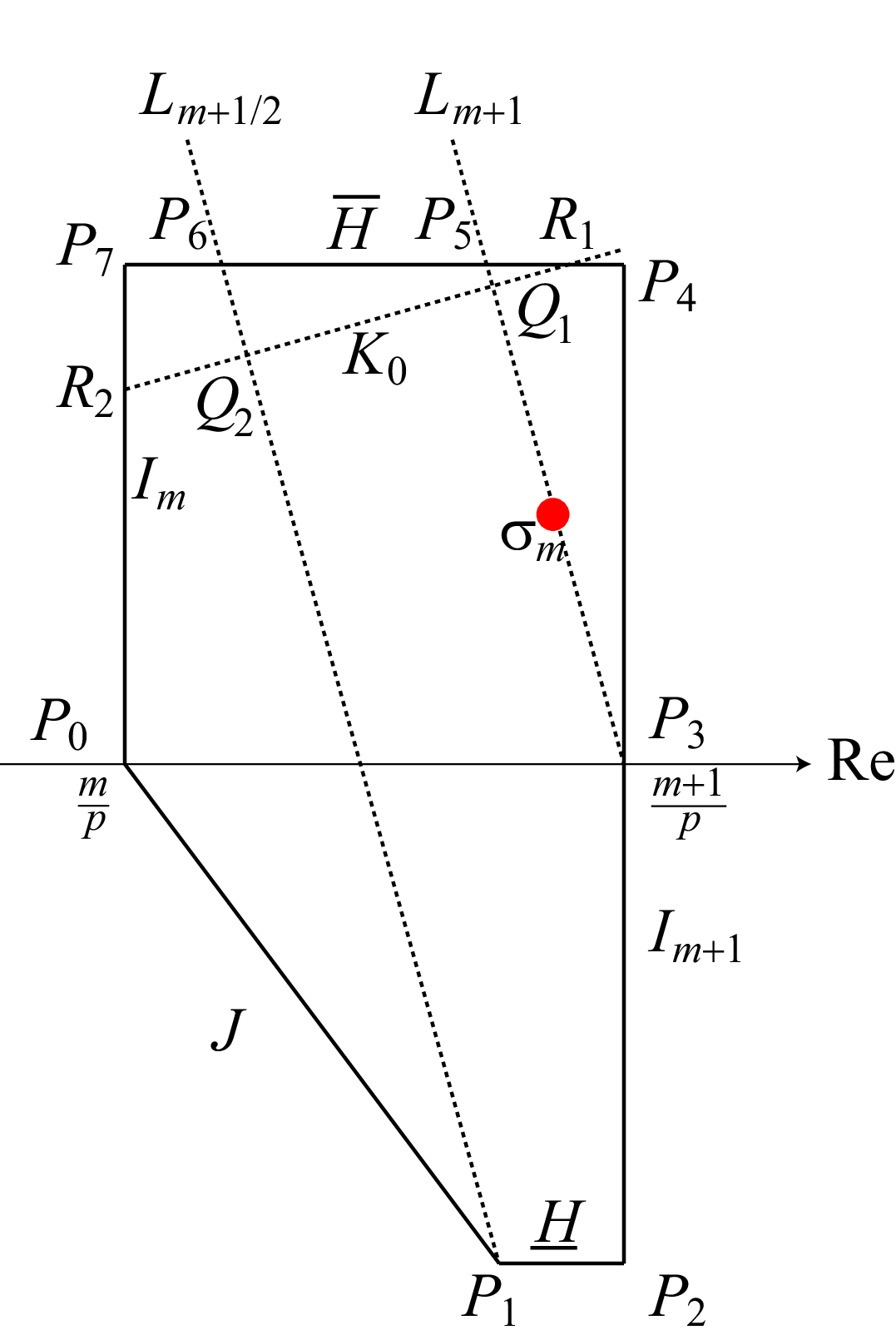}\\(ii)\end{center}
\end{minipage}
\quad
\begin{minipage}{39mm}
\begin{center}\pic{0.21}{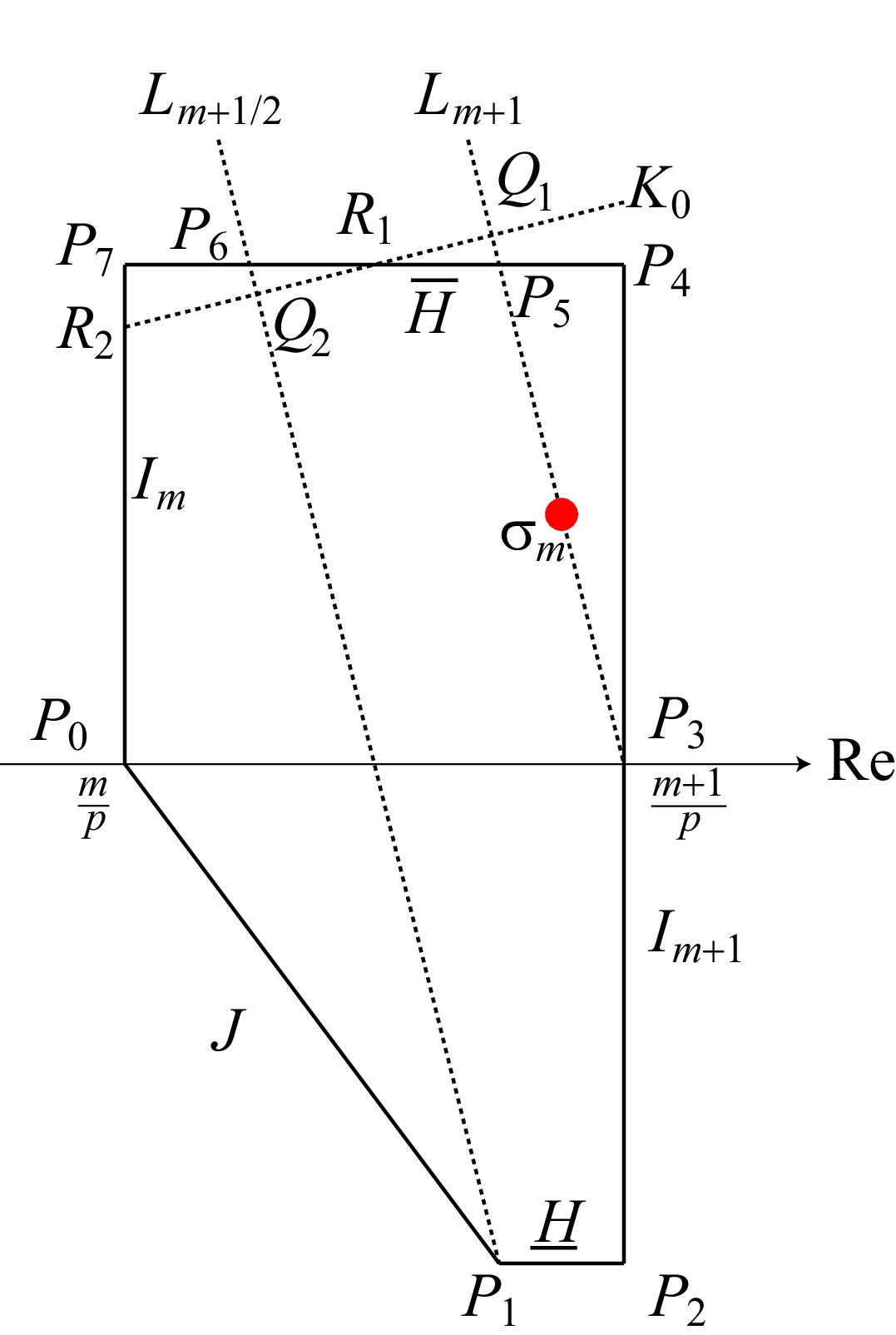}\\(iii)\end{center}
\end{minipage}
\par
\begin{minipage}{39mm}
\begin{center}\pic{0.21}{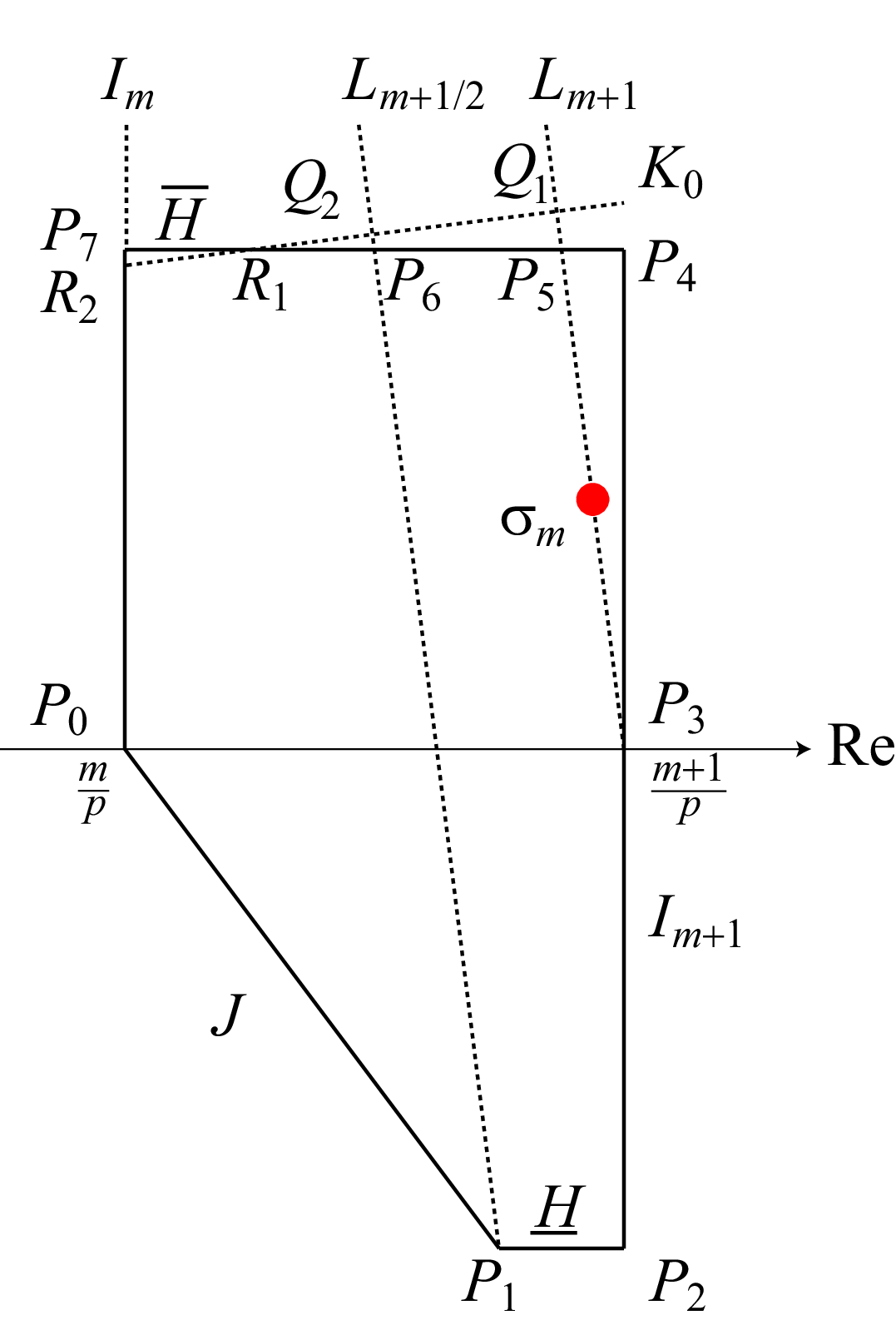}\\(iv)\end{center}
\end{minipage}
\quad
\begin{minipage}{39mm}
\begin{center}\pic{0.21}{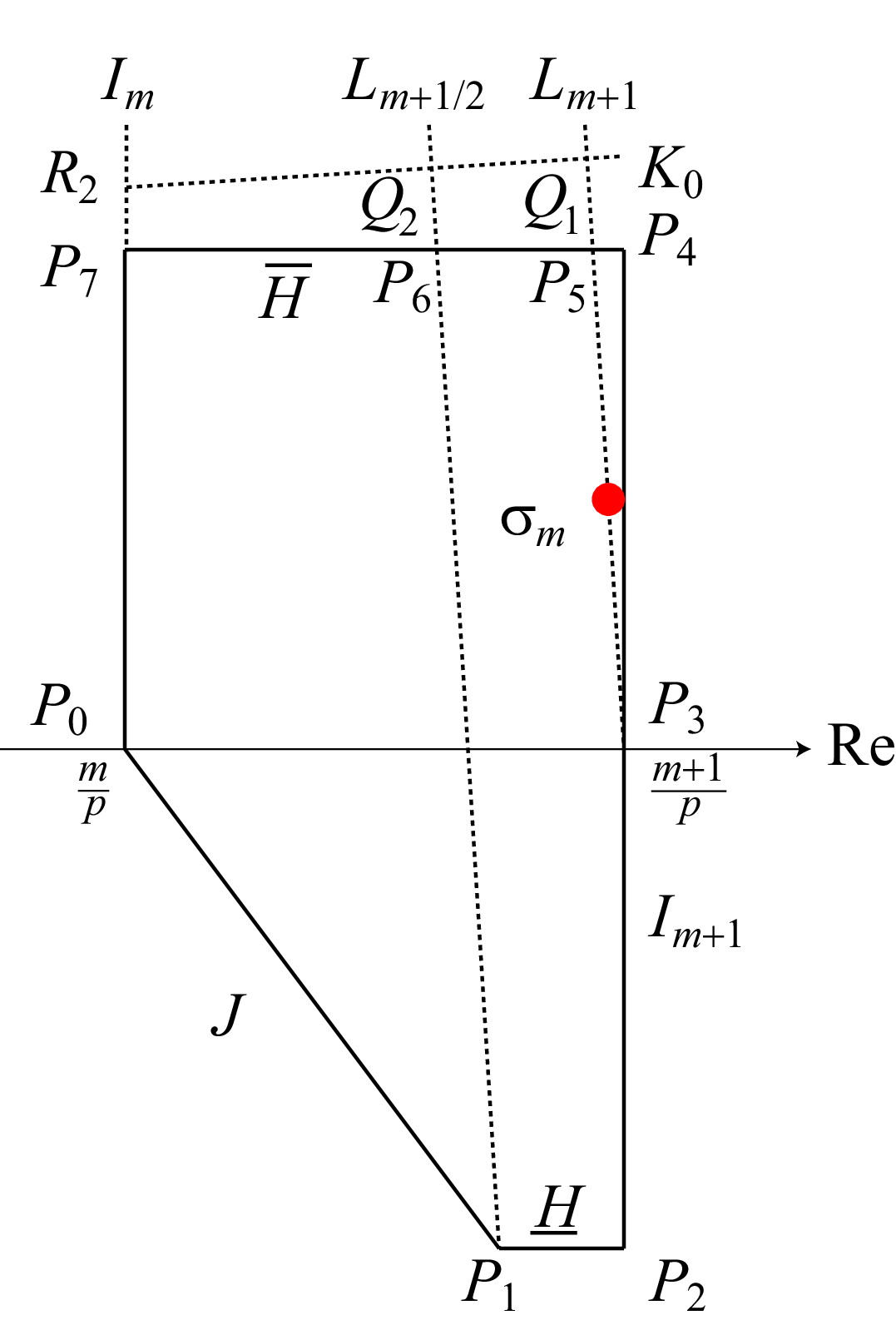}\\(v)\end{center}
\end{minipage}
\caption{The region $\Xi_{m,\mu}$.
The numbers (i)--(v) correspond to the items in Corollary~\ref{cor:QR}.
Note that in (i) and (v), the point $R_1$ is out of range.
Note that if $m=p-1$, we need to remove a small neighborhood $\underline{\nabla}_{p-1,\mu}^{-}$ of $P_3$.}
\label{fig:Xi}
\end{figure}
The five points $P_0,P_1,P_2,P_4,P_7$ form a pentagon, and $P_3,P_5,P_6$ are on its boundary, as shown there because the following lemma holds.
\begin{lem}
The line $K_0$ intersects with $I_m$ above or on the real axis.
Moreover, the lines $L_{m+1}$ and $L_{m+1/2}$ intersect with $\oH$ between $I_m$ and $I_{m+1}$.
\end{lem}
\begin{proof}
Since the imaginary part of the intersection point between $K_0$ and $I_m$ is $\frac{mu}{2p^2\pi}\ge0$, the first assertion follows.
\par
The real part of the point $L_{m+1/2}\cap\oH$ is $\frac{2m+1}{2p}-\frac{u\Im\sigma_m}{p\pi}$.
Now, we have
\begin{equation*}
  \frac{2m+1}{2p}-\frac{u\Im\sigma_m}{p\pi}
  -
  \frac{m}{p}
  =
  \frac{4p^2\pi^2-(4m+3)u^2+4pu\varphi(u)}{2p|\xi|^2},
\end{equation*}
and its numerator is greater than or equal to $4p^2\pi^2+u^2-4pu\bigl(u-\varphi(u)\bigr)$ since $m\le p-1$.
This is greater than $4p^2\pi^2+\kappa^2-4p\kappa^2>0$ since it is increasing with respect to $u$ from Lemma~\ref{lem:phi} (iii).
\par
So we conclude that $L_{m+1/2}$ intersects with $\oH$ at the right of $I_{m}$.
Since the intersection point is clearly to the left of $I_{m+1}$, $L_{m+1/2}$ intersects with $\oH$ between $I_{m}$ and $I_{m+1}$.
\par
Since the line $L_{m+1}$ is to the right of the line $L_{m+1/2}$ and to the left of the point $P_4$, it also intersects with $\oH$ between $I_{m}$ and $I_{m+1}$.
\end{proof}
The points $Q_1$, $Q_2$, $R_1$, and $R_2$ may not be in $\Xi_{m,\mu}$, as shown in Lemmas~\ref{lem:Q} and \ref{lem:R} below.
\begin{lem}\label{lem:Q}
The points $Q_1$ and $Q_2$ are between the lines $I_{m}$ and $I_{m+1}$.
Moreover, there exist $u_1,u_2\in\R$ with $\kappa<u_1<u_2<\upsilon_{p,m}$ such that $Q_1$ {\rm(}$Q_2$, respectively{\rm)} is below $\oH$ if and only if $u<u_1$ {\rm(}$u<u_2$, respectively{\rm)}.
\end{lem}
\begin{proof}
It is clear that both $Q_1$ and $Q_2$ are to the left of $I_{m+1}$.
\par
The coordinate of $Q_2$ is $\frac{(2m+1)\pi}{|\xi|^2}(2p\pi+u\i)$.
Since we assume that $\varphi(u)/u<m/p$, we have $u<\upsilon_{p,m}<\sqrt{p}$ from Lemma~\ref{lem:upsilon}.
So we have
\begin{equation*}
  \frac{2(2m+1)p\pi^2}{|\xi|^2}
  -
  \frac{m}{p}
  =
  \frac{2p^2\pi^2-mu^2}{p|\xi|^2}
  >
  \frac{2p\pi^2-m}{|\xi|^2}
  >0,
\end{equation*}
proving that $Q_2$ is to the right of $I_{m}$.
Since $Q_1$ is to the right of $Q_2$, the point $Q_1$ is also to the right of $I_m$.
\par
The point $Q_2$ is below $\oH$ if and only if
\begin{equation*}
  \frac{4\pi\bigl((m+1)u-p\varphi(u)\bigr)}{|\xi|^2}
  -
  \frac{(2m+1)u\pi}{|\xi|^2}
  =
  \frac{\pi}{|\xi|^2}\bigl((2m+3)u-4p\varphi(u)\bigr)
  >0.
\end{equation*}
Since the function $\varphi(u)/u$ is continuous and strictly increasing for $u>\kappa$ from Lemma~\ref{lem:phi} (iv), there exists a unique real number $u_2$ satisfying $\varphi(u_2)/u_2=\frac{2m+3}{4p}$.
Therefore we conclude that $Q_2$ is below $\oH$ if and only if $u<u_2$.
\par
Since the coordinate of $Q_1$ is $\frac{2(m+1)\pi}{|\xi|^2}(2p\pi+u\i)$, the point $Q_1$ is below $\oH$ if and only if
\begin{equation*}
  \frac{4\pi\bigl((m+1)u-p\varphi(u)\bigr)}{|\xi|^2}
  -
  \frac{2(m+1)u\pi}{|\xi|^2}
  =
  \frac{2\pi}{|\xi|^2}\bigl((m+1)u-2p\varphi(u)\bigr)
  >0.
\end{equation*}
In a similar way as above, we conclude that $Q_1$ is below $\oH$ if and only if $u<u_1$, where $u_1$ is the unique real number such that $\varphi(u_1)/u_1=\frac{m+1}{2p}$.
\par
Since $\varphi(u)/u$ is increasing, we conclude that $u_1<u_2$.
\end{proof}
As for the points $R_1$ and $R_2$, we have the following lemma.
\begin{lem}\label{lem:R}
There exist real numbers $u_L$ and $u_R$ with $\kappa<u_R<u_L<\upsilon_{p,m}$ such that
\begin{itemize}
\item
if $\kappa<u<u_R$, then $R_1$ is to the right of $P_4$, and $R_2$ is below $\oH$,
\item
if $u_R<u<u_L$, then $R_1$ is between $P_7$ and $P_4$, and $R_2$ is below $\oH$,
\item
if $u_L<u<\upsilon_{p,m}$, then $R_1$ is to the left of $P_7$, and $R_2$ is above $\oH$.
\end{itemize}
\end{lem}
\begin{proof}
The real part of the coordinate of $R_1$ is $\frac{4p\pi}{u}\Im\sigma_m$.
So the point $R_1$ is to the left of $P_7$, between $P_4$ and $P_7$, or $P_4$ if and only if $r(u)<m/p$, $m/p<r(u)<(m+1)/p$, or $(m+1)/p<r(u)$, respectively, where we put $r(u):=\frac{4p\pi}{u}\Im\sigma_m$.
Since the function $r(u)=\frac{8p\pi^2\bigl((m+1)u-p\varphi(u)\bigr)}{u(u^2+4p^2\pi^2)}=\frac{8p\pi^2}{u(u^2+4p^2\pi^2)}\Bigl(p\bigl(u-\varphi(u)\bigr)-(p-m-1)u)\Bigr)$ is strictly decreasing from Lemma~\ref{lem:phi} (iii), it is between $r(\upsilon_{p,m})=0$ and $r(\kappa)=\frac{8(m+1)p\pi^2}{\kappa^2+4p^2\pi^2}$.
Since we have
\begin{equation*}
  r(\kappa)-\frac{m+1}{p}
  =
  \frac{m+1}{p(\kappa^2+4p^2\pi^2)}(4p^2\pi^2-\kappa^2)
  >0
\end{equation*}
and $r(u)$ is continuous, there exist $u_R$ and $u_L$ with $\kappa<u_R<u_L<\upsilon_{p,m}$ such that $r(u_L)=m/p$ and $r(u_R)=(m+1)/p$.
Therefore if $u<u_R$, then $R_1$ is to the right of $P_4$, if $u_R<u<u_L$, then $R_1$ is between $P_4$ and $P_7$, and if $u_L<u$, then $R_1$ is to the left of $P_7$.
\par
If $R_1$ is to the left (right, respectively) of $I_m$, then $R_2$ is above (below, respectively) $P_7$.
So the statements about $R_2$ follow.
\end{proof}
Since we have
\begin{equation*}
  r(u_R)-r(u_1)
  =
  \frac{m+1}{p}-\frac{4(m+1)p\pi^2}{|\xi|^2}
  =
  \frac{(m+1)u^2}{p|\xi|^2}
  >0,
\end{equation*}
and
\begin{equation*}
  r(u_L)-r(u_2)
  =
  \frac{m}{p}-\frac{2(2m+1)p\pi^2}{|\xi|^2}
  =
  \frac{1}{p|\xi|^2}(mu^2-2p^2\pi^2)
  <
  \frac{1}{|\xi|^2}(m-2p\pi^2)
  <0,
\end{equation*}
where we use $u<\sqrt{p}$ from Lemma~\ref{lem:upsilon}, we have the inequalities $\kappa<u_R<u_1<u_2<u_L<\upsilon_{p,m}$.
To summarize, we have the following corollary.
\begin{cor}\label{cor:QR}
The points $Q_1$, $Q_2$, $R_1$ and $R_2$ are located as follows.
\quad\par
\begin{enumerate}
\item
If $\kappa<u<u_R$, then $Q_1$, $Q_2$, and $R_2$ are below $\oH$.
Moreover $R_1$ is to the right of $P_4$.
\item
If $u_R<u<u_1$, then $Q_1$, $Q_2$, and $R_2$ are below $\oH$.
Moreover $R_1$ is between $P_5$ and $P_4$.
\item
If $u_1<u<u_2$, then $Q_1$ is above $\oH$, and $Q_2$ and $R_2$ are below $\oH$.
Moreover $R_1$ is between $P_6$ and $P_5$.
\item
If $u_2<u<u_L$, then $Q_1$ and $Q_2$ are above $\oH$, and $R_2$ is below $\oH$.
Moreover $R_1$ is between $P_7$ and $P_6$.
\item
If $u_L<u<\upsilon_{p,m}$, then $Q_1$, $Q_2$, and $R_2$ are above $\oH$.
Moreover $R_1$ is to the left of $P_7$.
\end{enumerate}
See Figure~\ref{fig:Xi}.
\end{cor}
We use \eqref{eq:Gm} to prove the following lemma.
\begin{lem}\label{lem:G_der_y}
Write $z=x+y\i$ with $x,y\in\R$.
If $z$ is between $L_{m}$ and $L_{m+1}$, or between $K_{u}$ and $K_{-u}$, then we have
\begin{itemize}
\item
$\frac{\partial}{\partial\,y}\Re{G_m(x+y\i)}>0$ if and only if $\Im(\gamma z)>0$ and $k-1/2<\Re(\gamma z)<k$ for some integer $k$, or $\Im(\gamma z)<0$ and $l<\Re(\gamma z)<l+1/2$ for some integer $l$,
\item
$\frac{\partial}{\partial\,y}\Re{G_m(x+y\i)}<0$ if and only if $\Im(\gamma z)<0$ and $k-1/2<\Re(\gamma z)<k$ for some integer $k$, or $\Im(\gamma z)>0$ and $l<\Re(\gamma z)<l+1/2$ for some integer $l$.
\end{itemize}
\end{lem}
The following proof is similar to those of \cite[Proposition~1.5]{Murakami:CANJM2023} and \cite[Lemma~4.5]{Murakami:arXiv2023}
\begin{proof}
From \eqref{eq:Gm}, we have
\begin{equation*}
  \frac{\partial\,\Re G_m(z)}{\partial\,y}
  =
  -\arg\Bigl(2\cosh{u}-2\cosh\bigl(\xi z\bigr)\Bigr).
\end{equation*}
Since we have
\begin{equation*}
  \Im\Bigl(2\cosh{u}-2\cosh\bigl(\xi z\bigr)\Bigr)
  =
  -2\sinh\bigl(\Re(\xi z)\bigr)\sin\bigl(\Im(\xi z)\bigr),
\end{equation*}
we see that $\partial\,\Re G_m(z)/\partial\,y$ is positive (negative, respectively) if and only if $\sinh\bigl(\Re(\xi z)\bigr)\sin\bigl(\Im(\xi z)\bigr)$ is positive (negative, respectively).
So we have
\begin{itemize}
\item
$\partial\,\Re G_m(z)/\partial\,y>0$ if and only if $\Re(\xi z)>0$ and $2k\pi<\Im(\xi z)<(2k+1)\pi$ for some integer $k$, or $\Re(\xi z)<0$ and $(2l-1)\pi<\Im(\xi z)<2l\pi$ for some integer $l$, and
\item
$\partial\,\Re G_m(z)/\partial\,y<0$ if and only if $\Re(\xi z)<0$ and $2k\pi<\Im(\xi z)<(2k+1)\pi$ for some integer $k$, or $\Re(\xi z)>0$ and $(2l-1)\pi<\Im(\xi z)<2l\pi$ for some integer $l$.
\end{itemize}
Since $\Re(\gamma z)=\frac{1}{2\pi}\Im(\xi z)$ and $\Im(\gamma z)=\frac{-1}{2\pi}\Re(\xi z)$, we have the conclusion.
\end{proof}
\begin{rem}\label{rem:G_der_y}
Since $-\pi<\arg{w}\le\pi$ for any complex number $w$, we can see that $\partial\Re{G_m(z)}/\partial\,y\le\pi$ when $\partial\Re{G_m(z)}/\partial\,y>0$, and that $\partial\Re{G_m(z)}/\partial\,y>-\pi$ when $\partial\Re{G_m(z)}/\partial\,y<0$.
\end{rem}
From \eqref{eq:K} and \eqref{eq:L}, we have the following corollary.
\begin{cor}\label{cor:G_der_y}
The function $\Re G_m(z)$ is strictly monotonically increasing with resect to $\Im{z}$ if
\begin{itemize}
\item
$z$ is above $K_0$ and between $L_{m+1/2}$ and $L_{m+1}$, or
\item
$z$ is below $K_{0}$ and between $L_{m}$ and $L_{m+1/2}$.
\end{itemize}
It is strictly monotonically decreasing with resect to $\Im{z}$ if
\begin{itemize}
\item
$z$ is above $K_{0}$ and between $L_{m}$ and $L_{m+1/2}$, or
\item
$z$ is below $K_{0}$ and between $L_{m+1/2}$ and $L_{m+1}$.
\end{itemize}
In particular, it is strictly monotonically increasing with respect to $\Im{z}$ in the quadrilateral $P_0P_1Q_2R_2$.
\end{cor}
Now, we want to know how the region $W_m$ looks like.
Recall \eqref{eq:Wm} for the definition of $W_m$.
\par
First of all, from Lemma~\ref{lem:sigma_m+1}, the line segment $Q_1P_3\subset W_m$ except for $\sigma_m$, if $\sigma_m$ is above the real axis.
We can find more line segments in $W_m$.
\begin{lem}\label{lem:R2Q1}
The line segment $\overline{R_2Q_1}$ is in $W_m$.
\end{lem}
\begin{proof}
A point on $K_{0}$ is parametrized as $t/\gamma$ ($t\in\R$) oriented from left to right.
From \eqref{eq:Gm}, we have
\begin{equation*}
\begin{split}
  \frac{d}{d\,t}\Re G_m(t/\gamma)
  &=
  \Re
  \left(
    \frac{2\pi\i}{\xi}
    \log\bigl(2\cosh{u}-2\cosh(2\pi\i t)\bigr)
  \right)
  \\
  &=
  \frac{4p\pi^2}{|\xi|^2}\log\bigl(2\cosh{u}-2\cos(2\pi t)\bigr)
  >0.
\end{split}
\end{equation*}
Therefore we have $\Re G_m(z)<\Re G_m(Q_1)$ if $z$ is on $K_0$ and to the left of $Q_1$.
Now, from Lemma~\ref{lem:sigma_m+1} (ii), we know that $\Re G_m(Q_1)<\Re F(\sigma_0)$, proving the lemma.
\end{proof}
In particular, we have $R_2\in W_m$.
Since $\Re{G_m(x+y\i)}$ is increasing with respect to $y$ on the line segment $\overline{R_2P_0}$ from Corollary~\ref{cor:G_der_y}, we have the following corollary.
\begin{cor}\label{cor:R2P0}
The line segment $\overline{R_2P_0}$ is in $W_m$.
\end{cor}
We can also prove the following lemma.
Since its proof is lengthy and technical, we postpone it to Appendix~\ref{sec:P1}.
\begin{lem}\label{lem:P1}
The point $P_1$ is in $W_m$.
\end{lem}
The following lemma follows from Lemma~\ref{lem:P1}.
\begin{lem}\label{lem:Q2P1}
The line segment $\overline{Q_2P_1}$ is in $W_m$.
\end{lem}
\begin{proof}
In a similar way to the proof of Lemma~\ref{lem:sigma_m+1}, a point on the line segment $\overline{Q_2P_1}$ is parametrized as $\ell(t):=(m+1/2)/p+\overline{\xi}t$ with
\begin{equation*}
  -\frac{(m+1/2)u}{p|\xi|^2}\le t\le\frac{\Im\sigma_m}{p\pi}.
\end{equation*}
Note that $\ell(t)$ is oriented downward.
Since $\overline{Q_2P_1}$ is between $L_m$ and $L_{m+1}$, we can use \eqref{eq:Gm}.
We have
\begin{equation*}
\begin{split}
  \frac{d}{d\,t}\Re G_m\bigl(\ell(t)\bigr)
  =&
  \Re
  \left(
    \overline{\xi}
    \log
    \left(
      2\cosh{u}-2\cosh\left(\frac{(m+1/2)}{p}\xi+|\xi|^2t\right)
    \right)
  \right)
  \\
  =&
  u\log
  \left(
    2\cosh{u}+2\cosh\left(\frac{(m+1/2)u}{p}+|\xi|^2t\right)
  \right)
  \\
  &\text{(since $t\ge-\frac{(m+1/2)u}{p|\xi|^2}$)}
  \\
  >&
  u\log(2\cosh{u}+2)
  >0.
\end{split}
\end{equation*}
\par
So we see that $\frac{d}{d\,t}\Re G_m\bigl(\ell(t)\bigr)$ is increasing.
Since $P_1\in W_m$ from Lemma~\ref{lem:P1}, we conclude that $\overline{Q_2P_1}\subset W_m$.
\end{proof}
Figure~\ref{fig:Xi_W} summarizes  Lemma~\ref{lem:sigma_m+1} (ii), Lemma~\ref{lem:R2Q1}, Corollary~\ref{cor:R2P0}, and Lemma~\ref{lem:Q2P1}.
\begin{figure}[h]
\begin{minipage}{39mm}
\begin{center}\pic{0.21}{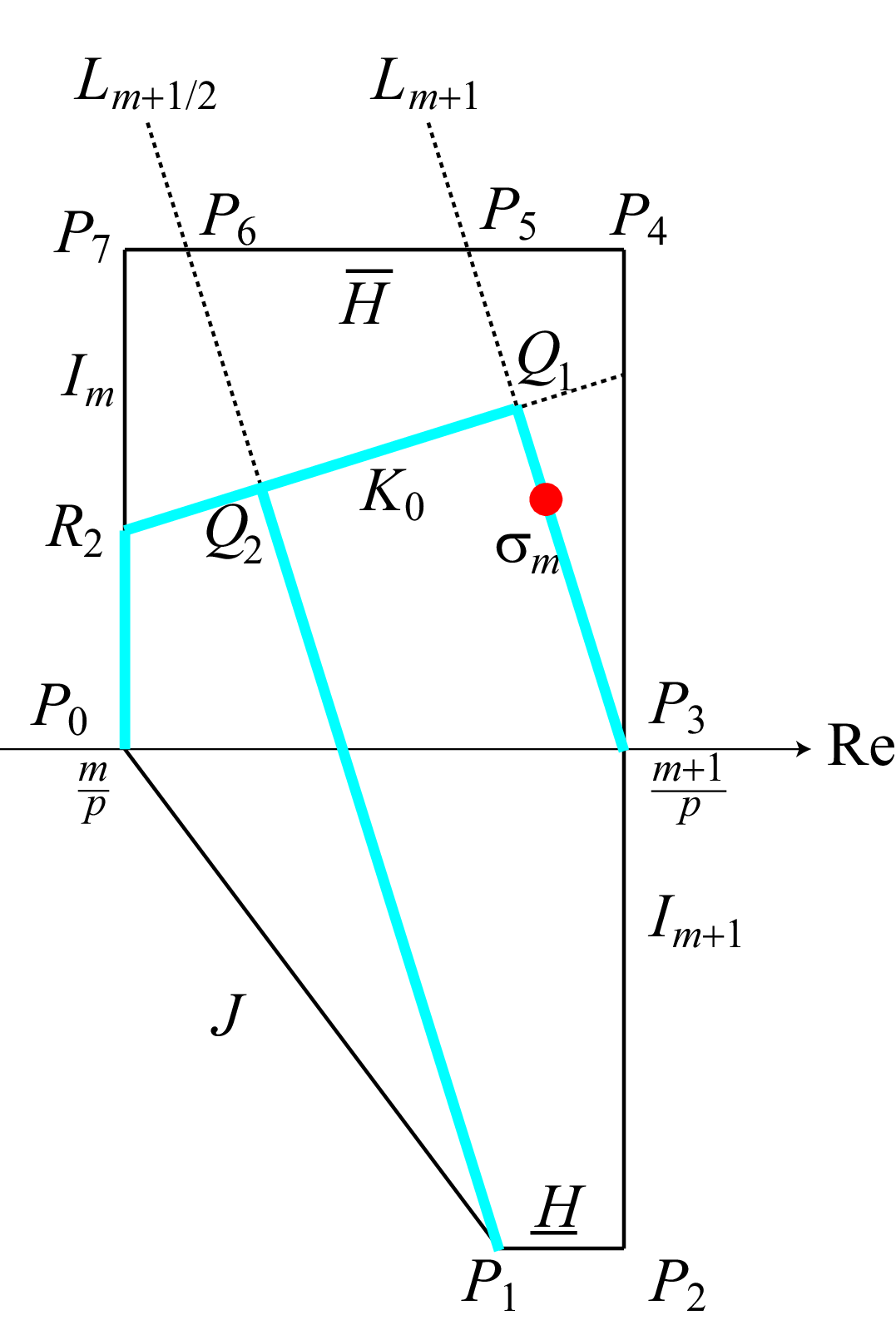}\\(i)\end{center}
\end{minipage}
\quad
\begin{minipage}{39mm}
\begin{center}\pic{0.21}{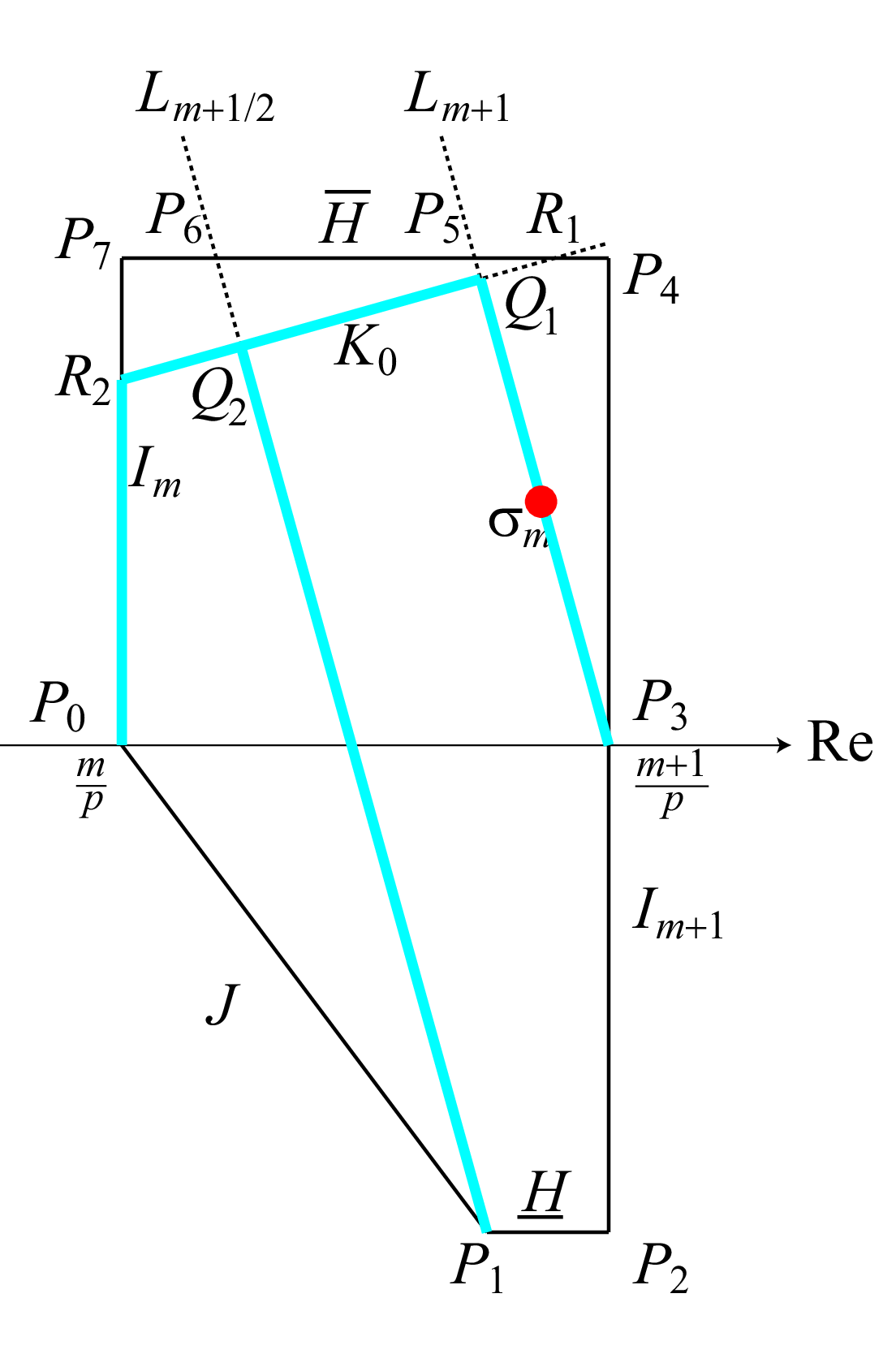}\\(ii)\end{center}
\end{minipage}
\quad
\begin{minipage}{39mm}
\begin{center}\pic{0.21}{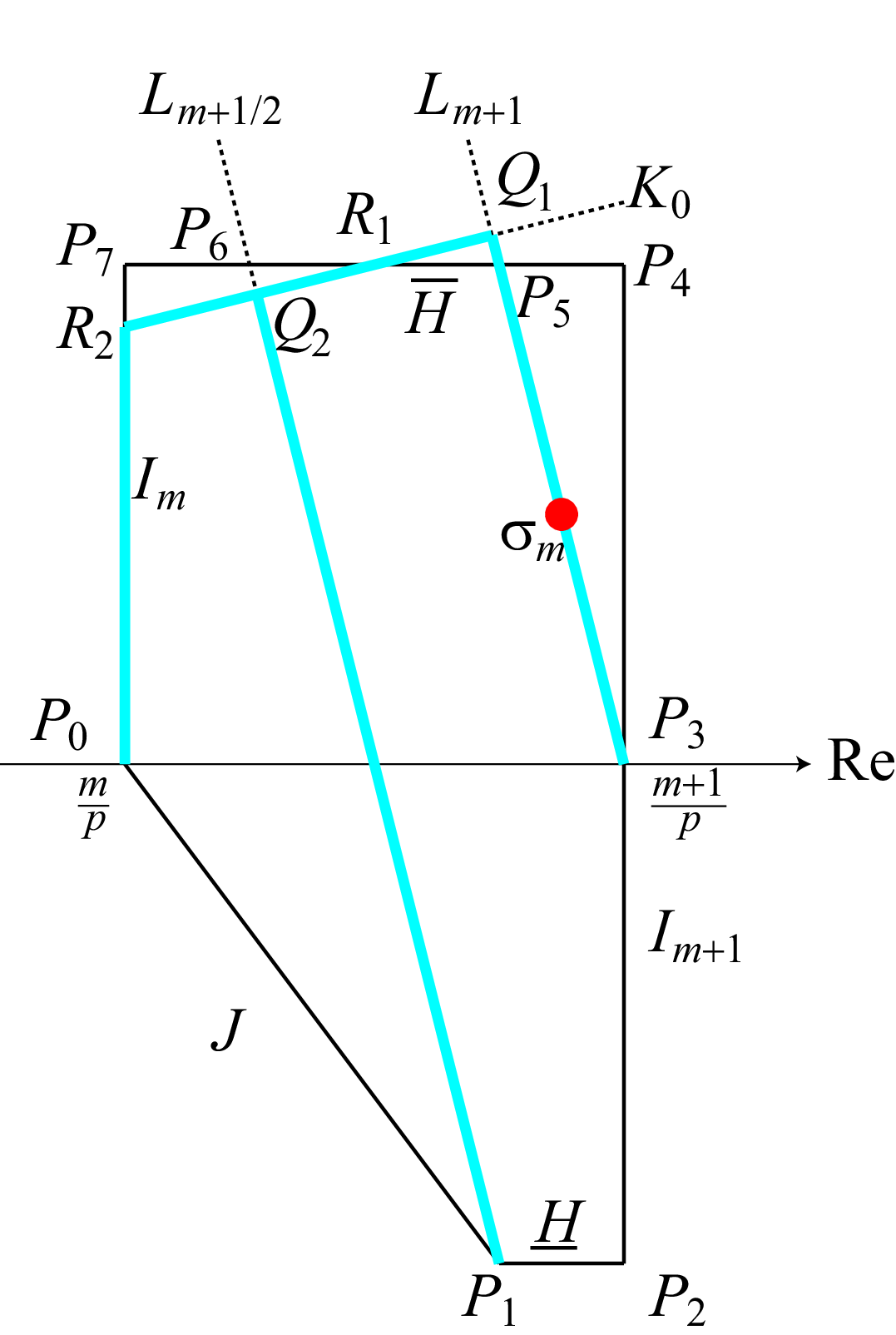}\\(ii)\end{center}
\end{minipage}
\par
\begin{minipage}{39mm}
\begin{center}\pic{0.21}{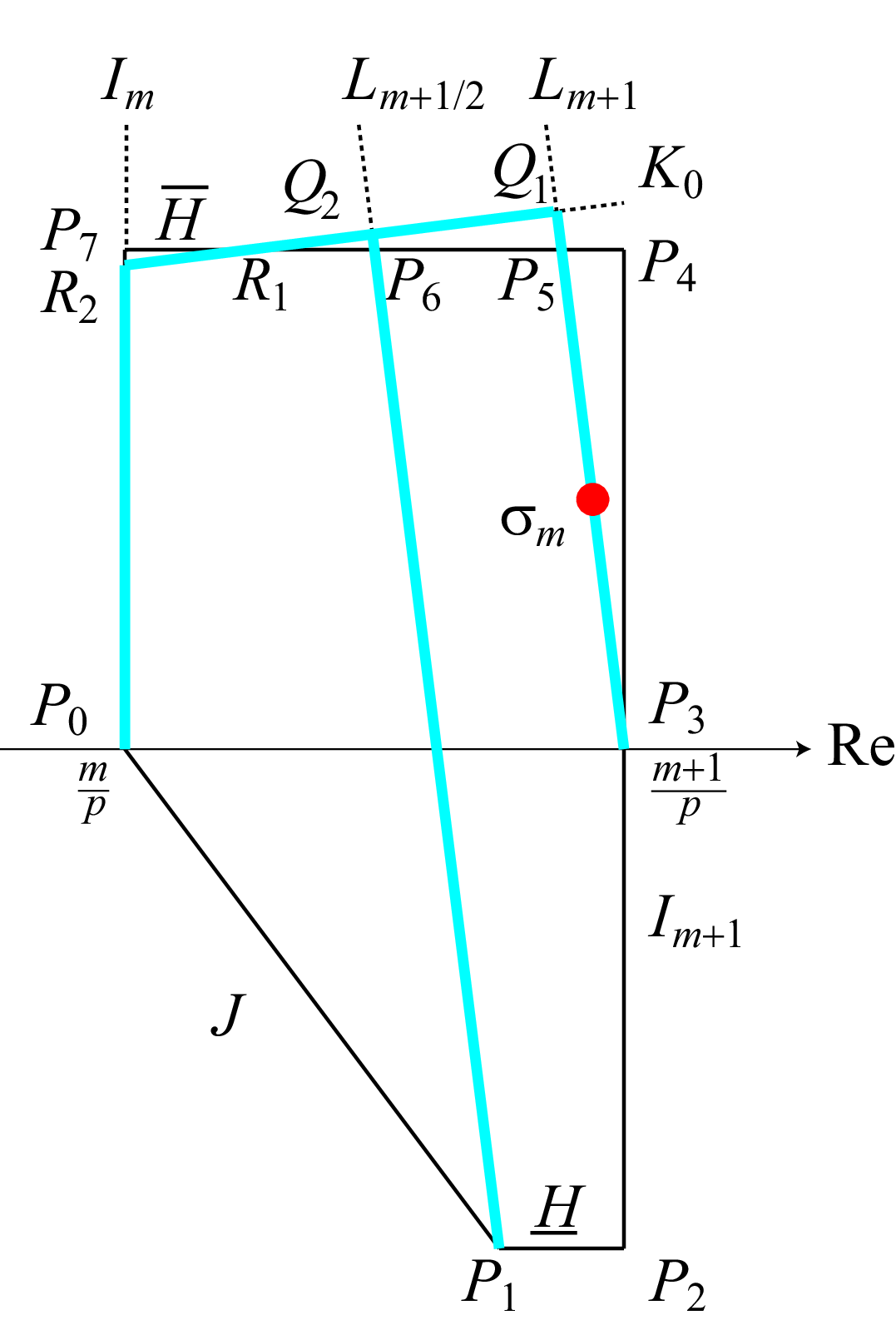}\\(iii)\end{center}
\end{minipage}
\quad
\begin{minipage}{39mm}
\begin{center}\pic{0.21}{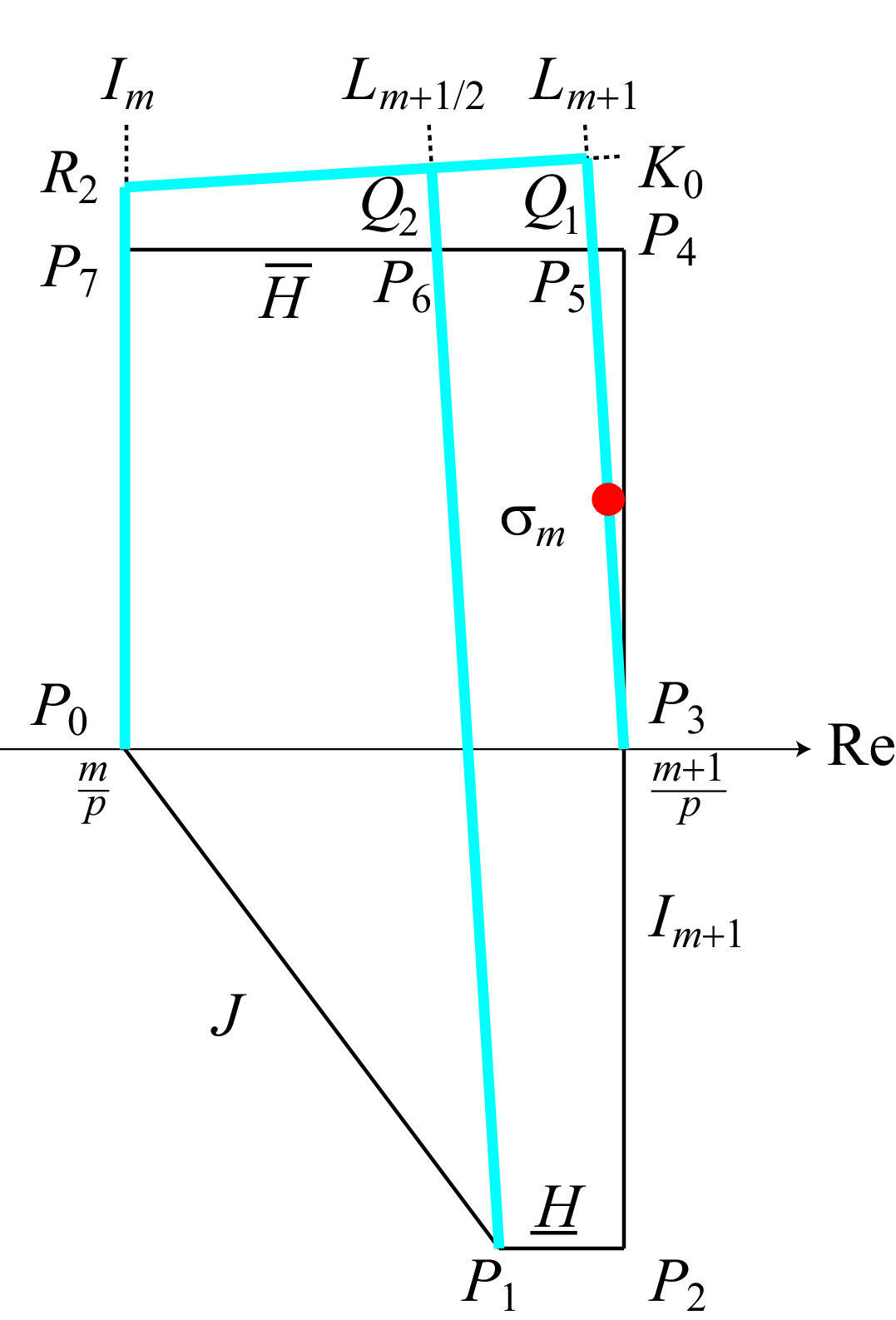}\\(iv)\end{center}
\end{minipage}
\caption{The cyan lines indicate the segments included in $W_m$.}
\label{fig:Xi_W}
\end{figure}
\par
Using Figure~\ref{fig:Xi_W}, we can prove the following proposition.
\begin{prop}\label{prop:P0P3}
For $0\le m<p$, there exists a path connecting $P_0$ and $\sigma_m$ in $W_m\cap\Xi_{m,\mu}$ except for $\sigma_m$.
\par
Moreover, if $m<p-1$, then there exists a path connecting $P_3$ and $\sigma_m$ in $W_m\cap\Xi_{m,\mu}$ except for $\sigma_m$.
If $m=p-1$, then we can choose $\delta>0$ small enough so that there exists a path connecting $1-\delta$ and $\sigma_{p-1}$ in $W_{p-1}\cap\Xi_{p-1,\mu}$ except for $\sigma_{p-1}$.
\end{prop}
\begin{proof}
We will show that there is a path connecting $P_0$ and $\sigma_m$ in $W_m$ except for $\sigma_m$, considering the cases (i)--(v) in Corollary~\ref{cor:QR} separately.
See Figures~\ref{fig:Xi} and \ref{fig:Xi_W}.
\begin{itemize}
\item[(i) and (ii):]
From Figure~\ref{fig:Xi_W} (i) and (ii), we can see that the polygonal chain $\overline{P_0R_2Q_1\sigma_m}$ is in $W_m\cap\Xi_{m,\mu}$ except for $\sigma_m$.
\item[(iii):]
Similarly, from Figure~\ref{fig:Xi_W} (iii), we can see that the polygonal chain $\overline{P_0R_2Q_1\sigma_m}$ is in $W_m$ except for $\sigma_m$.
\par
Since $\Re G_m(x+y\i)$ is increasing with respect to $x$ on the segment $\overline{R_1P_5}$ from Lemma~\ref{lem:G_der_x}, the segment $\overline{R_1P_5}$ is also in $W_m$, proving that $\overline{P_0R_2R_1P_5\sigma_m}\subset W_m\cap\Xi_{m,\mu}$ except for $\sigma_m$.
Precisely speaking, we need to push the segment $\overline{R_1P_5}$ slightly downward into $\Xi_{m,\mu}$.
\item[(iv):]
By the same reason as above, the polygonal chain $\overline{P_0R_2R_1P_5\sigma_m}$ is a desired one.
\item[(v):]
The polygonal chain $P_0P_7P_5\sigma_m$ is a desired one, by the same reason as above.
\end{itemize}
\par
The polygonal chain $\overline{P_3\sigma_m}$ is in $W_m$ except for $\sigma_m$ from Lemma~\ref{lem:sigma_m+1} (ii).
\par
If $m<p-1$, then the line segment $\overline{P_3\sigma_m}$ is in $W_m\cap\Xi_{m,\mu}$ from Figure~\ref{fig:Xi_W}.
\par
If $m=p-1$, we choose $\delta$ small enough so that the line segment connecting $\sigma_m$ and $1-\delta$ is in $W_{p-1}\cap\Xi_{p-1,\mu}$.
\end{proof}
\section{Poisson summation formula and the saddle point method}\label{sec:Poisson}
In this section, we apply the Poisson summation formula to the right hand side of \eqref{eq:sum_f_g}, and replace it with an integral.
Then we apply the saddle point method to obtain an asymptotic formula of the integral.
\par
We can prove the following proposition \cite[Proposition~4.1]{Murakami:arXiv2023} by using the Poisson summation formula, following \cite[Proposition~4.2]{Ohtsuki:QT2016}.
A proof is omitted.
\begin{prop}\label{prop:Poisson}
Let $[a,b]\in\R$ be a closed interval, and $\{\psi_N(z)\}_{N=2,3,4,\dots}$ be a series of holomorphic functions defined in a domain $D\subset\C$ containing $[a,b]$.
We assume the following:
\begin{enumerate}
\item
The series of functions $\{\psi_N(z)\}$ uniformly converges to a holomorphic function $\psi(z)$ in $D$,
\item
$\Re\psi(a)<0$ and $\Re\psi(b)<0$.
\item
There exist paths $C_{+}$ and $C_{-}$ such that $C_{\pm}\subset R_{\pm}$ and that $C_{\pm}$ is homotopic to $[a,b]$ in $D$ with $a$ and $b$ fixed, where we put
\begin{align*}
  R_{+}
  &:=
  \{z\in D\mid\Im{z}\ge0,\Re\psi(z)<2\pi\Im{z}\},
  \\
  R_{-}
  &:=
  \{z\in D\mid\Im{z}\le0,\Re\psi(z)<-2\pi\Im{z}\}.
\end{align*}
\end{enumerate}
Then there exists $\varepsilon>0$ independent of $N$ such that
\begin{equation*}
  \frac{1}{N}\sum_{a\le k/N\le b}e^{N\psi_N(k/N)}
  =
  \int_{a}^{b}e^{N\psi_N(z)}\,dz+O(e^{-\varepsilon N})
\end{equation*}
as $N\to\infty$.
\end{prop}
Now, we would like to apply Proposition~\ref{prop:Poisson} to $\psi_N(z):=g_{N,m}(z)-g_{N,m}(\sigma_m)$.
\par
Defining $R_m^{\pm}$ as
\begin{align*}
  R_{m}^{+}
  &:=
  \{z\in\Xi_{m,\mu}\mid\Im{z}\ge0,\Re G_m(z)<\Re F(\sigma_0)+2\pi\Im{z}\},
  \\
  R_{m}^{-}
  &:=
  \{z\in\Xi_{m,\mu}\mid\Im{z}\le0,\Re G_m(z)<\Re F(\sigma_0)-2\pi\Im{z}\},
\end{align*}
we have the following lemma.
\begin{lem}\label{lem:Poisson}
Assume that $\varphi(u)/u<(m+1)/p$.
\par
If $m<p-1$, then the following hold:
\begin{enumerate}
\item
the points $m/p$ and $(m+1)/p$ are in $W_m$,
\item
there exists a path $C_{+}\subset R_{m}^{+}$ connecting $m/p$ and $(m+1)/p$,
\item
there exists a path $C_{-}\subset R_{m}^{-}$ connecting $m/p$ and $(m+1)/p$.
\end{enumerate}
\par
If $m=p-1$, there exists $\delta>0$ such that the following hold:
\begin{enumerate}
\setcounter{enumi}{3}
\item
the points $1-1/p$ and $1-\delta$ are in $W_{p-1}$,
\item
there exists a path $C_{+}\subset R_{p-1}^{+}$ connecting $1-1/p$ and $1-\delta$,
\item
there exists a path $C_{-}\subset R_{p-1}^{-}$ connecting $1-1/p$ and $1-\delta$.
\end{enumerate}
\end{lem}
\begin{proof}
First, we assume that $m<p-1$.
\begin{enumerate}
\item
From Lemma~\ref{lem:sigma_m+1} (ii), we know that $(m+1)/p\in W_m$.
From Proposition~\ref{prop:P0P3}, $m/p$ is also in $W_m$.
\item
From Proposition~\ref{prop:P0P3}, there exists a path $C_{+}$ in $W_m$ connecting $m/p$ and $(m+1)/p$.
Clearly it is in $R_m^{+}$.
\item
Let $C_{-}$ be the polygonal chain $\overline{P_0P_1P_2P_3}$.
We will show that $C_{-}\subset R_{m}^{-}$.
(Precisely speaking we need to push it slightly inside $\Xi_{m,\mu}$)
\par
We first show that the segments $\overline{P_1P_2}$ and $\overline{P_2P_3}$ are in $R_{m}^{-}$.
\par
Putting $r(x,y):=\Re G_m(x+y\i)-\Re F(\sigma_0)+2\pi y$ for real variables $x$ and $y$, we will prove that $r(x,y)<0$ if $x+y\i$ is on $\overline{P_1P_2}$ or $\overline{P_2P_3}$.
Since these line segments are between $L_{m+1/2}$ and $L_{m+1}$ and below $K_0$, we see that $\partial\,\Re G_m(x+y\i)/\partial\,y<0$ from Corollary~\ref{cor:G_der_y}.
So we have $\pi<\frac{\partial}{\partial\,y}r(x,y)<2\pi$ from Remark~\ref{rem:G_der_y}.
\begin{itemize}
\item
$\overline{P_1P_2}$:
For any point $x-2\Im\sigma_m\i\in\overline{P_1P_2}$, choose $0\le y_0\le2\Im\sigma_m$ so that $(x,y_0)$ is on the polygonal chain $C_{+}$ above.
Since $C_{+}\subset W_m$, we conclude that $r(x,y_0)<2\pi y_0$.
So we have
\begin{equation*}
\begin{split}
  r(x,-2\Im\sigma_m)
  =&
  \int_{y_0}^{-2\Im\sigma_m}\frac{\partial}{\partial\,y}r(x,s)\,ds
  +
  r(x,y_0)
  \\
  <&
  (-2\Im\sigma_m-y_0)\times\pi+2\pi y_0
  \\
  =&
  -2\pi\Im\sigma_m+\pi y_0
  \le
  0,
\end{split}
\end{equation*}
where we use $\partial\,r(x,y)/\partial\,y>\pi$ at the first inequality.
\item
$\overline{P_2P_3}$:
Since $\frac{\partial}{\partial\,y}r\bigl((m+1)/p,y\bigr)>\pi$ and $r\bigl((m+1)/p,0\bigr)<0$ as above, we have
\begin{equation*}
\begin{split}
  &r\bigl((m+1)/p,y\bigr)
  \\
  =&
  \int_{0}^{y}\frac{\partial}{\partial\,y}r\bigl((m+1)/p,s\bigr)\,ds
  +
  r\bigl((m+1)/p,0\bigr)
  <0
\end{split}
\end{equation*}
if $y<0$.
\end{itemize}
Next we show that the line segment $\overline{P_0P_1}$ is in $W_m$.
\par
From Lemmas~\ref{lem:R2Q1} and \ref{lem:Q2P1} (see also Figure~\ref{fig:Xi_W}), we know that the polygonal chain $\overline{R_2Q_2P_1}\subset W_m$.
Since $\Re G_m(x+y\i)$ is monotonically increasing in the quadrilateral $P_0P_1Q_2R_2$ from Corollary~\ref{cor:G_der_y}, we conclude that the segment $P_0P_1$ is in $W_m$.
\end{enumerate}
\par
As for the case $m=p-1$, we just move the line segment $\overline{P_2P_3}$ to the left so that it avoids $\underline{\nabla}_{p-1,\nu}^{-}$.
\end{proof}
Therefore, when $m<p-1$, if we put $\psi_N(z):=g_{N,m}(z)-g_{N,m}(\sigma_m)$, $\psi(z):=G_m(z)-G_m(\sigma_m)=G_m(z)-F(\sigma_0)$, $a:=m/p$, $b:=(m+1)/p$, and $D:=\Xi_{m,\mu}$, then they satisfy the assumptions of Proposition~\ref{prop:Poisson}.
So there exists $\varepsilon>0$ such that
\begin{multline}\label{eq:Poisson}
  \frac{1}{N}e^{-Ng_{N,m}(\sigma_m)}
  \sum_{m/p\le k/N\le(m+1)/p}e^{Ng_{N,m}(k/N)}
  \\
  =
  e^{-Ng_{N,m}(\sigma_m)}
  \int_{m/p}^{(m+1)/p}e^{Ng_{N,m}(z)}\,dz
  +
  O(e^{-\varepsilon N})
\end{multline}
as $N\to\infty$.
\par
Similarly, when $m=p-1$, there exist $\varepsilon>0$ and $\delta>0$ such that
\begin{multline}\label{eq:Poisson_p-1}
  \frac{1}{N}
  e^{-Ng_{N,p-1}(\sigma_{p-1})}
  \sum_{1-1/p\le k/N\le1-\delta}e^{Ng_{N,p-1}(k/N)}
  \\
  =
  e^{-Ng_{N,p-1}(\sigma_{p-1})}
  \int_{1-1/p}^{1-\delta}e^{Ng_{N,p-1}(z)}\,dz
  +
  O(e^{-\varepsilon N})
\end{multline}
as $N\to\infty$.
\par
Next, we use the saddle point method to obtain asymptotic formulas of the integrals appearing in the right hand sides of \eqref{eq:Poisson} and \eqref{eq:Poisson_p-1}.
The following proposition can be obtained in a similar way to \cite[Proposition~5.2]{Murakami:arXiv2023}, and so a proof is omitted.
See also \cite[Proposition~3.2 and Remark~3.3]{Ohtsuki:QT2016}.
\begin{prop}\label{prop:saddle}
Let $\eta(z)$ be a holomorphic function in $D\ni0$ with $\eta(0)=\eta'(0)=0$ and $\eta''(0)\ne0$, and put $V:=\{z\in D\mid\Re\eta(z)<0\}$.
Let $a$ and $b$ be two points in $V$, and $C$ be a path in $D$ from $a$ to $b$.
We also assume
\begin{itemize}
\item
that there exists a neighborhood $\hat{D}\subset D$ of $0$ so that $V\cap\hat{D}$ has two connected components, and
\item
that there exists a path $\hat{C}\subset V\cup\{0\}$ connecting $a$ and $b$ via $0$ that is homotopic to $C$ in $D$ keeping $a$ and $b$ fixed, such that $\bigl(\hat{C}\cap\hat{D}\bigr)\setminus\{0\}$ splits into two connected components that are in distinct components of $V$.
\end{itemize}
\par
Let $\{h_N(z)\}_{N=1,2,3,\dots}$ be a series of holomophic functions in $D$ that uniformly converges to a holomorphic function $h(z)$ with $h(0)\ne0$.
We also assume that $|h_N(z)|$ is bounded irrelevant to $N$.
Then we have
\begin{equation*}
  \int_{C}h_N(z)e^{N\eta(z)}\,dz
  =
  \frac{h(0)\sqrt{2\pi}}{\sqrt{-\eta''(0)}\sqrt{N}}
  \bigl(1+O(N^{-1})\bigr),
\end{equation*}
where the sign of $\sqrt{-\eta''(0)}$ is chosen so that $\Re(b\sqrt{-\eta''(0)})>0$
\end{prop}
When $m<p-1$, we put $\eta(z):=G_m(z+\sigma_m)-G_m(\sigma_m)=G_m(z+\sigma_m)-F(\sigma_0)$, $h_N(z):=e^{N\bigl(g_{N,m}(z+\sigma_m)-G_m(z+\sigma_m)\bigr)}$, $h(z):=1$, $a:=m/p-\sigma_m$, $b:=(m+1)/p-\sigma_m$, $D:=\{z\mid z+\sigma_m\in\Xi_{m,\mu}\}$, and $V:=\{z\mid z+\sigma_m\in W_m\}$.
We also let $C$ be the line segment from $a$ to $b$, $\hat{C}$ be the polygonal chain described in Proposition~\ref{prop:P0P3} (shifted by $-\sigma_m$ so that it passes through $0$), and $\hat{D}\subset D$ be a small neighborhood of $0$.
Then they satisfy the assumptions of Proposition~\ref{prop:saddle} from \eqref{eq:Gmcosh} and Lemma~\ref{lem:Poisson}.
\par
Therefore we have
\begin{multline*}
  \int_{m/p-\sigma_m}^{(m+1)/p-\sigma_m}
  e^{N\bigl(g_{N,m}(z+\sigma_m)-F(\sigma_0)\bigr)}\,dz
  \\
  =
  \frac{\sqrt{2\pi}}{\bigl((2\cosh{u}+1)(2\cosh{u}-3)\bigr)^{1/4}\sqrt{\xi N}}
  \left(1+O(N^{-1})\right)
\end{multline*}
from \eqref{eq:Gmcosh} again.
Here we put $\sqrt{-\eta''(0)}:=\xi^{1/2}\bigl((2\cosh{u}+1)(2\cosh{u}-3)\bigr)^{1/4}$ because
\begin{equation*}
\begin{split}
  \Re\left(b\xi^{1/2}\right)
  =&
  \Re\left(\left(\frac{m+1}{p}-\sigma_m\right)\xi^{1/2}\right)
  \\
  =&
  \Re
  \left(
    \xi^{-1/2}
    \left(
      \frac{m+1}{p}\xi-\varphi(u)-2(m+1)\pi\i
    \right)
  \right)
  \\
  =&
  \left(
    \frac{m+1}{p}u-\varphi(u)
  \right)
  \Re\xi^{-1/2}>0
\end{split}
\end{equation*}
since $\varphi(u)/u<(m+1)/p$.
Putting $w:=z+\sigma_m$, we have
\begin{multline}\label{eq:saddle}
  \int_{m/p}^{(m+p)/p}e^{Ng_{N,m}(w)}\,dw
  \\
  =
  \frac{\sqrt{2\pi}}{\bigl((2\cosh{u}+1)(2\cosh{u}-3)\bigr)^{1/4}\sqrt{\xi N}}
  e^{NF(\sigma_0)}
  \left(1+O(N^{-1})\right)
\end{multline}
when $m<p-1$.
\par
When $m=p-1$, we similarly have
\begin{multline}\label{eq:saddle_p-1}
  \int_{1-1/p}^{1-\delta}e^{Ng_{N,p-1}(w)}\,dw
  \\
  =
  \frac{\sqrt{2\pi}}{\bigl((2\cosh{u}+1)(2\cosh{u}-3)\bigr)^{1/4}\sqrt{\xi N}}
  e^{NF(\sigma_0)}
  \left(1+O(N^{-1})\right).
\end{multline}
\par
From \eqref{eq:Poisson} and \eqref{eq:saddle} we have
\begin{equation*}
\begin{split}
  &\sum_{m/p\le k/N\le(m+1)/p}e^{Ng_{N,m}(k/N)}
  \\
  =&
  N\int_{m/p}^{(m+1)/p}e^{Ng_{N,m}(z)}\,dz
  +
  Ne^{Ng_{N,m}(\sigma_m)}\times O(e^{-\varepsilon N})
  \\
  =&
  \frac{\sqrt{2\pi}}{\bigl((2\cosh{u}+1)(2\cosh{u}-3)\bigr)^{1/4}}\sqrt{\frac{N}{\xi}}
  e^{NF(\sigma_0)}
  \left(1+O(N^{-1})\right)
  \\
  &+
  Ne^{Ng_{N,m}(\sigma_m)}\times O(e^{-\varepsilon N})
  \\
  =&
  \frac{\sqrt{2\pi}}{\bigl((2\cosh{u}+1)(2\cosh{u}-3)\bigr)^{1/4}}\sqrt{\frac{N}{\xi}}
  e^{NF(\sigma_0)}
  \left(1+O(N^{-1/2})\right)
\end{split}
\end{equation*}
for $0\le m<p$.
The last equality follows because
\begin{equation*}
\begin{split}
  &Ne^{Ng_{N,m}(\sigma_m)}
  \\
  =&
  \frac{\sqrt{2\pi}}{\bigl((2\cosh{u}+1)(2\cosh{u}-3)\bigr)^{1/4}}\sqrt{\frac{N}{\xi}}
  e^{NF(\sigma_0)}
  \\
  &\times
  \left(
    \frac{\bigl((2\cosh{u}+1)(2\cosh{u}-3)\bigr)^{1/4}}{\sqrt{2\pi}}\sqrt{\frac{\xi}{N}}
    e^{N\bigl(g_{N,m}(\sigma_m)-F(\sigma_0)\bigr)}
  \right)
\end{split}
\end{equation*}
and the expression in the big parentheses is of order $O(N^{-1/2})$ since the series $\{g_{N,m}(\sigma_m)\}_{N=2,3,\dots}$ converges to $G_m(\sigma_m)=F(\sigma_0)$,.
\par
Similarly, when $m=p-1$, from \eqref{eq:Poisson_p-1} and \eqref{eq:saddle_p-1} we have
\begin{equation*}
\begin{split}
  &\sum_{1-1/p\le k/N\le1-\delta}e^{Ng_{N,p-1}(k/N)}
  \\
  =&
  \frac{\sqrt{2\pi}}{\bigl((2\cosh{u}+1)(2\cosh{u}-3)\bigr)^{1/4}}\sqrt{\frac{N}{\xi}}
  e^{NF(\sigma_0)}
  \left(1+O(N^{-1/2})\right).
\end{split}
\end{equation*}
\par
From Lemma~\ref{lem:g_G_p-1} below, the sum $\sum_{1-\delta<k/N<1}e^{Ng_{N,p-1}(k/N)}$ is of order $O\left(N e^{N\bigl(\Re F(\sigma_0)-\varepsilon\bigr)}\right)$ if $\delta<\tilde{\delta}$.
So we have
\begin{equation}\label{eq:sum_p-1}
\begin{split}
  &\sum_{1-1/p\le k/N<1}e^{Ng_{N,p-1}(k/N)}
  \\
  =&
  \frac{\sqrt{2\pi}}{\bigl((2\cosh{u}+1)(2\cosh{u}-3)\bigr)^{1/4}}\sqrt{\frac{N}{\xi}}
  e^{NF(\sigma_0)}
  \left(1+O(N^{-1/2})\right).
\end{split}
\end{equation}
\par
Recalling that $g_{N,m}(z)=f_N(z-m/\gamma)$, from \eqref{eq:Jones_sum} and Proposition~\ref{prop:sigma_below}, we have
\begin{equation*}
\begin{split}
  &
  \frac{2\sinh(u/2)}{1-e^{2pN\pi\i/\gamma}}J_N\left(\FE;e^{\xi/N}\right)
  \\
  =&
  \left(\sum_{m\le p\varphi(u)/u-1}+\sum_{m>p\varphi(u)/u-1}\right)
  \left(
    \beta_{p,m}
    \sum_{m/p<k/N<(m+1)/p}
    e^{Ng_{N,m}\left(\frac{2k+1}{2N}\right)}
  \right)
  \\
  =&
  \left(\sum_{m>p\varphi(u)/u-1}\beta_{p,m}\right)
  O\left(Ne^{\left(\Re F(\sigma_0)-\varepsilon\right)}\right)
  \\
  &+
  \left(\sum_{m>p\varphi(u)/u-1}\beta_{p,m}\right)
  \frac{\sqrt{2\pi}}{\bigl((2\cosh{u}+1)(2\cosh{u}-3)\bigr)^{1/4}}
  \\
  &\quad\times
  \sqrt{\frac{N}{\xi}}
  e^{NF(\sigma_0)}
  \left(1+O(N^{-1/2})\right),
  \\
  =&
  \left(\sum_{m>p\varphi(u)/u-1}\beta_{p,m}\right)
  \sqrt{\pi}T_{\FE}(\rho_u)^{-1/2}
  \sqrt{\frac{N}{\xi}}
  e^{\frac{N}{\xi}S_{\FE}(u)}
  \left(1+O(N^{-1/2})\right),
\end{split}
\end{equation*}
where $T_{\FE}(\rho_u)$ is defined in \eqref{eq:T_def}, and from \eqref{eq:FKL} $\xi F(\sigma_0)$ coincides with $S_{\FE}(u)$ defined in \eqref{eq:S_def}.
Since we have from \eqref{eq:def_beta}
\begin{equation*}
  \beta_{p,m}
  =
  e^{-4mpN\pi^2/\xi}
  \prod_{j=1}^{m}
  \left(1-e^{4(p-j)N\pi^2/\xi}\right)
  \left(1-e^{4(p+j)N\pi^2/\xi}\right),
\end{equation*}
we see that
\begin{equation*}
   \sum_{m=0}^{p-1}\beta_{p,m}
   =
   J_p\left(\FE;e^{4N\pi^2/\xi}\right)
\end{equation*}
from \eqref{eq:fig8}.
Since $\left(1-e^{4(p-j)N\pi^2/\xi}\right)\left(1-e^{4(p+j)N\pi^2/\xi}\right)\underset{N\to\infty}{\sim}e^{8pN\pi^2/\xi}$, we have $\beta_{p,m}\underset{N\to\infty}{\sim}e^{4mpN\pi^2/\xi}$.
Therefore, for any $0\le m_0\le p-1$ the sum $\sum_{m=m_0}^{p-1}\beta_{p,m}$ is dominated by the term $\beta_{p,p-1}\underset{N\to\infty}{\sim}e^{4p(p-1)N\pi^2/\xi}$, and we conclude that
\begin{equation*}
   \sum_{m>p\varphi(u)/u-1}\beta_{p,m}
   \underset{N\to\infty}{\sim}
   e^{4p(p-1)N\pi^2/\xi}
   \underset{N\to\infty}{\sim}
   J_p\left(\FE;e^{4N\pi^2/\xi}\right)
\end{equation*}
because $p-1>p\varphi(u)/u-1$.
\par
Therefore we finally have
\begin{equation*}
\begin{split}
  &J_N(\FE;e^{\xi/N})
  \\
  =&
  J_p\left(\FE;e^{4N\pi^2/\xi}\right)
  \frac{1-e^{2pN\pi\i/\gamma}}{2\sinh(u/2)}
  \sqrt{\pi}T_{\FE}(\rho_u)^{-1/2}
  \sqrt{\frac{N}{\xi}}
  e^{\frac{N}{\xi}S_{\FE}(u)}
  \left(1+O(N^{-1/2})\right).
\end{split}
\end{equation*}
Now, \eqref{eq:main} in Theorem~\ref{thm:main} follows since $e^{2pN\pi\i/\gamma}=e^{-4p\pi^2N/\xi}\to0$ as $N\to\infty$.
\par
The following lemma was used to prove \eqref{eq:sum_p-1}.
\begin{lem}\label{lem:g_G_p-1}
For any $\varepsilon>0$, there exists $\tilde{\delta}>0$ such that
\begin{equation*}
  \Re g_{N,p-1}\left(\frac{2k+1}{2N}\right)
  <
  \Re F(\sigma_0)-\varepsilon
\end{equation*}
for sufficiently large $N$, if $1-1/p<1-\tilde{\delta}<k/N<1$.
\end{lem}
\begin{proof}
The proof is similar to that of \cite[Lemma~6.1]{Murakami:CANJM2023}.
\par
Recall that $F(\sigma_0)=G_{p-1}(\sigma_{p-1})$ and that $\sigma_{p-1}$ is always above the real axis from Remark~\ref{rem:sigma}.
\par
Since $\Re G_{p-1}(1)<\Re G_{p-1}(\sigma_{p-1})$ from Lemma~\ref{lem:sigma_m+1}, there exists $\delta_1>0$ such that $\Re G_{p-1}\left(\frac{2k+1}{2N}\right)<\Re G_{p-1}(\sigma_{p-1})-2\varepsilon$ if $1-1/p\le1-\delta_1<\frac{2k+1}{2N}<1$.
We assume that $\nu$ is small enough so that $\min\{\nu/p,2\nu\pi/u\}<\delta_1<(2-\nu)/p$.
\par
Since $\underline{\nabla}^{-}_{p-1,\nu}\cap\R=\bigl\{x\in\R\mid 1-\min\{\nu/p,2\nu\pi/u\}<x\le1+(1-\nu)/p\bigr\}$ from \eqref{eq:nabla}, the series of functions $\{g_{N,p-1}(x)\}_{N=2,3,\dots}$ converges to $G_{p-1}(x)$ if $1-(2-\nu)/p\le x\le1-\min\{\nu/p,2\nu\pi/u\}$ from \eqref{eq:Theta}.
Therefore if $N$ is sufficiently large and $1-(2-\nu)/p\le\frac{2k+1}{2N}\le1-\min\{\nu/p,2\nu\pi/u\}$, then $\left|\Re g_{N,p-1}\left(\frac{2k+1}{2N}\right)-\Re G_{p-1}\left(\frac{2k+1}{2N}\right)\right|<\varepsilon$.
Since $\min\{\nu/p,2\nu\pi/u\}<\delta_1<(2-\nu)/p$, we have
\begin{equation}\label{eq:delta}
\begin{split}
  \Re g_{N,p-1}\left(\frac{2k+1}{2N}\right)
  <&
  \Re G_{p-1}\left(\frac{2k+1}{2N}\right)+\varepsilon
  \\
  <&
  \Re G_{p-1}(\sigma_{p-1})-\varepsilon
\end{split}
\end{equation}
if $1-\delta_1<\frac{2k+1}{2N}<1-\min\{\nu/p,2\nu\pi/u\}$.
\par
Since the function $\sinh\left(\frac{\xi}{2}(1-x)\right)\sinh\left(\frac{\xi}{2}(1+x)\right)$ of $x\in\R$ vanishes when $x=1$, we see that there exists $0<\delta_2<1/p$ such that its absolute value is less than $1/4$ if $1-\delta_2<x<1$.
This means that if we define $h_N(k):=\prod_{l=1}^{k}\left(4\sinh\left(\frac{\xi}{2}(1-l/N)\right)\sinh\left(\frac{\xi}{2}(1+l/N)\right)\right)$, then $\left|h_N(k)\right|$ is decreasing with respect to $k$ when $1-\delta_2<k/N<1$.
From \cite[(8.3)]{Murakami:CANJM2023}, we have
\begin{equation*}
  h_N(k)
  =
  \frac{1-e^{-4pN\pi^2/\xi}}{2\sinh(u/2)}
  \beta_{p,p-1}
  \exp\left(Ng_{N,p-1}\left(\frac{2k+1}{2N}\right)\right)
\end{equation*}
for $1-1/p<k/N<1$.
It follows that
\begin{equation*}
  \Re g_{N,p-1}\left(\frac{2k+1}{2N}\right)
  =
  \frac{1}{N}
  \log
  \left|
    \frac{2\sinh(u/2)h_N(k)}{\left(1-e^{-4pN\pi^2/\xi}\right)\beta_{p,p-1}}
  \right|
\end{equation*}
is also decreasing with respect to $k$ whenever $1-\delta_2<k/N<1$.
\par
If we put $\tilde{\delta}:=\min\{\delta_1,\delta_2\}$, then for $k$ with $1-\tilde{\delta}<k/N<1$, we can choose $k'<k$ such that $1-\delta_1<\frac{2k'+1}{2N}<1-\min\{\nu/p,2\nu\pi/u\}$ if $N$ is sufficiently large.
So from \eqref{eq:delta} we finally conclude that
\begin{equation*}
  \Re g_{N,p-1}\left(\frac{2k+1}{2N}\right)
  <
  \Re g_{N,p-1}\left(\frac{2k'+1}{2N}\right)
  <
  \Re G_{p-1}(\sigma_{p-1})-\varepsilon,
\end{equation*}
where the first inequality follows since $\Re g_{N,p-1}\left(\frac{2k+1}{N}\right)$ is decreasing.
\end{proof}
\section{The Chern--Simons invariant and the Reidemeister torsion}\label{sec:CS}
In this section, we study $\SL(2;\R)$ representations of the fundamental group of $S^3\setminus\FE$, and the associated Chern--Simons invariants and the Reidemeister torsions.
We also describe their correspondence to the asymptotic behavior of $J_N\left(\FE;e^{(u+2p\pi\i)/N}\right)$ as $N\to\infty$ with $u>\kappa$.
\subsection{Representation}\label{subsec:representation}
Let $x$ and $y$ be elements in $\pi_1(S^3\setminus\FE)$ indicated in Figure~\ref{fig:fig8}.
\begin{figure}[h]
\pic{0.3}{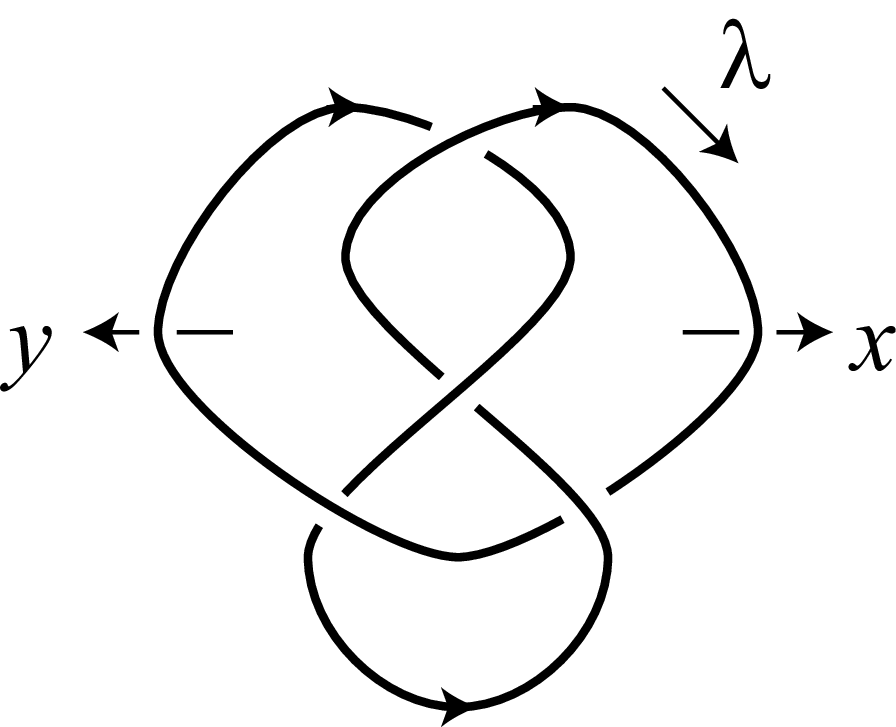}
\caption{Generators of $\pi_1(S^3\setminus\FE)$}
\label{fig:fig8}
\end{figure}
For $u>\kappa=\arccosh(3/2)$, let $\rho_u\colon\pi_1(S^3\setminus\FE)\to\SL(2;\R)$ be the representation defined as follows:
\begin{equation*}
  \rho_u(x)
  :=
  \begin{pmatrix}e^{u/2}&1\\0&e^{-u/2}\end{pmatrix},
  \quad
  \rho_u(y)
  :=
  \begin{pmatrix}e^{u/2}&0\\d(u)&e^{-u/2}\end{pmatrix},
\end{equation*}
where
\begin{equation*}
  d(u)
  :=
  \frac{3}{2}-\cosh{u}-\frac{1}{2}\sqrt{(2\cosh{u}-3)(2\cosh{u}+1)}.
\end{equation*}
Since $d(u)$ is real, $\rho_u$ is indeed an $\SL(2;\R)$ representation.
\begin{rem}
In \cite[(5.5)]{Murakami/Yokota:2018}, we define $d_{+}(u)$ as
\begin{equation*}
  d_{+}(u)
  :=
  \frac{3}{2}-\cosh{u}+\frac{1}{2}\sqrt{(2\cosh{u}-3)(2\cosh{u}+1)},
\end{equation*}
and use it instead of $d(u)$ to define $\rho_u$ for $0<u<\kappa$.
Here the branch of the square root in the definition of $d_{+}(u)$ is chosen so that $\sqrt{x}:=\i\sqrt{-x}$ for a negative real number $x$.
Therefore, precisely speaking, $d_{+}(u)$ should be defined as
\begin{equation*}
  \tilde{d}_{+}(u)
  :=
  \frac{3}{2}-\cosh{u}+\frac{1}{2}\i\sqrt{(3-2\cosh{u})(2\cosh{u}+1)}
\end{equation*}
to make it analytic around $0$.
Note that $\tilde{d}_{+}(u)$ has a branch cut along $(-\infty,-\kappa]\cup[\kappa,\infty)$.
When $u>\kappa$, if we choose the branch so that $\sqrt{(3-2\cosh{u})(2\cosh{u}+1)}$ has the positive imaginary part, then $\tilde{d}_{+}(u)$ coincides with $d(u)$.
\end{rem}
\par
The preferred longitude $\lambda$ also indicated in Figure~\ref{fig:fig8} is mapped to
\begin{equation*}
  \rho_u(\lambda)
  =
  \begin{pmatrix}-e^{v(u)/2}&\ast\\0&-e^{-v(u)/2}\end{pmatrix}
\end{equation*}
with
\begin{equation*}
\begin{split}
  &v(u)
  \\
  :=&
  2\log\left(\cosh(2u)-\cosh{u}-1+\sinh{u}\sqrt{(2\cosh{u}-3)(2\cosh{u}+1)}\right)+2\pi\i
  \\
  =&
  2\log\left(2\cosh\bigl(u+\varphi(u)\bigr)-2\right)+2\pi\i.
\end{split}
\end{equation*}
Since the argument of $\log$ is increasing and equals $1$ when $u=\kappa$, we see that $\Im{v(u)}=2\pi$ and $\Re{v(u)}>0$.
Since $S_{\FE}(u)=L_p(u)-4p\pi^2+2u\pi\i$ from \eqref{eq:S_def} and \eqref{eq:Lp}, we have from \eqref{eq:Lp_der}
\begin{equation*}
  \frac{d}{d\,u}S_{\FE}(u)
  =
  2\cosh\bigl(u+\varphi(u)\bigr)-2+2\pi\i.
\end{equation*}
So we conclude that
\begin{equation}\label{eq:v_S}
  v(u)
  =
  2\frac{d}{d\,u}S_{\FE}(u)-2\pi\i,
\end{equation}
which proves \eqref{eq:v_main} in Theorem~\ref{thm:main}.
\par
In \cite{Hodgson/PhD}, C.~Hodgson studies Dehn surgery space of the figure-eight knot.
He uses the parameter
\begin{equation*}
  \alpha
  :=
  \Tr(xy^{-1})
  =
  2-d(u)
  =
  2\cosh{u}-e^{\varphi(u)}
  =
  e^{u}-e^{\varphi(u)}+e^{-u}.
\end{equation*}
Since $e^{u}-e^{\varphi(u)}$ is monotonically decreasing for $u>\kappa$, and approaches to $1$ as $u\to\infty$ from Lemma~\ref{lem:phi} (vi) and (vii), we have $1<\alpha<2$.
So our case corresponds to the line segment connecting $\alpha=1$ and $\alpha=2$ in the figure in \cite[Page~148]{Hodgson/PhD}.
See also \cite[P.~48]{Thurston:GT3M} and \cite[P.~103]{Cooper/Hodgson/Kerckhoff:2000}.
\subsection{Chern--Simons invariant}
For a knot $K\subset S^3$, put $M:=S^3\setminus{N(K)}$, where $N(K)$ is the open tubular neighborhood of $K$.
Note that the boundary $\partial{M}$ of $M$ is a torus.
Then, following \cite[P.~543]{Kirk/Klassen:COMMP1993}, the $\mathrm{PSL}(2;\C)$ Chern--Simons function $\cs_{K}$ is defined as a map from the $\mathrm{PSL}(2;\C)$ character variety $X(M)$ of $M$ to $E(M):=\left[\Hom\bigl(\pi_1(\partial{M}),\C\bigr)\times\C^{\ast}\right]/\sim$, where $\sim$ is the equivalence relation generated by
\begin{equation*}
\begin{split}
  (\alpha,\beta;z)
  \sim
  (\alpha+1/2,\beta;ze^{-4\pi\i\beta})
  \sim
  (\alpha,\beta+1/2;ze^{4\pi\i\alpha})
  \sim
  (-\alpha,-\beta;z).
\end{split}
\end{equation*}
Here we fix a basis $(\mu^{\ast},\lambda^{\ast})$ of $\Hom\bigl(\pi_1(\partial{M}),\C\bigr)\cong\C^2$ with $(\mu,\lambda)$ the standard meridian/longitude pair.
We denote by $[\alpha,\beta;z]$ the equivalence class of $(\alpha,\beta;z)\in\Hom\bigl(\pi_1(\partial{M}),\C\bigr)\times\C^{\ast}$.
Let $\rho\colon\pi_1(M)\to\SL(2;\C)$ be a representation sending the meridian $\mu$ to $\begin{pmatrix}e^{m/2}&\ast\\0&e^{-m/2}\end{pmatrix}$ and the preferred longitude $\lambda$ to $\begin{pmatrix}e^{l/2}&\ast\\0&e^{-l/2}\end{pmatrix}$.
Then, the Chern--Simons invariant $\CS_{K}(\rho;m,l)\in\C\pmod{\pi^2\Z}$ is defined by the following equation:
\begin{equation*}
  \cs_{K}([\rho])
  =
  \left[
    \frac{m}{4\pi\i},\frac{l}{4\pi\i};\exp\left(\frac{2}{\pi\i}\CS_{K}(\rho;m,l)\right)
  \right],
\end{equation*}
where $[\rho]\in X(M)$ denotes the equivalence class.
\par
In \cite[Theorem~1.5]{Murakami/Tran:Takata}, we proved that $\CS_{\FE}\bigl(\rho_{u};u,\eta\bigr)=\check{S}_{\FE}(u)-u\eta/4$, where $\check{S}_{\FE}(u):=S_{\FE}(u)-2\pi\i u$ and $\eta=v(u)-2\pi\i$.
Since the Chern--Simons function $\cs_{\FE}([\rho_u])$ is given by
\begin{equation*}
\begin{split}
  &
  \left[
    \frac{u}{4\pi\i},
    \frac{\eta}{4\pi\i};
    \exp\left(\frac{2}{\pi\i}\CS_{\FE}(\rho_u;u,\eta)\right)
  \right]
  \\
  =&
  \left[
    \frac{u}{4\pi\i},
    \frac{v(u)-2\pi\i}{4\pi\i};
    \exp\left(\frac{2}{\pi\i}\CS_{\FE}(\rho_u;u,\eta)\right)
  \right]
  \\
  =&
  \left[
    \frac{u}{4\pi\i},
    \frac{v(u)}{4\pi\i};
    \exp\left(\frac{2}{\pi\i}\CS_{\FE}(\rho_u;u,\eta)\right)
    \times
    \exp\left(4\pi\i\times\frac{u}{4\pi\i}\right)
  \right]
  \\
  =&
  \left[
    \frac{u}{4\pi\i},
    \frac{v(u)}{4\pi\i};
    \exp
    \left(
      \frac{2}{\pi\i}
      \left(
        \CS_{\FE}(\rho_u;u,\eta)
        +
        \frac{1}{2}\pi\i u
      \right)
    \right)
  \right],
\end{split}
\end{equation*}
we have
\begin{equation}\label{eq:CS_S}
\begin{split}
  &\CS_{\FE}\bigl(\rho_u;u,v(u)\bigr)
  \\
  =&
  \CS_{\FE}(\rho_u;u,\eta)+\frac{1}{2}\pi\i u
  \\
  =&
  S_{\FE}(u)-2\pi\i u-\frac{1}{4}u\bigl(v(u)-2\pi\i\bigr)+\frac{1}{2}\pi\i u
  \\
  =&
  S_{\FE}(u)-\pi\i u-\frac{1}{4}uv(u).
\end{split}
\end{equation}
This gives a topological interpretation of the term $S_{\FE}(u)$ in \eqref{eq:main}.
\begin{rem}
If we use $\check{v}(u):=2\dfrac{d}{d\,u}\check{S}_{\FE}(u)-2\pi\i=\eta-2\pi\i$ instead of $\eta$  following \eqref{eq:v_S}, we have
\begin{equation*}
\begin{split}
  &\left[
    \frac{u}{4\pi\i},\frac{\eta}{4\pi\i};
    \exp\left(\frac{2}{\pi\i}\CS_{\FE}(\rho_u;u,\eta)\right)
  \right]
  \\
  =&
  \left[
    \frac{u}{4\pi\i},\frac{\check{v}(u)+2\pi\i}{4\pi\i};
    \exp
    \left(
      \frac{2}{\pi\i}\left(\check{S}_{\FE}(u)-\frac{u\bigl(\check{v}(u)+2\pi\i\bigr)}{4}\right)
    \right)
  \right]
  \\
  =&
  \left[
    \frac{u}{4\pi\i},\frac{\check{v}(u)}{4\pi\i};
  \right.
  \\
  &\quad
  \left.
    \exp
    \left(
      \frac{2}{\pi\i}\left(\check{S}_{\FE}(u)-\frac{u\bigl(\check{v}(u)+2\pi\i\bigr)}{4}\right)
      \times\exp\left(-4\pi\i\times\frac{u}{4\pi\i}\right)
    \right)
  \right]
  \\
  =&
  \left[
    \frac{u}{4\pi\i},\frac{\check{v}(u)}{4\pi\i};
    \exp
    \left(
      \frac{2}{\pi\i}\left(\check{S}_{\FE}(u)-\pi\i u-\frac{u\check{v}(u)}{4}\right)
    \right)
  \right].
\end{split}
\end{equation*}
Therefore we have
\begin{equation*}
  \CS_{\FE}(\rho_u;u,\check{v}(u))
  =
  \check{S}_{\FE}(u)-\pi\i u-\frac{1}{4}u\check{v}(u),
\end{equation*}
which is consistent with \eqref{eq:CS_S}.
\end{rem}
\subsection{Reidemeister torsion}
Put $M:=S^3\setminus{N(K)}$ as in the previous subsection.
Given an $\SL(2;\C)$ representation $\rho$ of $\Pi:=\pi_1(M)$, one can construct a chain complex $C_{\ast}:=\{C_i,\partial_i\}_{i=0,1,2}$, where $C_i:=C_{i}(\tilde{M})\otimes_{\Z[\Pi]}\mathfrak{sl}_2(\C)$ with $\tilde{M}$ the universal cover of $M$, and $\Pi$ acts on $\tilde{M}$ as the deck transformation and on the Lie algebra $\mathfrak{sl}_2(\C)$ by the adjoint action of $\rho$.
Note that we regard $\tilde{M}$ as a two-dimensional complex up to homptopy equivalence.
Let
\begin{itemize}
\item
$\mathbf{c}_{i}=\{c_{i,1},\dots,c_{i,l_i}\}$ be a basis of $C_{i}$,
\item
$\mathbf{b}_{i}=\{b_{i,1},\dots,b_{i,m_i}\}$ be a set of vectors such that $\{\partial_{i}(b_{i,j})\}_{j=1,\dots,m_{i}}$ forms a basis of $\Im\partial_{i}$,
\item
$\mathbf{h}_{i}=\{h_{i,1},\dots,h_{i,n_i}\}$ be a basis of the $i$-th homology group $H_i(C_{\ast})$,
\item
$\tilde{\mathbf{h}}_{i}=\{\tilde{h}_{i,1},\dots,\tilde{h}_{i,n_i}\}$ be lifts of $h_{i,k}$ in $\Ker\partial_i$.
\end{itemize}
One can see that $\partial_{i+1}(\mathbf{b}_{i+1})\cup\tilde{\mathbf{h}}_{i}\cup\mathbf{b}_{i}$ forms a basis of $C_{i}$.
Then the torsion $\Tor(C_{\ast},\mathbf{c}_{\ast},\mathbf{h}_{\ast})$ is defined as
\begin{equation*}
  \Tor(C_{\ast},\mathbf{c}_{\ast},\mathbf{h}_{\ast})
  :=
  \prod_{i=0}^{2}
  \bigl[
    \partial_{i+1}(\mathbf{b}_{i+1})\cup\tilde{\mathbf{h}}_{i}\cup\mathbf{b}_{i}
    \bigm|
    \mathbf{c}_{i}
  \bigr]^{(-1)^{i+1}},
\end{equation*}
where $[\mathbf{x}\mid\mathbf{y}]$ is the determinant of the base change matrix from $\mathbf{x}$ to $\mathbf{y}$.
Note that this depends only on $\mathbf{c}_{i}$ and $\mathbf{h}_i$, not on $\mathbf{b}_{i}$ or $\tilde{\mathbf{h}}_{i}$.
\par
We choose the so-called geometric basis for $\mathbf{c}_{i}$ \cite[D{\'e}finition~0.4]{Porti:MAMCAU1997}.
Letting $\mu$ be the meridian in $\partial{M}$, if the representation $\rho$ is irreducible and $\mu$-regular in the sense of \cite[D{\'e}finition~3.21]{Porti:MAMCAU1997}, we can choose bases of $\mathbf{h}_1$ and $\mathbf{h}_2$ by using $\mu$ because $\dim H_1(C_{\ast})=\dim H_2(C_{\ast})=1$.
Note that $H_0(C_{\ast})$ vanishes if $\rho$ is irreducible.
Now the (homological) adjoint Reidemeister torsion $T_{K}(\rho)$ of $\rho$ associated with the meridian is defined as $\Tor(C_{\ast},\mathbf{c}_{\ast},\mathbf{h}_{\ast})$, where $\mathbf{c}_{\ast}$ and $\mathbf{h}_{\ast}$ are chosen as above.
\par
For the figure-eight knot $\FE$ and the representation $\rho_u$ given in Subsection~\ref{subsec:representation}, the Reidemeister torsion $T_{\FE}(\rho_u)$ is given as \eqref{eq:T_def} up to a sign.
For details see \cite{Porti:MAMCAU1997}.
See also \cite[5.3.1]{Murakami/Yokota:2018}.
\appendix
\section{$P_1$}\label{sec:P1}
In this appendix, we give a proof of Lemma~\ref{lem:P1}.
We often use Mathematica \cite{Mathematica}.
Especially, numerals with $\ldots$ such as $1.2345\ldots$ are calculated by using Mathematica.
\begin{proof}[Proof of Lemma~\ref{lem:P1}]
Since the coordinate of $P_1$ is given by $\frac{(6m+5)u^2+4p^2(2m+1)\pi^2-4pu\varphi(u)}{2p|\xi|^2}-2\Im\sigma_m\i$, and $F(\sigma_0)=G_m(\sigma_m)$, we have from \eqref{eq:Gm}
\begin{equation*}
\begin{split}
  &\xi\bigl(G_m(P_1)-F(\sigma_0)\bigr)
  \\
  =&
  \Li_2\left(-e^{-u-\frac{(6m+5)u}{2p}+2\varphi(u)}\right)
  -
  \Li_2\left(-e^{-u+\frac{(6m+5)u}{2p}-2\varphi(u)}\right)
  \\
  &-
  \Li_2\left(e^{-u-\varphi(u)}\right)
  +
  \Li_2\left(e^{-u+\varphi(u)}\right)
  +\frac{6m+5}{2p}u^2-3u\varphi(u)-u\pi\i.
\end{split}
\end{equation*}
So we have
\begin{equation*}
\begin{split}
  &
  \frac{|\xi|^2}{u}\bigl(G_m(P_1)-G_m(\sigma_m)\bigr)
  \\
  =&
  \Li_2\left(-e^{-u-\frac{(6m+5)u}{2p}+2\varphi(u)}\right)
  -
  \Li_2\left(-e^{-u+\frac{(6m+5)u}{2p}-2\varphi(u)}\right)
  -
  \Li_2\left(e^{-u-\varphi(u)}\right)
  \\
  &+
  \Li_2\left(e^{-u+\varphi(u)}\right)
  +\frac{6m+5}{2p}u^2-3u\varphi(u)-2p\pi^2
  \\
  &\text{(since $m\le p-1$)}
  \\
  \le&
  \Li_2\left(-e^{-4u+\frac{u}{2p}+2\varphi(u)}\right)
  -
  \Li_2\left(-e^{2u-\frac{u}{2p}-2\varphi(u)}\right)
  -
  \Li_2\left(e^{-u-\varphi(u)}\right)
  \\
  &+
  \Li_2\left(e^{-u+\varphi(u)}\right)
  +\frac{6p-1}{2p}u^2-3u\varphi(u)
  -2p\pi^2.
\end{split}
\end{equation*}
Denote by $b(r,u)$ the last expression above, and show that $b(r,u)<0$.
Here we regard it as a function of two real variables $p$ and $u$ with $p\ge1$ and $u>\kappa$.
\par
We will show the following:
\begin{enumerate}
\item
If $\kappa<u\le p$, then $\frac{\partial}{\partial\,p}b(p,u)<0$,
\item
If $1\le p<u$, then $\frac{\partial}{\partial\,u}b(p,u)<0$,
\item
For $u>\kappa$, $\tilde{b}(u):=b(u,u)<0$.
\end{enumerate}
Then we see that $b(p,u)<b(u,u)$ from (i), that $b(p,u)<b(p,p)$ from (ii), and that $b(p,u)<0$ from (iii).
\par
So in the following we will prove (i), (ii), and (iii).
\begin{enumerate}
\item
We have
\begin{equation*}
  \frac{\partial}{\partial\,p}b(p,u)
  =
  \frac{u}{2p^2}
  \log\left(2\cosh{u}+2\cosh\left(\left(3u-\frac{u}{2p}-2\varphi(u)\right)\right)\right)
  -2\pi^2.
\end{equation*}
Then since $p\ge u>\kappa$, we have
\begin{equation*}
\begin{split}
  \frac{\partial}{\partial\,p}b(p,u)
  <&
  \frac{1}{2p}
  \log\bigl(2\cosh{p}+2\cosh(3p)\bigr)
  -2\pi^2
  \\
  <&
  \frac{1}{2p}
  \log\bigl(4\cosh(3p)\bigr)
  -2\pi^2
  \\
  =&
  \frac{1}{2p}
  \log\left(2e^{3p}\bigl(1+e^{-6p}\bigr)\right)
  -2\pi^2
  \\
  =&
  \frac{3}{2}
  +
  \frac{1}{2p}
  \log\left(2\bigl(1+e^{-6p}\bigr)\right)
  -2\pi^2
  \\
  <&
  \frac{3}{2}
  +
  \frac{1}{2\kappa}
  \log4
  -2\pi^2
  =
  -17.519\ldots
  <0,
\end{split}
\end{equation*}
proving (i).
\item
The partial derivative of $b(p,u)$ with respect to $u$ becomes
\begin{equation*}
\begin{split}
  &\frac{\partial}{\partial\,u}b(p,u)
  \\
  =&
  6u-\frac{u}{p}-3\varphi(u)-3u\varphi'(u)
  \\
  &+
  \bigl(1-\varphi'(u)\bigr)\log(1-e^{-u+\varphi(u)})
  -
  \bigl(1+\varphi'(u)\bigr)\log(1-e^{-u-\varphi(u)})
  \\
  &+
  \bigl(4-\frac{1}{2p}-2\varphi'(u)\bigr)\log(1+e^{-4u+\frac{u}{2p}+2\varphi(u)})
  \\
  &+
  \bigl(2-\frac{1}{2p}-2\varphi'(u)\bigr)\log(1+e^{2u-\frac{u}{2p}-2\varphi(u)})
  \\
  =&
  6u-\frac{u}{p}-3\varphi(u)-3u\varphi'(u)
  +
  \log\left(\frac{1-e^{-u+\varphi(u)}}{1-e^{-u-\varphi(u)}}\right)
  \\
  &-
  \varphi'(u)\left(\log\bigl(2\cosh{u}-2\cosh\varphi(u)\bigr)-u\right)
  \\
  &+
  \bigl(3-\frac{1}{2p}-2\varphi'(u)\bigr)
  \left(\log\left(2\cosh{u}+2\cosh\bigl(3u-\frac{u}{2p}-2\varphi(u)\bigr)\right)-u\right)
  \\
  &+
  \log\left(\frac{1+e^{-4u+\frac{u}{2p}+2\varphi(u)}}{1+e^{2u-\frac{u}{2p}-2\varphi(u)}}\right)
  \\
  &\text{(since $\cosh{u}-\cosh\varphi(u)=1/2$)}
  \\
  =&
  3u-\frac{u}{2p}-3\varphi(u)
  +
  \log\left(\frac{1-e^{-u+\varphi(u)}}{1-e^{-u-\varphi(u)}}\right)
  +
  \log\left(\frac{1+e^{-4u+\frac{u}{2p}+2\varphi(u)}}{1+e^{2u-\frac{u}{2p}-2\varphi(u)}}\right)
  \\
  &+
  \bigl(3-\frac{1}{2p}-2\varphi'(u)\bigr)
  \log\left(2\cosh{u}+2\cosh\bigl(3u-\frac{u}{2p}-2\varphi(u)\bigr)\right)
  \\
  &\text{(since $u/p>1$ and $p\ge1$)}
  \\
  <&
  3u-\frac{1}{2}-3\varphi(u)
  +
  \log\left(\frac{1-e^{-u+\varphi(u)}}{1-e^{-u-\varphi(u)}}\right)
  +
  \log\left(\frac{1+e^{-\frac{7u}{2}+2\varphi(u)}}{1+e^{\frac{3u}{2}-2\varphi(u)}}\right)
  \\
  &+
  \bigl(3-2\varphi'(u)\bigr)
  \log\left(2\cosh{u}+2\cosh\left(3u-\frac{u}{2p}-2\varphi(u)\right)\right)
  \\
  &\text{(from Lemma~\ref{lem:log_log_u} below)}
  \\
  <&
  2u-3\varphi(u)-\frac{1}{2}
  \\
  &+
  \bigl(3-2\varphi'(u)\bigr)
  \log\left(2\cosh{u}+2\cosh\left(3u-\frac{u}{2p}-2\varphi(u)\right)\right).
\end{split}
\end{equation*}
The last expression above is monotonically increasing (decreasing, respectively) with respect to $p$ if $\varphi'(u)<3/2$ ($\varphi'(u)>3/2$, respectively) because $3u-\frac{u}{2p}-2\varphi(u)=2\bigl(u-\varphi(u)\bigr)+\left(1-\frac{1}{2p}\right)u>0$.
Note that $\varphi'(u)$ is monotonically decreasing with $\lim_{u\to\kappa}\varphi'(u)=\infty$, and $\lim_{u\to\infty}\varphi'(u)=1$ from Lemma~\ref{lem:phi} (v).
So there exists unique $u_0$ with $\kappa<u_0$ such that $\varphi'(u)>3/2$ ($\varphi'(u)<3/2$, respectively) if $u<u_0$ ($u>u_0$, respectively).
In fact, since $u_0$ satisfies the equation $\sinh{u}=\frac{3}{2}\sinh\varphi(u)$, we have $u_0=\arccosh\left(\frac{9+2\sqrt{34}}{10}\right)=1.35436\ldots$.
\par
Now, if $u=u_0$, we have
\begin{equation}\label{eq:u0}
\begin{split}
  &\frac{\partial}{\partial\,u}b(p,u)
  <
  2u_0-3\varphi(u_0)-\frac{1}{2}
  =
  -0.849534\ldots<0.
\end{split}
\end{equation}
\par
If $\kappa<u<u_0$, that is, if $3-2\varphi'(u)<0$, we have
\begin{equation*}
\begin{split}
  &\frac{\partial}{\partial\,u}b(p,u)
  <
  \beta_{\infty}(u)
  \\
  :=&
  2u-3\varphi(u)-\frac{1}{2}
  +
  \bigl(3-2\varphi'(u)\bigr)
  \log\left(2\cosh{u}+2\cosh\left(3u-2\varphi(u)\right)\right).
\end{split}
\end{equation*}
\par
If $u>u_0$, that is, if $3-2\varphi'(u)>0$, we have
\begin{equation*}
\begin{split}
  &\frac{\partial}{\partial\,u}b(p,u)
  <
  \beta_1(u)
  \\
  :=&
  2u-3\varphi(u)-\frac{1}{2}
  +
  \bigl(3-2\varphi'(u)\bigr)
  \log\left(2\cosh{u}+2\cosh\left(\frac{5u}{2}-2\varphi(u)\right)\right).
\end{split}
\end{equation*}
\par
We will show (a) $\beta_{\infty}(u)<0$ for $\kappa<u<u_0$, and (b) $\beta_{1}(u)<0$ for $u>u_0$.
\begin{enumerate}
\item
The case where $\kappa<u<u_0$.
\par
Since the function $2u-3\varphi(u)-1/2=2\bigl(u-\varphi(u)\bigr)-\varphi(u)-1/2$ is decreasing from Lemma~\ref{lem:phi} (iii), it is between $2\kappa-1/2>0$ and $2u_0-3\varphi(u_0)-1/2<0$ from \eqref{eq:u0}.
So there uniquely exists $u_1$ ($\kappa<u_1<u_0$) such that $2u-3\varphi(u)-1/2\le0$ ($2u-3\varphi(u)-1/2>0$, respectively) if $u\ge u_1$ ($u<u_1$, respectively).
Mathematica tells us that $u_1=1.09613\ldots$.
\par
If $u\ge u_1$, then we have
\begin{equation*}
  \beta_{\infty}(u)
  <
  2u-3\varphi(u)-\frac{1}{2}
  \le0
\end{equation*}
since $3-2\varphi'(u)<0$.
If $u<u_1$, then $3-2\varphi'(u)<3-2\varphi'(u_1)=-1.4718\ldots<-1$ since $\varphi'(u)$ is decreasing from Lemma~\ref{lem:phi} (v).
So we have
\begin{equation}\label{eq:u<u1}
  \beta_{\infty}(u)
  <
  2u-3\varphi(u)-\frac{1}{2}
  -
  \log\left(2\cosh{u}+2\cosh\left(3u-2\varphi(u)\right)\right).
\end{equation}
Now, consider the function
\begin{multline*}
  \Bigl(2\cosh{u}+2\cosh\bigl(3u-2\varphi(u)\bigr)\Bigr)\times e^{-2u+3\varphi(u)+1/2}
  \\
  =
  e^{1/2}
  \left(
    e^{-u+3\varphi(u)}
    +
    e^{-3u+3\varphi(u)}
    +
    e^{u+\varphi(u)}
    +
    e^{-5u+5\varphi(u)}
  \right).
\end{multline*}
Since this is monotonically increasing from Lemma~\ref{lem:phi} (ii) and (iii), it is greater than $\Bigl(2\cosh{\kappa}+2\cosh(3\kappa)\Bigr)\times e^{-2\kappa+1/2}=5.05145\ldots>1$.
Taking $\log$ of the function above, we conclude that the right hand side of \eqref{eq:u<u1} is negative.
\item
The case where $u>u_0$.
\par
Since $\varphi'(u)$ is monotonically decreasing as above and $\lim_{u\to\infty}\varphi'(u)=1$ from Lemma~\ref{lem:phi} (v), we see that $0<3-2\varphi'(u)<1$.
So we have
\begin{equation*}
\begin{split}
  \beta_1(u)
  <
  2u-3\varphi(u)-\frac{1}{2}
  +
  \log\left(2\cosh{u}+2\cosh\left(\frac{5u}{2}-2\varphi(u)\right)\right).
\end{split}
\end{equation*}
In a similar to above, we can show that the function
\begin{multline*}
  A(u)
  :=
  \Bigl(2\cosh{u}+2\cosh\bigl(5u/2-2\varphi(u)\bigr)\Bigr)\times e^{2u-3\varphi(u)-1/2}
  \\
  =
  e^{-1/2}
  \left(
    e^{3u-3\varphi(u)}
    +
    e^{u-3\varphi(u)}
    +
    e^{9u/2-5\varphi(u)}
    +
    e^{-u/2-\varphi(u)}
  \right)
\end{multline*}
is monotonically deceasing.
From Lemma~\ref{lem:phi} (iii), the limit $\lim_{u\to\infty}A(u)=e^{-1/2}$.
So we conclude that $0.606531\ldots=e^{-1/2}<A(u)<\Bigl(2\cosh(u_0)+2\cosh\bigl(5u_0/2-2\varphi(u_0)\bigr)\Bigr)e^{2u_0-3\varphi(u_0)-1/2}=3.52289\ldots$, which implies that there uniquely exists $u_2$ such that $A(u)\ge1$ ($A(u)<1$, respectively) if $u_0<u\le u_2$ ($u_2<u$, respectively).
\par
Therefore, if $u>u_2$ we have $\log{A(u)}<0$, which means that $\beta_1(u)<0$ in this case.
\par
It remains to show that $\beta_1(u)<0$ when $u_0<u\le u_2=2.6639\ldots$.
\par
Let us consider the function
\begin{equation*}
  B(u)
  :=
  2u-3\varphi(u)-\frac{1}{2}+\bigl(3-2\varphi'(u)\bigr)\left(u+\frac{4}{5}\right)
\end{equation*}
for $u_0<u\le u_2$.
Since the function
\begin{multline*}
  e^{-u-4/5}
  \left(2\cosh{u}+2\cosh\bigl(5u/2-2\varphi(u)\bigr)\right)
  \\
  =
  e^{-4/5}
  \left(
    1+e^{-2u}+e^{3u/2-2\varphi(u)}+e^{-7u/2+2\varphi(u)}
  \right)
\end{multline*}
is decreasing because $\frac{d}{d\,u}\bigl(-7u/2+2\varphi(u)\bigr)=2\varphi'(u)-7/2<-1/2$, it is less than its value at $u=u_0$, which equals $0.955481\ldots<1$.
So we conclude that $\log\left(2\cosh{u}+2\cosh\bigl(5u/2-2\varphi(u)\bigr)\right)<u+4/5$, which implies $\beta_1(u)<B(u)$ when $u>u_0$.
\par
We will show that $B(u)\le0$ for $u_0<u\le u_2$.
\par
We calculate
\begin{equation}\label{eq:dB}
\begin{split}
  &\frac{d}{d\,u}B(u)
  \\
  =&
  5-5\frac{\sinh{u}}{\sinh\varphi(u)}
  +
  2\left(u+\frac{4}{5}\right)
  \frac{\cosh\varphi(u)\sinh^2(u)-\cosh{u}\sinh^2\varphi(u)}{\sinh^3\varphi(u)}
  \\
  >&
  5-5\frac{\sinh{u}}{\sinh\varphi(u)}
  +
  2\left(u_0+\frac{4}{5}\right)
  \frac{\cosh\varphi(u)\sinh^2(u)-\cosh{u}\sinh^2\varphi(u)}{\sinh^3\varphi(u)},
\end{split}
\end{equation}
where the inequality holds since we have
\begin{equation*}
\begin{split}
  &\cosh\varphi(u)\sinh^2(u)-\cosh{u}\sinh^2\varphi(u)
  \\
  =&
  \cosh\varphi(u)\bigl(\cosh^2(u)-1\bigr)-\cosh{u}\bigl(\cosh^2\varphi(u)-1\bigr)
  \\
  =&
  \bigl(\cosh{u}-\cosh\varphi(u)\bigr)
  \bigl(1+\cosh{u}\cosh\varphi(u)\bigr)>0.
\end{split}
\end{equation*}
Putting $s:=\sinh{u}$, $t:=\sinh\varphi(u)$ and $\hat{u}:=u_0+4/5$, the last expression of \eqref{eq:dB} becomes $C(s,t)/t$ with
\begin{equation*}
  C(s,t)
  :=
  5t-5s+2\hat{u}\left(\left(\frac{s}{t}\right)^2\sqrt{t^2+1}-\sqrt{s^2+1}\right).
\end{equation*}
The partial derivative with respect to $t$ becomes
\begin{equation*}
  \frac{d}{d\,t}C(s,t)
  =
  5-2\hat{u}\frac{s^2(t^2+2)}{t^3\sqrt{t^2+1}}.
\end{equation*}
On the other hand, we have
\begin{equation*}
  \bigl(5t^3\sqrt{t^2+1}\bigr)^2-\bigl(2\hat{u}t^2(t^2+2)\bigr)^2
  =
  t^4\left((25-4\hat{u}^2)t^4+(25-16\hat{u}^2)t^2-16\hat{u}^2\right)
  <0
\end{equation*}
for $u_0<t<u_2$ from a standard argument.
So we have $5t^3\sqrt{t^2+1}<2\hat{u}t^2(t^2+2)$.
Since $t<s$, we have
\begin{equation*}
  5t^3\sqrt{t^2+1}
  <
  2\hat{u}t^2(t^2+2)
  <
  2\hat{u}s^2(t^2+2),
\end{equation*}
which implies that $d\,C(s,t)/d\,t<0$.
\par
Therefore, we have $C(s,t)>C(s,s)=0$ and so from \eqref{eq:dB} we conclude that $d\,B(u)/d\,u>0$.
\par
Now, we finally have $B(u)<B(u_2)=-0.00961956\ldots<0$ if $u\le u_2$, as desired.
\end{enumerate}
Thus in any case we have $\frac{\partial}{\partial\,u}b(p,u)<0$ when $u>p$, proving (ii).
\item
We have
\begin{equation*}
\begin{split}
  \tilde{b}(u)
  =&
  \Li_2\left(-e^{-4u+1/2+2\varphi(u)}\right)
  -
  \Li_2\left(-e^{2u-1/2-2\varphi(u)}\right)
  -
  \Li_2\left(e^{-u-\varphi(u)}\right)
  \\
  &+
  \Li_2\left(e^{-u+\varphi(u)}\right)
  -2u\pi^2+3u^2-\frac{u}{2}-3u\varphi(u)
\end{split}
\end{equation*}
and its derivative becomes
\begin{equation*}
\begin{split}
  &\frac{d}{d\,u}\tilde{b}(u)
  \\
  =&
  \bigl(4-2\varphi'(u)\bigr)\log\left(1+e^{-4u+1/2+2\varphi(u)}\right)
  +
  \bigl(2-2\varphi'(u)\bigr)\log\left(1+e^{2u-1/2-2\varphi(u)}\right)
  \\
  &-
  \bigl(1+\varphi'(u)\bigr)\log\left(1-e^{-u-\varphi(u)}\right)
  +
  \bigl(1-\varphi'(u)\bigr)\log\left(1-e^{-u+\varphi(u)}\right)
  \\
  &-3\varphi(u)-3u\varphi'(u)+6u-2\pi^2-\frac{1}{2}
  \\
  =&
  \left(3-2\varphi'(u)\right)
  \bigg[
    \log\Bigl(2\cosh{u}+2\cosh\bigl(3u-1/2-2\varphi(u)\bigr)\Bigr)-u
  \bigg]
  \\
  &+
  \log
  \left(
    \frac{1+e^{-4u+1/2+2\varphi(u)}}{1+e^{2u-1/2-2\varphi(u)}}
  \right)
  +
  \log
  \left(\frac{1-e^{-u+\varphi(u)}}{1-e^{-u-\varphi(u)}}\right)
  \\
  &-
  \varphi'(u)
  \left(
    \log\bigl(2\cosh{u}-2\cosh\varphi(u)\bigr)-u
  \right)
  -3\varphi(u)-3u\varphi'(u)+6u-2\pi^2-\frac{1}{2}
  \\
  &\text{(from Lemma~\ref{lem:log_log_u2} below)}
  \\
  <&
  \bigl(3-2\varphi'(u)\bigr)
  \log\Bigl(2\cosh{u}+2\cosh\bigl(3u-1/2-2\varphi(u)\bigr)\Bigr)
  \\
  &-
  \varphi'(u)\log\bigl(2\cosh{u}-2\cosh\varphi(u)\bigr)
  -3\varphi(u)+2u-2\pi^2-\frac{1}{2}
  \\
  &\text{(since $\cosh{u}-\cosh\varphi(u)=1/2$)}
  \\
  =&
  \left(3-2\varphi'(u)\right)
  \log\Bigl(2\cosh{u}+2\cosh\bigl(3u-1/2-2\varphi(u)\bigr)\Bigr)
  -3\varphi(u)+2u-2\pi^2-\frac{1}{2}.
\end{split}
\end{equation*}
\par
If $u=u_0$, then we have
\begin{equation*}
  \frac{d}{d\,u}\tilde{b}(u_0)
  =
  -3\varphi(u_0)+2u_0-2\pi^2-\frac{1}{2}
  =
  -20.5887\ldots<0.
\end{equation*}
\par
If $\kappa<u<u_0$, then $3-2\varphi'(u)<0$.
We have
\begin{equation*}
  \frac{d}{d\,u}\tilde{b}(u)
  <
  -3\varphi(u)+2u-2\pi^2-\frac{1}{2}
  =
  2(u-\varphi(u))-\varphi(u)-2\pi^2-\frac{1}{2},
\end{equation*}
which is monotonically decreasing from Lemma~\ref{lem:phi} (iii).
So we conclude that $\frac{d}{d\,u}\tilde{b}(u)<-3\varphi(\kappa)+2\kappa-2\pi^2-\frac{1}{2}=-18.3144\ldots<0$.
\par
If $u>u_0$, then $3-2\varphi'(u)>0$.
Since $\varphi'(u)>1$ from Lemma~\ref{lem:phi} (v), we have
\begin{equation}\label{eq:u>u0}
  \frac{d}{d\,u}\tilde{b}(u)
  <
  \log\Bigl(2\cosh{u}+2\cosh\bigl(3u-1/2-2\varphi(u)\bigr)\Bigr)
  -3\varphi(u)+2u-2\pi^2-\frac{1}{2}.
\end{equation}
We will consider the function
\begin{multline*}
  \Bigl(2\cosh{u}+2\cosh\bigl(3u-1/2-2\varphi(u)\bigr)\Bigr)
  \times
  e^{2u-3\varphi(u)}
  \\
  =
  e^{3u-3\varphi(u)}+e^{u-3\varphi(u)}+e^{5u-1/2-5\varphi(u)}+e^{-u+1/2-\varphi(u)}.
\end{multline*}
Since it is monotonically decreasing, we have
\begin{equation*}
\begin{split}
  &\Bigl(2\cosh{u}+2\cosh\bigl(3u-1/2-2\varphi(u)\bigr)\Bigr)
  \times
  e^{2u-3\varphi(u)}
  \\
  <&
  \Bigl(2\cosh{u_0}+2\cosh\bigl(3u_0-1/2-2\varphi(u_0)\bigr)\Bigr)
  \times
  e^{2u_0-3\varphi(u_0)}
  \\
  =&
  6.30412\ldots
  <
  e^{2\pi^2+1/2}
\end{split}
\end{equation*}
for $u>u_0$.
It follows that the right hand side of \eqref{eq:u>u0} is negative.
\par
So we conclude that $\tilde{b}(u)$ is decreasing and that it is less than $\tilde{b}(\kappa)=-14.3021<0$, proving (iii).
\end{enumerate}
\end{proof}
The following lemmas are used in the proof of Lemma~\ref{lem:P1}.
\begin{lem}\label{lem:log_log_u}
We have the following inequality for $u>\kappa$:
\begin{equation*}
  \log
  \left(
    \frac{1+e^{-7u/2+2\varphi(u)}}{1+e^{3u/2-2\varphi(u)}}
  \right)
  +
  \log
  \left(\frac{1-e^{-u+\varphi(u)}}{1-e^{-u-\varphi(u)}}\right)
  +u
  <0.
\end{equation*}
\end{lem}
\begin{proof}
In the proof, most calculations are done by Mathematica \cite{Mathematica}.
\par
We will show that
\begin{equation*}
  \left(1+e^{-7u/2+2\varphi(u)}\right)
  \left(1-e^{-u+\varphi(u)}\right)e^{u}
  -
  \left(1+e^{3u/2-2\varphi(u)}\right)
  \left(1-e^{-u-\varphi(u)}\right)
  <0.
\end{equation*}
The left hand side becomes
\begin{multline}\label{eq:log_log_u}
  e^u-1
  -e^{\varphi(u)}+e^{-u}e^{-\varphi(u)}
  \\
  +e^{-5u/2}e^{2\varphi(u)}-e^{3u/2}e^{-2\varphi(u)}
  -e^{-7u/2}e^{3\varphi(u)}+e^{u/2}e^{-3\varphi(u)}.
\end{multline}
Since $e^{\varphi(u)}$ satisfies the quadratic equation $x^2-\bigl(2\cosh{u}-1\big)x+1=0$ from \eqref{eq:phi}, we have
\begin{equation}\label{eq:phi_D}
\begin{split}
  e^{\pm\varphi(u)}
  &=
  \frac{a}{2}\pm\frac{1}{2}\sqrt{D},
  \\
  e^{\pm2\varphi(u)}
  &=
  \frac{a^2}{2}-1\pm\frac{a}{2}\sqrt{D},
  \\
  e^{\pm3\varphi(u)}
  &=
  \frac{a^3}{2}-\frac{3}{2}a
  \pm
  \frac{1}{2}(a^2-1)\sqrt{D}.
\end{split}
\end{equation}
where we put $a:=2\cosh{u}-1$ and $D:=a^2-4$.
\par
Therefore \eqref{eq:log_log_u} turns out to be
\begin{equation*}
  \tau_1(u)+\tau_2(u)\sqrt{D}
\end{equation*}
with
\begin{align*}
  \tau_1(u)
  :=&
  e^{u}-1
  +\left(-1+e^{-u}\right)\times\frac{a}{2}
  +\left(e^{-5u/2}-e^{3u/2}\right)\times\left(\frac{a^2}{2}-1\right)
  \\
  &+\left(-e^{-7u/2}+e^{u/2}\right)\times\left(\frac{a^3}{2}-\frac{3}{2}a\right),
  \\
  \tau_2(u)
  :=&
  \left(-1-e^{-u}\right)\times\frac{1}{2}
  +\left(e^{-5u/2}+e^{3u/2}\right)\times\frac{a}{2}
  \\
  &+\left(-e^{-7u/2}-e^{u/2}\right)\times\frac{1}{2}(a^2-1).
\end{align*}
Now, we have
\begin{equation*}
  \bigl(\tau_1(u)+\tau_2(u)\sqrt{D}\bigr)e^{3\varphi(u)}
  =
  \tilde{\tau}_1(u)+\tilde{\tau}_2(u)\sqrt{D}
\end{equation*}
from \eqref{eq:phi_D}, where we put
\begin{align*}
  \tilde{\tau}_1(u)
  &:=
  \left(\frac{a^3}{2}-\frac{3}{2}a\right)\tau_1(u)
  +
  \frac{1}{2}(a^2-1)D\tau_2(u),
  \\
  \tilde{\tau}_2(u)
  &:=
  \left(\frac{a^3}{2}-\frac{3}{2}a\right)\tau_2(u)
  +
  \frac{1}{2}(a^2-1)\tau_1(u).
\end{align*}
Since $\sqrt{D}>0$ when $u>\kappa$, it is sufficient to prove both $\tilde{\tau}_1(u)$ and $\tilde{\tau}_2(u)$ are negative for $u>\kappa$.
\par
Putting $x:=e^u$, we have
\begin{equation*}
\begin{split}
  \tilde{\tau}_1(\log{x})
  =&
  \frac{1}{2x^{19/2}}(x-1)
  \left(
    -x^{11}
    +x^{10}
    +2x^{19/2}
    -3x^9
    -2x^{17/2}
    +8x^8
    +x^{15/2}
  \right.
  \\
  &\left.
    -11x^7
    -3x^{13/2}
    +16x^6
    +x^{11/2}
    -16x^5
    +17x^4
    -12x^3
    +9x^2
    -5x
    +1
  \right).
\end{split}
\end{equation*}
The real roots of the right hand are $0.390642\ldots$, $0.509556\ldots$, $1$, and $1.62125\ldots$.
Since $e^{\kappa}=2.61808\ldots$, the function $\tilde{\tau}_1(u)$ never vanishes for $u>\kappa$.
Clearly, the limit $\lim_{u\to\infty}\tilde{\tau}_1(u)$ is $-\infty$, and we conclude that $\tilde{\tau}_1(u)<0$ for $u>\kappa$.
\par
Similarly, we have
\begin{equation*}
\begin{split}
  \tilde{\tau}_2(\log{x})
  =&
  \frac{1}{2x^{17/2}}
  \left(
    -x^{10}
    +x^9
    +2x^{17/2}
    -4x^8
    -2x^{15/2}
    +8x^7
    +3x^{13/2}
    -13x^6
  \right.
  \\
  &\left.
    -x^{11/2}
    +16x^5
    -17x^4
    +14x^3
    -10x^2
    +5x
    -1
  \right),
\end{split}
\end{equation*}
whose real roots are $0.423932\ldots$, $0.848316\ldots$, and $1$.
Therefore we see that $\tilde{\tau}_2(u)<0$ for $u>\kappa$ by the same reason as above.
\end{proof}
\begin{rem}
Since we have
\begin{equation*}
  \tau_2(\log{x})
  =
  \frac{1}{2 x^{11/2}}
  \left(x^7-x^6-x^{11/2}+2x^5-x^{9/2}-x^4+x^3-x^2+2x-1\right),
\end{equation*}
we see that $\tau_2(u)$ becomes positive when $u$ is large.
\end{rem}
Similarly we can prove the following lemma.
\begin{lem}\label{lem:log_log_u2}
The following inequality holds:
\begin{equation*}
  \log
  \left(
    \frac{1+e^{-4u+1/2+2\varphi(u)}}{1+e^{2u-1/2-2\varphi(u)}}
  \right)
  +
  \log
  \left(\frac{1-e^{-u+\varphi(u)}}{1-e^{-u-\varphi(u)}}\right)
  +u
  <0.
\end{equation*}
\end{lem}
\begin{proof}
In a similar way to the proof of the previous lemma, we will show the following inequality.

\begin{multline}\label{eq:log_log_u2}
  e^u-1
  -e^{\varphi(u)}+e^{-u}e^{-\varphi(u)}
  \\
  +e^{-3u+1/2}e^{2\varphi(u)}-e^{2u-1/2}e^{-2\varphi(u)}
  -e^{-4u+1/2}e^{3\varphi(u)}+e^{u-1/2}e^{-3\varphi(u)}
  <0.
\end{multline}
Putting
\begin{align*}
  \tau_3(u)
  :=&
  e^{u}-1
  +\left(-1+e^{-u}\right)\times\frac{a}{2}
  +\left(e^{-3u+1/2}-e^{2u-1/2}\right)\times\left(\frac{a^2}{2}-1\right)
  \\
  &+\left(-e^{-4u+1/2}+e^{u-1/2}\right)\times\left(\frac{a^3}{2}-\frac{3}{2}a\right),
  \\
  \tau_4(u)
  :=&
  \left(-1-e^{-u}\right)\times\frac{1}{2}
  +\left(e^{-3u+1/2}+e^{2u-1/2}\right)\times\frac{a}{2}
  \\
  &+\left(-e^{-4u+1/2}-e^{u-1/2}\right)\times\frac{1}{2}(a^2-1),
  \\
  \tilde{\tau}_3(u)
  :=&
  \left(\frac{a^3}{2}-\frac{3}{2}a\right)\tau_3(u)
  +
  \frac{1}{2}(a^2-1)D\tau_4(u),
  \\
  \tilde{\tau}_4(u)
  :=&
  \left(\frac{a^3}{2}-\frac{3}{2}a\right)\tau_4(u)
  +
  \frac{1}{2}(a^2-1)\tau_3(u).
\end{align*}
the left hand side of \eqref{eq:log_log_u2} becomes
\begin{equation*}
  e^{-3\varphi(u)}\left(\tilde{\tau}_3(u)+\tilde{\tau}_4(u)\sqrt{D}\right).
\end{equation*}
\par
We will show that $\tilde{\tau}_3(u)<0$ and $\tilde{\tau}_4(u)<0$ when $u>\kappa$.
\par
We have
\begin{align*}
  \tilde{\tau}_3(\log{x})
  =&
  \frac{1}{2e^{1/2}x^{10}}
  \left(
    -x^{13}+x^{12}+\left(2 \sqrt{e}+e+1\right) x^{11}+\left(-4 \sqrt{e}-5 e\right) x^{10}
  \right.
  \\&
    +\left(3\sqrt{e}+11e\right)x^9+\left(-4\sqrt{e}-19e\right)x^8+\left(4\sqrt{e}+27e\right)x^7
  \\&
  \left.
    +\left(-\sqrt{e}-32 e\right) x^6+33 e x^5-29 e x^4+21 e x^3-14 e x^2+6 e x-e
  \right),
  \\
  \tilde{\tau}_4(\log{x})
  =&
  \frac{1}{2e^{1/2}x^2}
  \left(
    -x^{11}+\left(2\sqrt{e}+e\right)x^9+\left(-2\sqrt{e}-4e\right)x^8+\left(3\sqrt{e}+8e\right)x^7
  \right.
  \\&
  \left.
    +\left(-\sqrt{e}-13 e\right) x^6+16 e x^5-17 e x^4+14 e x^3-10 e x^2+5 e x-e
  \right).
\end{align*}
The positive real roots of the first equation are $0.391427\ldots$ and $0.498682\ldots$, and those of the second one are $0.424217\ldots$, $0.69792\ldots$, $1.13263\ldots$, and $1.45782\ldots$.
By the same reason as above, we conclude that $\tilde{\tau}_3(u)$ and $\tilde{\tau}_4(u)$ are negative, proving the lemma.
\end{proof}
\begin{rem}
Mathematica calculates
\begin{multline*}
  \tau_4(\log{x})
  \\
  =
  \frac{1}{2e^{1/2}x^6}
  \left(
    x^8-x^7+\left(2-\sqrt{e}\right) x^6+\left(-\sqrt{e}-1\right) x^5+e x^3-e x^2+2 e x-e
  \right).
\end{multline*}
So, $\tau_4(u)$ becomes positive if $u$ is large.
\end{rem}

\bibliography{mrabbrev,hitoshi}
\bibliographystyle{amsplain}
\end{document}